\documentclass[a4paper, intlimits, reqno]{amsart}

\usepackage[english]{babel}
\usepackage[T1]{fontenc}
\usepackage[utf8]{inputenc}

\usepackage{amsmath}
\usepackage{amssymb}
\usepackage{MnSymbol}
\usepackage{amsthm}
\allowdisplaybreaks
\usepackage{amsfonts}
\usepackage{mathrsfs} 
\usepackage{mathtools}
\usepackage{enumitem}
\usepackage{url}
\usepackage[arrow, matrix, curve]{xy}
\usepackage{undertilde}
\usepackage{epsfig}
\usepackage{epstopdf}
\usepackage{caption}

\usepackage[numbers,sort&compress]{natbib}
\usepackage{prettyref}
\usepackage{doi}
\usepackage{orcidlink}

\newrefformat{def}{Definition \ref{#1}}
\newrefformat{rem}{Remark \ref{#1}}
\newrefformat{sect}{Section \ref{#1}}
\newrefformat{sub}{Section \ref{#1}}
\newrefformat{prop}{Proposition \ref{#1}}
\newrefformat{thm}{Theorem \ref{#1}}
\newrefformat{cor}{Corollary \ref{#1}}
\newrefformat{ex}{Example \ref{#1}}
\newrefformat{prob}{Problem \ref{#1}}
\newrefformat{fig}{Figure \ref{#1}}

\swapnumbers
\newtheoremstyle{dotless}{}{}{\itshape}{}{\bfseries}{}{}{}
\theoremstyle{dotless}
\theoremstyle{plain}
\newtheorem{thm}{Theorem}[section]
\newtheorem{lem}[thm]{Lemma}
\newtheorem{prop}[thm]{Proposition}
\newtheorem{cor}[thm]{Corollary}
\theoremstyle{definition}
\newtheorem{defn}[thm]{Definition}
\newtheorem{rem}[thm]{Remark}
\newtheorem{exa}[thm]{Example}
\newtheorem{prob}[thm]{Problem}

\newcommand{\N} {\mathbb{N}}
\newcommand{\R} {\mathbb{R}}
\newcommand{\C} {\mathbb{C}}
\newcommand{\D} {\mathbb{D}}

\DeclareMathOperator{\id}{id}
\DeclareMathOperator{\re}{Re}
\DeclareMathOperator{\im}{Im}
\providecommand{\differential}{\mathrm{d}}
\renewcommand{\d}{\differential}
\newcommand{\e}{\mathrm{e}}

\newcommand{\vertiii}[1]{{\left\vert\kern-0.25ex\left\vert\kern-0.25ex\left\vert #1 
    \right\vert\kern-0.25ex\right\vert\kern-0.25ex\right\vert}}
  
\makeatletter
\newcommand{\fakephantomsection}{%
  \Hy@GlobalStepCount\Hy@linkcounter%
  \Hy@MakeCurrentHref{\@currenvir.\the\Hy@linkcounter}
  \Hy@raisedlink{\hyper@anchorstart{\@currentHref}\hyper@anchorend}%
}
\makeatother  
  
\begin{document}

\title[Fourier hyperfunctions]{Vector-valued Fourier hyperfunctions and boundary values}
\author[K.~Kruse]{Karsten Kruse\,\orcidlink{0000-0003-1864-4915}}
\address{University of Twente, Department of Applied Mathematics, P.O. Box 217, 7500 AE Enschede, The Netherlands, and 
Hamburg University of Technology, Institute of Mathematics, Am Schwarzenberg-Campus~3, 21073 Hamburg, Germany}

\email{k.kruse@utwente.nl}



\date{\today}
\dedicatory{To the memory of Professor Pawe{\l} Doma\'nski.}
\begin{abstract}
This work is dedicated to the development of the theory of Fourier hyperfunctions in one variable 
with values in a complex non-necessarily metrisable locally convex Hausdorff space $E$. 
Moreover, necessary and sufficient conditions are described such that 
a reasonable theory of $E$-valued Fourier hyperfunctions exists. 
In particular, if $E$ is an ultrabornological PLS-space, such a theory is possible 
if and only if E satisfies the so-called property $(PA)$. 
Furthermore, many examples of such spaces having $(PA)$ resp.\ not having $(PA)$ are provided. 
We also prove that the vector-valued Fourier hyperfunctions can be realized as the sheaf 
generated by equivalence classes of certain compactly supported $E$-valued functionals 
and interpreted as boundary values of slowly increasing holomorphic functions. 
\end{abstract}

\maketitle

\section{Introduction}

The aim of the present work, which is the main result of the author's PhD thesis \cite{ich} 
with some improvements, is the development of the theory of Fourier hyperfunctions in one variable 
with values in a complex non-necessarily metrisable locally convex Hausdorff space $E$ and 
to find necessary and sufficient conditions such that a reasonable theory of $E$-valued Fourier hyperfunctions is possible. 
In particular, we show that, if $E$ is an ultrabornological PLS-space, such a theory exists if and only if 
$E$ satisfies the so-called property $(PA)$. It turns out that the vector-valued Fourier hyperfunctions can be realised 
as the sheaf generated by equivalence classes of certain compactly supported $E$-valued functionals and 
interpreted as boundary values of slowly increasing holomorphic functions. 
 
Scalar-valued Fourier hyperfunctions $\mathcal{R}$, indicated by Sato \cite{sato1958}(1958), 
were introduced by Kawai \cite{Kawai} in 1970. 
He constructed them as a flabby sheaf on $D^{d}$, 
where $D^{d}$ means the radial compactification of $\R^{d}$, $d\in\N$, 
using cohomology theory and H\"ormander's $L^{2}$-estimates \cite{H3}. 
He proved that the global sections are stable under Fourier transformation $\mathscr{F}$, 
i.e.\ $\mathscr{F}\colon \mathcal{R}(D^{d})\to\mathcal{R}(D^{d})$ is an isomorphism. 
This sheaf is a generalisation of the sheaf $\mathcal{B}$ of hyperfunctions on $\R^{d}$, 
which was developed by Sato \cite{Sato1} (and \cite{Sato2}); 
in particular, $\mathcal{R}_{\mid\R^{d}}=\mathcal{B}$ holds. 
Hyperfunctions emerged as an useful tool in the theory of partial differential equations 
(see \cite{Kom2}), in particular, in the solution of the abstract Cauchy problem. 
Komatsu developed the theory of Laplace hyperfunctions, a theory of operator-valued generalised functions 
with a suitable Laplace transform, more precisely, for operators in Banach spaces, 
and the abstract Cauchy problem was solved by a condition on the resolvent of the operator 
which characterised the generators of hyperfunction semigroups (see \cite{Kom3,Kom4,Kom5,Kom6}, 
Honda and Umeta \cite{honda2013}). 
This theory was improved and extended beyond operators in Banach spaces by Doma\'nski and Langenbruch (see \cite{D/L3,D/L4}). 
Since some partial differential equations can be taken as ordinary vector-valued equations (e.g.\ \cite{Ouchi1,Ouchi2}), 
the question arose whether there was a vector-valued counterpart for the theory of (Fourier) hyperfunctions. 
Whereas Schwartz achieved this in the analogous theory of distributions by tensor products \cite{Sch1}, 
one faces a crucial problem in the development of such a theory of vector-valued, in short, 
$E$-valued where $E$ is a locally convex Hausdorff space over $\C$, (Fourier) hyperfunctions, namely, 
the lack of a natural linear Hausdorff topology on the scalar-valued (Fourier) hyperfunctions 
(with the exception of the space of global sections in the case of Fourier hyperfunctions). 
Despite of this difficulty, Ion and Kawai \cite{Ion/Ka}(1975) developed a theory of hyperfunctions with values in Fr\'echet spaces, 
Ito and Nagamachi \cite{ItoNag1975_1,ItoNag1975_2}(1975) a theory of Fourier hyperfunctions 
with values in separable Hilbert spaces (see \cite{Ito1992_2} for general Hilbert spaces), 
which was used by Mugibayashi and Nagamachi (\cite{Mu/Na1,Mu/Na2}) 
for an axiomatic formulation of quantum field theory  
in terms of Fourier hyperfunctions, and Junker \cite{J}(1979) a theory of Fourier hyperfunctions 
with values in Fr\'echet spaces (cf.\ \cite{Ito1984,Ito1985,Ito1988,Ito1989,Ito1990_1,Ito1990_2,Ito1992_1}). 
Since Fourier hyperfunctions with values in non-metrisable spaces $E$ like the space of distributions, are of interest as well, 
there were some efforts to extend the theory of Fourier hyperfunctions to non-Fr\'echet spaces $E$ (see \cite{Ito2002}).
However, to the best of our knowledge the present paper is the only fully correct theory of $E$-valued Fourier hyperfunctions 
including non-Fr\'echet spaces $E$ (see \prettyref{rem:itos_luecke_1}, \prettyref{rem:itos_luecke_2}).

Doma\'nski and Langenbruch \cite{D/L}(2008) not only overcame these obstacles and 
developed a theory of vector-valued hyperfunctions beyond the class of Fr\'echet spaces, 
but also found natural limits of this kind of theory. They characterised in a large natural class of locally convex Hausdorff spaces 
those spaces for which a reasonable theory of $E$-valued hyperfunctions exists at all (see \cite[Theorem 8.9, p.\ 1139]{D/L}). 
To be more precise: they state that a reasonable theory of $E$-valued hyperfunctions should generate a flabby sheaf with the property 
that the set of sections supported by a compact subset $K\subset\R^{d}$ should coincide with $L(\mathscr{A}(K),E)$, 
the space of continuous linear operators from $\mathscr{A}(K)$ to $E$ where 
$\mathscr{A}(K)$ denotes the space of germs of real analytic functions on $K$. 
Transferring this condition to the theory of Fourier hyperfunctions, we are convinced that 
a reasonable theory of $E$-valued Fourier hyperfunctions (in one variable) should produce a flabby sheaf such that 
the set of sections supported by a compact subset $K\subset\overline{\R}$ should coincide with 
``the space of $E$-valued $\mathcal{P}_{\ast}$-functionals'' $L(\mathcal{P}_{\ast}(K),E)$ where 
$D^1=\overline{\R}$ is the radial compactification of $\R$ and $\mathcal{P}_{\ast}(K)$ 
the space of rapidly decreasing holomorphic germs near $K$ (see \prettyref{prop:DFS}). 
If one restricts such a sheaf to $\R$, the restricted sheaf fulfils the condition of Doma\'nski and Langenbruch 
for a reasonable theory of $E$-valued hyperfunctions, since $\mathcal{P}_{\ast}(K)=\mathscr{A}(K)$ for compact $K\subset\R$, 
which is desirable in the spirit of the property $\mathcal{R}_{\mid\R}=\mathcal{B}$ of the scalar-valued case. 
Furthermore, the global sections of such a sheaf are stable under Fourier transformation (see \prettyref{cor:Fourier-Trafo}). 
This implies that for those spaces $E$, for which a reasonable theory of $E$-valued hyperfunctions is impossible, 
a reasonable theory of $E$-valued Fourier hyperfunctions is impossible as well. 
A long list of examples of spaces $E$ for which a reasonable theory of $E$-valued Fourier hyperfunctions is possible 
resp.\ impossible can be found in \prettyref{ex:PLS_PA_DF_DN} resp.\ \prettyref{ex:PLS_non_PA}.

In the approach of Doma\'nski and Langenbruch the existence of an $E$-valued sheaf of hyperfunctions 
is deeply connected with the solvability of the $E$-valued Laplace equation; namely, if the $(d+1)$-dimensional Laplace operator
\[
\Delta_{d+1}\colon \mathcal{C}^{\infty}(\Omega,E)\to\mathcal{C}^{\infty}(\Omega,E)
\]
is surjective for every open set $\Omega\subset\R^{d+1}$ where $\mathcal{C}^{\infty}(\Omega,E)$ 
is the space of smooth $E$-valued functions on $\Omega$, then a reasonable theory of $E$-valued hyperfunctions 
on $\R^{d}$ is possible (see \cite[Theorem 6.9, p.\ 1125]{D/L}). 
For $E$-valued Fourier hyperfunctions in one variable the corresponding counterpart is the following. 
A complex locally convex Hausdorff space $E$ is called \emph{admissible} if the Cauchy--Riemann operator
\[
\overline{\partial}\colon \mathcal{E}^{exp}(\overline{\C}\setminus K,E)\to \mathcal{E}^{exp}(\overline{\C}\setminus K,E)
\]
is surjective for any compact set $K\subset\overline{\R}$ where $\overline{\C}\coloneq\overline{\R}+\mathsf{i}\,\R$ and 
$\mathcal{E}^{exp}(\overline{\C}\setminus K,E)$ is, roughly speaking, 
the space of slowly increasing smooth $E$-valued functions outside $K$ (see \prettyref{def:smooth_weighted_space}). 
$E$ is called \emph{strictly admissible} if $E$ is admissible and, in addition,
\[
\overline{\partial}\colon \mathcal{C}^{\infty}(\Omega,E)\to\mathcal{C}^{\infty}(\Omega,E)
\]
is surjective for any open set $\Omega\subset\C$. 
We prove that $E$ being strictly admissible yields to the existence of a reasonable theory 
of $E$-valued Fourier hyperfunctions in one variable (see \prettyref{thm:sheaf_flabby}).

The outline of the present paper is as follows. In \prettyref{sect:notation} we introduce some notations and 
preliminaries needed to phrase our concepts.
In \prettyref{sect:koethe_duality_fourier} we define the spaces $\mathcal{E}^{exp}(\overline{\C}\setminus K,E)$, 
its subspace $\mathcal{O}^{exp}(\overline{\C}\setminus K,E)$ of holomorphic functions and $\mathcal{P}_{\ast}(K)$. 
Further, we recall some of their properties and a kind of Silva--K\"othe--Grothendieck duality (see \prettyref{thm:duality}), 
give a boundary value representation of $L_{b}(\mathcal{P}_{\ast}(\overline{\R}),E)$ and define the Fourier 
transformation on this space. 
In \prettyref{sect:strictly_admissible} we collect some results on strict admissibility (see \prettyref{thm:examples_strictly_admiss}) 
and give many examples of strictly admissible spaces $E$.
In correspondence with the scalar-valued case, the $E$-valued Fourier hyperfunctions are defined in 
\prettyref{sect:duality_method} from two different points of view for a sequentially complete strictly admissible space $E$. 
On the one hand, as the sheaf generated by equivalence classes of $E$-valued $\mathcal{P}_{\ast}$-functionals, 
and on the other, as the sheaf of boundary values of the elements of $\mathcal{O}^{exp}(U\setminus\overline{\R},E)$. 
This is, to put it roughly, the space of holomorphic $E$-valued slowly increasing functions on $U$ 
outside an open set $\Omega\subset\overline{\R}$ where $U$ is an open set in $\overline{\C}$ 
with $U\cap\overline{\R}=\Omega$ (see \prettyref{def:bv}). 
The construction of these sheaves benefits from our kind of Silva--K\"othe--Grothendieck duality 
and it turns out that both sheaves are flabby and isomorphic (see \prettyref{thm:sheaf_flabby}), 
solving two problems of Ito (see \prettyref{lem:I_isom_op}, \prettyref{rem:itos_luecke_1}, \prettyref{cor:flabby}).
At the end of the fifth section, we show that, if $E$ is an ultrabornological PLS-space, 
a reasonable theory of $E$-valued Fourier hyperfunctions 
in one variable exists if and only if $E$ satisfies the property $(PA)$ (see \prettyref{thm:PA_necessary}). 

\section{Notation and Preliminaries}
\label{sect:notation}

The notation and preliminaries are essentially the same as in \cite[Section 2]{kruse2018_5,kruse2017,kruse2019_1}. 
We denote by $|\cdot|$ the Euclidean norm on $\R^2$ and $\C$, identify $\R^{2}$ and $\C$ as (normed) vector spaces, 
write $\D_{r}(z)\coloneq\{w\in\C\;|\;|w-z|<r\}$ for the open ball with radius $r>0$ around $z\in\C$  
and denote the restriction of a function $f\colon M\to\C$ to $K\subset M\subset\C$ by $f_{\mid K}$.
We define the distance of two subsets $M_{0}, M_{1} \subset\R^{2}$ w.r.t.\ $|\cdot|$ on $\R^{2}$ via
\[
  \d(M_{0},M_{1}) 
\coloneq\begin{cases}
   \inf_{x\in M_{0},\,y\in M_{1}}|x-y| &,\;  M_{0},\,M_{1} \neq \emptyset, \\
   \infty &,\;  M_{0}= \emptyset \;\text{or}\; M_{1}=\emptyset,
  \end{cases}
\]
and write $\d(z,M_{1})\coloneq\d(\{z\},M_{1})$ for $z\in\R^{2}$.   
We denote by $\overline{\R}\coloneq\R\cup\{\pm\infty\}$ the radial compactifaction of $\R$, i.e.\ 
we equip $\overline{\R}$ with the following topology. A set $\Omega\subset\overline{\R}$ is called open 
if $\Omega\cap\R$ is open in $(\R,|\cdot|)$ and, in addition, there exists $a\in\R$ such that $[-\infty,a]\subset\Omega$ 
resp.\ $[a,\infty]\subset\Omega$ if $-\infty\in\Omega$ resp.\ $\infty\in\Omega$. $\overline{\R}$ becomes a compact 
space with this topology. We set $\overline{\C}\coloneq\overline{\R}+\mathsf{i}\,\R$ and equip it with the product topology. 
For a topological space $X$ we denote the complement of a subset $M\subset X$ by $M^{C}\coloneq X\setminus M$, 
the closure of $M$ in $X$ by $\overline{M}$ and the boundary of $M$ by $\partial M$. 

By $E$ we always denote a non-trivial locally convex Hausdorff space over the field 
$\C$ ($\C$-lcHs) equipped with a directed fundamental system of seminorms $(p_{\alpha})_{\alpha\in \mathfrak{A}}$. 
If $E=\C$, then we set $(p_{\alpha})_{\alpha\in \mathfrak{A}}\coloneq\{|\cdot|\}$.
Further, we denote by $L(F,E)$ the space of continuous linear maps from 
a locally convex Hausdorff space $F$ to $E$ and sometimes use the notation $\langle T, f\rangle\coloneq T(f)$, $f\in F$,  
for $T\in L(F,E)$. If $E=\C$, we write $F'\coloneq L(F,\C)$ for the dual space of $F$. 
We denote by $L_{t}(F,E)$ the space $L(F,E)$ equipped with the locally convex topology of uniform convergence 
on the absolutely convex compact subsets of $F$ if $t=\kappa$, and on the bounded subsets of $F$ if $t=b$. 
The $\varepsilon$-product of Schwartz \cite[Chap.\ I, \S1, D\'{e}finition, p.\ 18]{Sch1} 
is defined by 
\[
F\varepsilon E\coloneq L_{e}(F_{\kappa}',E)
\]
where $L(F_{\kappa}',E)$ is equipped with the topology of uniform convergence on equicontinuous subsets of $F'$. 
By $F\widehat{\otimes}_{\pi} E$ we denote the completion of the projective tensor product $F\otimes_{\pi} E$. 
The space $F\widehat{\otimes}_{\pi} E$ is topologically isomorphic to $F\varepsilon E$ if $F$ and $E$ are complete 
and one of them is nuclear.

We recall the following well-known definitions concerning continuous partial differentiability of 
vector-valued functions (cf.\ \cite[p.\ 237]{kruse2018_2}). A function $f\colon\Omega\to E$ on an open set 
$\Omega\subset\R^{2}$ to $E$ is called continuously partially differentiable ($f$ is $\mathcal{C}^{1}$) 
if for the $n$th unit vector $e_{n}\in\R^{2}$ the limit
\[
\partial^{e_{n}}f(x)\coloneq \lim_{\substack{h\to 0\\ h\in\R, h\neq 0}}\frac{f(x+he_{n})-f(x)}{h}
\]
exists in $E$ for every $x\in\Omega$ and $\partial^{e_{n}}f$ 
is continuous on $\Omega$ ($\partial^{e_{n}}f$ is $\mathcal{C}^{0}$) for every $n\in\{1,2\}$. 
For $k\in\N$ a function $f$ is said to be $k$-times continuously partially differentiable 
($f$ is $\mathcal{C}^{k}$) if $f$ is $\mathcal{C}^{1}$ and all its first partial derivatives are $\mathcal{C}^{k-1}$.
A function $f$ is called infinitely continuously partially differentiable ($f$ is $\mathcal{C}^{\infty}$) 
if $f$ is $\mathcal{C}^{k}$ for every $k\in\N$.
The linear space of all functions $f\colon\Omega\to E$ which are $\mathcal{C}^{\infty}$ 
is denoted by $\mathcal{C}^{\infty}(\Omega,E)$ and we write $\mathcal{C}^{\infty}(\Omega)\coloneq\mathcal{C}^{\infty}(\Omega,\C)$. 
Let $f\in\mathcal{C}^{\infty}(\Omega,E)$. For $\beta=(\beta_{n})\in\N_{0}^{2}$ we set 
$\partial^{\beta_{n}}f\coloneq f$ if $\beta_{n}=0$, and
\[
\partial^{\beta_{n}}f
\coloneq\underbrace{\partial^{e_{n}}\cdots\partial^{e_{n}}}_{\beta_{n}\text{-times}}f
\]
if $\beta_{n}\neq 0$ as well as 
\[
\partial^{\beta}f
\coloneq\partial^{\beta_{1}}\partial^{\beta_{2}}f.
\]
Due to the vector-valued version of Schwarz' theorem $\partial^{\beta}f$ is independent of the order of the partial 
derivatives on the right-hand side and we call $|\beta|\coloneq\beta_{1}+\beta_{2}$ the order of differentiation. 

A continuous function $f\colon\Omega\to E$ on an open set 
$\Omega\subset\C$ to $E$ is called holomorphic if the limit
\[
\partial_{\C}^{1}f(z_{0})
\coloneq\lim_{\substack{h\to 0\\ h\in\C, h\neq 0}}\frac{f(z_{0}+h)-f(z_{0})}{h}
\]
exists in $E$ for every $z_{0}\in\Omega$. As before we define derivatives of higher order recursively, i.e.\ 
for $n\in\N_{0}$ we set $\partial_{\C}^{0}f\coloneq f$ and $\partial_{\C}^{n}f\coloneq\partial_{\C}^{1}\partial_{\C}^{n-1}f$, 
$n\geq 1$, if the corresponding limits exist. 
The linear space of all functions $f\colon\Omega\to E$ which are holomorphic 
is denoted by $\mathcal{O}(\Omega,E)$ and we write $\mathcal{O}(\Omega)\coloneq\mathcal{O}(\Omega,\C)$. 
If $E$ is locally complete and $f\in\mathcal{O}(\Omega,E)$, then $\partial_{\C}^{n}f(z_{0})$ exists in $E$ 
for every $z_{0}\in\Omega$ and $n\in\N_{0}$ by \cite[2.2 Theorem and Definition, p.\ 18]{grosse-erdmann1992} 
and \cite[5.2 Theorem, p.\ 35]{grosse-erdmann1992}. $E$ is called locally compete if every closed disk in $E$ 
is a Banach disk (see \cite[10.2.1 Proposition, p.\ 197]{Jarchow}). In particular, every sequentially complete space 
is locally complete.

For the convenience of the reader we recall the definition of a (pre)sheaf and a flabby sheaf.

\begin{defn}[{(pre)sheaf, \cite[1.1 Definition, p.\ 1, 1.7, p.\ 6]{Bre}}]
Let $X$ be a topological space and for every open $U\subset X$ let there be a vector space 
$\mathcal{F}(U)$ such that $\mathcal{F}(\varnothing)=0$ 
and for every pair $V\subset U$ of open sets in $X$ let there be a linear map 
$R_{U,V}\colon \mathcal{F}(U)\to\mathcal{F}(V)$. Let 
$\mathcal{F}\coloneq\{\mathcal{F}(U)\;|\;U\subset X\;\text{open}\}$ and 
$R\coloneq R^{\mathcal{F}}\coloneq\{R_{U,V}\;|\;U,V\subset X\;\text{open}\}$. 
The tuple $(\mathcal{F},R)$ is called a \emph{presheaf} on $X$ and 
the maps in $R$ \emph{restrictions} if:
\begin{enumerate}
\item[(i)] $R_{U,U}=\id$ for every open $U$, and
\item[(ii)] $R_{V,W}\circ R_{U,V}=R_{U,W}$ for every open $W\subset V\subset U$.
\end{enumerate}
A presheaf $(\mathcal{F},R)$ is a \emph{sheaf} on $X$ if for every family of open sets 
$\{U_{j}\;|\;j\in J\}$ with $U\coloneq\bigcup_{j\in J} U_{j}$ the following is valid:
\begin{enumerate}
\item[(S1)] If $f\in\mathcal{F}(U)$ is such that $R_{U,U_{j}}(f)=0$ for all $j\in J$, then $f=0$.
\item[(S2)] Let $f_{j}\in\mathcal{F}(U_{j})$, $j\in J$, be given such that for every pair $(j,i)\in J^2$
\[
R_{U_{j},U_{j}\cap U_{i}}(f_{j})=R_{U_{i},U_{j}\cap U_{i}}(f_{i})
\]
holds. Then there is $f\in\mathcal{F}(U)$ such that $R_{U,U_{j}}(f)=f_{j}$.
\end{enumerate}
\end{defn}

\begin{defn}[{flabby, \cite[5.1 Definition, p.\ 47]{Bre}}]
Let $X$ be a topological space. A sheaf $(\mathcal{F},R)$ on $X$ is called \emph{flabby} if 
$R_{X,U}\colon\mathcal{F}(X)\to\mathcal{F}(U)$ is surjective for every open set $U\subset X$.
\end{defn}

The following simple observation will turn out to be a useful tool in the proof of \prettyref{thm:sheaf_flabby} c).

\begin{prop}[{\cite[6.6 Proposition, p.\ 115]{ich}}]\label{prop:sheaf_presheaf_isom}
Let $X$ be a topological space, $(\mathcal{G},R^{\mathcal{G}})$ a presheaf and $(\mathcal{F},R^{\mathcal{F}})$ a sheaf on $X$. 
Let $h\colon\mathcal{G}\to\mathcal{F}$ be a homomorphism of presheaves such that 
$h_{\Omega}\colon \mathcal{G}(\Omega)\to\mathcal{F}(\Omega)$ is an isomorphism for every open set $\Omega\subset X$. 
Then $(\mathcal{G},R^{\mathcal{G}})$ is a sheaf (and $h$ an isomorphism of sheaves).
\end{prop}
\begin{proof}
First, we remark that $h\colon\mathcal{G}\to\mathcal{F}$ is a homomorphism of presheaves (see \cite[p.\ 8]{Bre}), i.e.\ the diagram
\[
\begin{xy}
\xymatrix{
\mathcal{G}(\Omega) \ar[r]^{h_{\Omega}} \ar[d]_{R^{\mathcal{G}}_{\Omega,\Omega_{1}}} & \mathcal{F}(\Omega) \ar[d]^{R^{\mathcal{F}}_{\Omega,\Omega_{1}}} \\
\mathcal{G}(\Omega_{1}) \ar[r]_{h_{\Omega_{1}}}                                      &   \mathcal{F}(\Omega_{1})
}
\end{xy}
\]
commutes for open sets $\Omega_{1}\subset\Omega\subset X$. Let $f\in\mathcal{F}(\Omega)$. 
Since $h_{\Omega}$ and $h_{\Omega_{1}}$ are isomorphisms by our assumption, we have
\[
 h^{-1}_{\Omega_{1}}(f_{\mid\Omega_{1}})
=h^{-1}_{\Omega_{1}}(h_{\Omega}(h^{-1}_{\Omega}(f))_{\mid\Omega_{1}})
=h^{-1}_{\Omega_{1}}(h_{\Omega_{1}}(h^{-1}_{\Omega}(f)_{\mid\Omega_{1}}))
=h^{-1}_{\Omega}(f)_{\mid\Omega_{1}}
\]
since $h$ is a homomorphism of presheaves, which means that the diagram
\[
\begin{xy}
\xymatrix{
\mathcal{G}(\Omega) \ar[d]_{R^{\mathcal{G}}_{\Omega,\Omega_{1}}} & \ar[l]_{h^{-1}_{\Omega}} \mathcal{F}(\Omega) \ar[d]^{R^{\mathcal{F}}_{\Omega,\Omega_{1}}} \\
\mathcal{G}(\Omega_{1})                                          & \mathcal{F}(\Omega_{1}) \ar[l]^{h^{-1}_{\Omega_{1}}}
}
\end{xy}
\]
commutes as well, so $h^{-1}$ is homomorphism of presheaves.

$(S1)$: Let $\{\Omega_{j}\;|\;j\in J\}$ be a familiy of open subsets of $X$ and $\Omega\coloneq\bigcup_{j\in J}\Omega_{j}$. 
Let $f\in\mathcal{G}(\Omega)$ such that $f_{\mid\Omega_{j}}=0$ for all $j\in J$. 
Then $h_{\Omega}(f)\in\mathcal{F}(\Omega)$ and
\[
h_{\Omega}(f)_{\mid\Omega_{j}}=h_{\Omega}(f_{\mid\Omega_{j}})=h_{\Omega}(0)=0
\]
for all $j\in J$ due to the assumption and since $h$ is a homomorphism of presheaves. 
As $\mathcal{F}$ is a sheaf, hence satisifies $(S1)$, we obtain $h_{\Omega}(f)=0$. 
Due to the injectivity of $h_{\Omega}$, we get $f=0$.

$(S2)$: Let $\{\Omega_{j}\;|\;j\in J\}$ and $\Omega$ be like above. Let $f_{j}\in \mathcal{G}(\Omega_{j})$ 
such that ${f_{j}}_{\mid\Omega_{j}\cap\Omega_{k}}={f_{k}}_{\mid\Omega_{j}\cap\Omega_{k}}$ for all $j,k\in J$. 
Then $h_{\Omega_{j}}(f_{j})\in\mathcal{F}(\Omega_{j})$ and
\[
 h_{\Omega_{j}}(f_{j})_{\mid\Omega_{j}\cap\Omega_{k}}-h_{\Omega_{k}}(f_{k})_{\mid\Omega_{j}\cap\Omega_{k}}
=h_{\Omega_{j}\cap\Omega_{k}}({f_{j}}_{\mid\Omega_{j}\cap\Omega_{k}})
 -h_{\Omega_{j}\cap\Omega_{k}}({f_{k}}_{\mid\Omega_{j}\cap\Omega_{k}})
=0
\]
for all $j,k\in J$ by the assumption and since $h$ is a homomorphism of presheaves. 
As $\mathcal{F}$ is a sheaf, hence satisifies $(S2)$, there exists $G\in \mathcal{F}(\Omega)$ 
such that $G_{\mid\Omega_{j}}=h_{\Omega_{j}}(f_{j})$ for every $j\in J$. 
Now, we define $F\coloneq h^{-1}_{\Omega}(G)\in\mathcal{G}(\Omega)$. By virtue of the remark in the beginning, we gain
\[
 F_{\mid\Omega_{j}}=h^{-1}_{\Omega}(G)_{\mid\Omega_{j}}
=h^{-1}_{\Omega_{j}}(G_{\mid\Omega_{j}})
=h^{-1}_{\Omega_{j}}(h_{\Omega_{j}}(f_{j}))
=f_{j}
\]
for all $j\in J$. Therefore, $\mathcal{G}$ is a sheaf and thus $h$ an isomorphism of sheaves.
\end{proof}

For the notions not explained we refer the reader to the literature. 
For the classical theory of (Fourier) hyperfunctions we refer the reader to \cite{graf2010,imai1992,Kan,Mori1,Schapira,stankovic1997}, for the sheaf theory to \cite{Bre,Kultze}, 
for the theory of locally convex spaces to \cite{F/W/Buch,Jarchow,meisevogt1997}, 
for PLS-spaces to \cite{Dom2} and for the theory of $\varepsilon$-products and tensor products to 
\cite{Defant,Jarchow,Kaballo,kruse2017,kruse2023}.

\section{Silva--K\"othe--Grothendieck duality, boundary values and Fourier transformation}
\label{sect:koethe_duality_fourier}

This section is devoted to a duality theorem, a resulting boundary value representation 
and the Fourier transformation.
We recall the well-known topological Silva--K\"othe--Grothendieck isomorphism
\begin{equation}\label{eq:koethe_iso}
\mathcal{O}(\C\setminus K,E)/\mathcal{O}(\C,E)\cong L_{b}(\mathscr{A}(K),E)
\end{equation}
for a quasi-complete $\C$-lcHs $E$ and compact $\varnothing\neq K\subset\R$
(see e.g.\ \cite[p.\ 6]{SebastiaoeSilva1950}, \cite[Proposition 1, p.\ 46]{Grothendieck1953}, 
\cite[Satz 9, p.\ 90]{Tillmann1955}, \cite[\S27.4, p.\ 375--378]{Koethe1969}, 
\cite[Theorem 2.1.3, p.\ 25]{Mori1}).  
Here $\mathcal{O}(\C\setminus K,E)$ is equipped with the compact-open topology, 
the quotient space with the induced quotient topology and $\mathscr{A}(K)$ 
is the space of germs of real analytic functions on $K$ with its inductive limit topology. 
We introduce the spaces $\mathcal{E}^{exp}(\overline{\C}\setminus K,E)$, 
$\mathcal{O}^{exp}(\overline{\C}\setminus K,E)$ and $\mathcal{P}_{\ast}(K)$ for a compact 
set $K\subset\overline{\R}$ in this section 
which will be used in the counterpart of the Silva--K\"othe--Grothendieck isomorphism 
with $\mathcal{O}(\C\setminus K,E)$ and $\mathscr{A}(K)$ replaced by 
$\mathcal{O}^{exp}(\overline{\C}\setminus K,E)$ and $\mathcal{P}_{\ast}(K)$, respectively.
Then we come to a boundary value representation of $L_{b}(\mathcal{P}_{\ast}(\overline{\R},E)$
and define the Fourier transformation on it.

For a compact set $K\subset\overline{\R}$ and $t\in\R$, $t\geq 1$, we define the open sets
\begin{align*}
U_{t}(K)&\coloneq\phantom{\cup} \{ z \in \C\;| \; \d(z,K\cap\C)< \tfrac{1}{t}\} \\
&\phantom{\coloneq}\cup
\begin{cases}
\varnothing &, K \subset \R, \\
(t,\infty)+\mathsf{i} (-\tfrac{1}{t}, \tfrac{1}{t}) &, \infty \in K, \, -\infty \notin K, \\ 
(-\infty, -t)+\mathsf{i} (-\tfrac{1}{t}, \tfrac{1}{t}) &,\infty \notin K, \,  -\infty \in K,\\ 
\bigl((-\infty, -t)\cup (t,\infty)\bigr)+\mathsf{i} (-\tfrac{1}{t}, \tfrac{1}{t}) &, \pm\infty \in K,  
\end{cases}
\end{align*}
\begin{center}
 \includegraphics[scale=0.85]{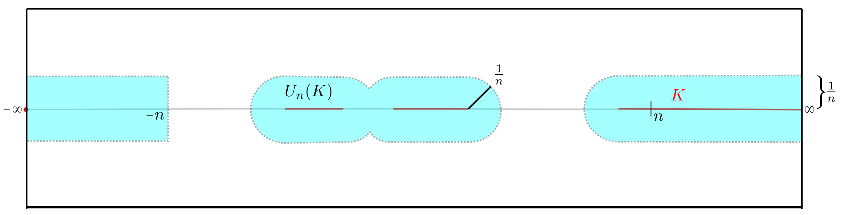}
 \captionsetup{type=figure}
 \caption{$U_{n}(K)$ for $\pm\infty\in K$ (cf.\ Fig.\ 1 \cite[p.\ 83]{kruse2019_2})}
\end{center}
and
\[
S_{t}(K)\coloneq \left(\overline{U_{t}(K)}\right)^{C} \cap \{z\in \C \; | \; |\im(z)| < t \},\quad t>1, 
\]
where the closure and the complement are taken in $\C$.
\begin{center}
 \includegraphics[scale=0.85]{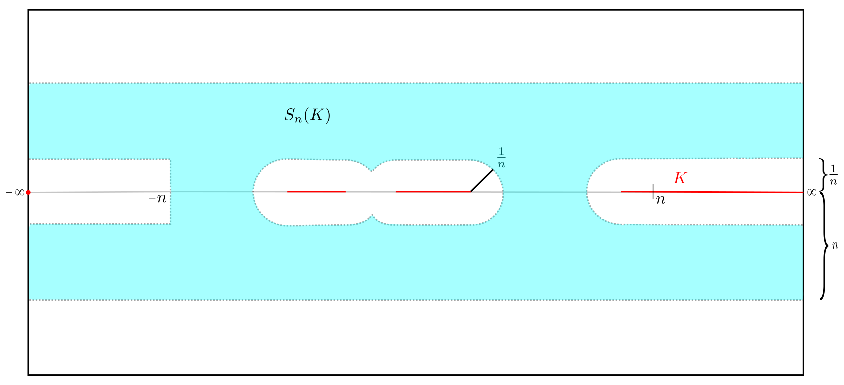}
 \captionsetup{type=figure}
 \caption{$S_{n}(K)$ for $\pm\infty\in K$ (cf.\ Fig.\ 2 \cite[p.\ 83]{kruse2019_2})}
\end{center}
 
\begin{defn}[{\cite[3.2 Definition, p.\ 12--13]{ich}}]\label{def:smooth_weighted_space}
Let $K\subset \overline{\R}$ be a compact set and $E$ a $\C$-lcHs. 
We define the space of $E$-valued slowly increasing smooth functions outside $K$ by
\[
\mathcal{E}^{exp}(\overline{\C}\setminus K,E)\coloneq\{ f\in \mathcal{C}^{\infty}(\C\setminus K, E)\; | 
\;\forall\; n \in \N,\,n\geq 2,\,m\in\N_{0},\,\alpha\in\mathfrak{A}:\;|f|_{K,n,m,\alpha}<\infty\}
\]
where
\[
|f|_{K,n,m,\alpha}\coloneq\sup_{\substack{z\in S_{n}(K)\\ \beta\in\N_{0}^{2},|\beta|\leq m}}
p_{\alpha}(\partial^{\beta}f(z))\e^{-(1/n)|\re(z)|}.
\]
We define the space of $E$-valued slowly increasing holomorphic functions outside $K$ by
\[
\mathcal{O}^{exp}(\overline{\C}\setminus K,E)\coloneq\{ f\in \mathcal{O}(\C\setminus K, E)\; | 
\;\forall\; n \in \N,\,n\geq 2,\,\alpha\in\mathfrak{A}:\;|f|_{K,n,\alpha}<\infty\}
\]
where 
\[
|f|_{K,n,\alpha}\coloneq|f|_{K,n,0,\alpha}=\sup_{z\in S_{n}(K)}p_{\alpha}(f(z))\e^{-(1/n)|\re(z)|}.
\]
Further, we set $\mathcal{E}^{exp}(\overline{\C}\setminus K)\coloneq\mathcal{E}^{exp}(\overline{\C}\setminus K,\C)$ 
and $\mathcal{O}^{exp}(\overline{\C}\setminus K)\coloneq\mathcal{O}^{exp}(\overline{\C}\setminus K,\C)$.
\end{defn}

We exclude the case $n=1$ because
\[
 \left(\overline{U_{1}(\overline{\R})}\right)^{C}\cap\{z\in\C\;|\;|\im(z)|<1\}=\varnothing
\]
but we could also include the case $n=1$ with definition $S_{1}(K)\coloneq S_{t_{0}}(K)$ for some $t_0\in (1,2)$ 
which would not change the spaces $\mathcal{E}^{exp}(\overline{\C}\setminus K,E)$ 
and $\mathcal{O}^{exp}(\overline{\C}\setminus K,E)$. 
In the literature the symbols $\widetilde{\mathcal{E}}(\overline{\C}\setminus K,E)$ resp.\ 
$\widetilde{\mathcal{O}}(\overline{\C}\setminus K,E)$ are also used for 
$\mathcal{E}^{exp}(\overline{\C}\setminus K,E)$ resp.\ $\mathcal{O}^{exp}(\overline{\C}\setminus K,E)$ 
(see \cite[1.2 Definition, p.\ 5]{J}).

Several times we will use the following useful relation between real and complex partial derivatives of 
a holomorphic function.

\begin{prop}[{\cite[Proposition 7.1, p.\ 270]{kruse2019_4}}]
Let $\Omega\subset\C$ be open, $E$ a locally complete $\C$-lcHs and $f\in\mathcal{O}(\Omega,E)$. 
Then $f\in\mathcal{C}^{\infty}(\Omega,E)$ and
\begin{equation}\label{eq:real.compl.part.deriv.1}
\partial^{\beta}f(z)=\mathsf{i}^{\beta_{2}}\partial^{|\beta|}_{\C}f(z),\quad z\in\Omega,\,\beta\in\N_{0}^{2}.
\end{equation}
\end{prop}

For complete $E$ this was already observed in \cite[3.4 Lemma, p.\ 17]{ich}.

\begin{prop}\label{prop:Cauchy_estimates}
Let $E$ be a locally complete $\C$-lcHs, $a_{1}\leq a_{2}<0$, 
$\Omega_{1}\subset\Omega_{2}\subset\C$ be open and let there exist $0<r\leq 1$ 
such that $\overline{\D_{r}(z)}\subset\Omega_{2}$ for all $z\in\Omega_{1}$.
Then for all $f\in\mathcal{O}(\Omega_{2},E)$, $m\in\N_{0}$ and $\alpha\in\mathfrak{A}$ it holds that
\[
    \sup_{\substack{z \in \Omega_{1}\\ k \in \N_{0},k \leq m}}p_{\alpha}(\partial_{\C}^{k}f(z))\e^{a_{1}|\re(z)|}
\leq \e^{-a_{1}r}\frac{m!}{r^{m}}\sup_{z \in \Omega_{2}}p_{\alpha}(f(z))\e^{a_{2}|\re(z)|}.
\]
\end{prop}
\begin{proof}
Let $m\in\N_{0}$. We note that
\[
a_{1}|\re(z)|\leq a_{1}|\re(\zeta)|-a_{1}|\re(z)-\re(\zeta)|\leq a_{2}|\re(\zeta)|-a_{1}r,
\quad z\in\Omega_{1},\;\zeta\in\overline{\D_{r}(z)}.
\]
By Cauchy's inequality \cite[Corollary 5.3 a), p.\ 263]{kruse2019_4} we have 
\[
p_{\alpha}(\partial_{\C}^{k}f(z)) \leq \frac{k!}{r^{k}}\sup_{\zeta\in\C,\,|\zeta-z|=r}p_{\alpha}(f(\zeta)),\quad z\in\Omega_{1},
\]
for all $f\in\mathcal{O}(\Omega_{2},E)$ and $\alpha\in\mathfrak{A}$, which implies 
\begin{flalign*}
&\hspace{0.35cm}\sup_{\substack{z \in \Omega_{1}\\ k \in \N_{0},k \leq m}}p_{\alpha}(\partial_{\C}^{k}f(z))\e^{a_{1}|\re(z)|}
 \leq\sup_{\substack{z \in \Omega_{1}\\ k \in \N_{0},k \leq m}}
 \frac{k!}{r^{k}}\sup_{\zeta\in\C,\,|\zeta-z|=r}p_{\alpha}(f(\zeta))\e^{a_{1}|\re(z)|}\\
&\leq \e^{-a_{1}r}\frac{m!}{r^{m}}\sup_{\zeta\in\Omega_{2}}p_{\alpha}(f(\zeta))\e^{a_{2}|\re(\zeta)|}.
\end{flalign*}
\end{proof}

\prettyref{prop:Cauchy_estimates} in combination with \eqref{eq:real.compl.part.deriv.1} implies 
that the topology of $\mathcal{O}^{exp}(\overline{\C}\setminus K,E)$ coincides 
with the induced topology of $\mathcal{E}^{exp}(\overline{\C}\setminus K,E)$ for 
compact $K\subset\overline{\R}$ and locally complete $E$. 
For Fr\'echet spaces $E$ this can also be found in \cite[1.4 Lemma (1), p.\ 5]{J} 
and for complete $E$ in \cite[3.6 Theorem (4), p.\ 21]{ich}. 

\begin{rem}\label{rem:eps_prod_nuclear}
Let $K\subset\overline{\R}$ be compact. 
\begin{enumerate}
\item[a)] The spaces $\mathcal{E}^{exp}(\overline{\C}\setminus K)$ and $\mathcal{O}^{exp}(\overline{\C}\setminus K)$ 
are nuclear Fr\'echet spaces, e.g.\ by \cite[Proposition 3.7, p.\ 240]{kruse2018_2} and \cite[Theorem 3.1, p.\ 188]{kruse2018_4} 
combined with \cite[Remark 2.7, p.\ 178--179]{kruse2018_4} and \cite[Example 2.8 (ii), p.\ 179]{kruse2018_4}. 
\item[b)] We have the topological isomorphisms
\[
\mathcal{E}^{exp}(\overline{\C}\setminus K,E)\cong\mathcal{E}^{exp}(\overline{\C}\setminus K)\varepsilon E
\quad\text{and}\quad
\mathcal{O}^{exp}(\overline{\C}\setminus K,E)\cong\mathcal{O}^{exp}(\overline{\C}\setminus K)\varepsilon E
\]
for a locally complete $\C$-lcHs $E$ by \cite[3.23 Corollary c), p.\ 316]{kruse2018_3} and thus
\[
\mathcal{E}^{exp}(\overline{\C}\setminus K,E)\cong\mathcal{E}^{exp}(\overline{\C}\setminus K)\widehat{\otimes}_{\pi} E
\quad\text{and}\quad
\mathcal{O}^{exp}(\overline{\C}\setminus K,E)\cong\mathcal{O}^{exp}(\overline{\C}\setminus K)\widehat{\otimes}_{\pi} E
\]
for a complete $\C$-lcHs $E$ due to nuclearity. 
\end{enumerate}
\end{rem}

Part a) of the preceding remark can also be found in \cite[1.4 Lemma (2), p.\ 5]{J} and \cite[1.6 Folgerung, p.\ 7]{J} 
as well as \cite[3.6 Theorem (2), p.\ 21]{ich} and \cite[3.7 Theorem, p.\ 23]{ich}. 
Part b) for Fr\'echet spaces $E$ is also given in \cite[1.7 Satz, p.\ 8--9]{J} and for complete spaces $E$ in 
\cite[3.11 Theorem, p.\ 31]{ich} and \cite[3.12 Corollary, p.\ 35]{ich}.

\begin{prop}\label{prop:DFS}
Let $K\subset\overline{\R}$ be a non-empty compact set. For $n\in\N$ let
\[
\mathcal{O}_{n}(\overline{U_{n}(K)})\coloneq\{f\in\mathcal{O}(U_{n}(K))\cap\mathcal{C}^{0}(\overline{U_{n}(K)})\;|\; 
\|f\|_{K,n}< \infty \}
\]
where
\[
\|f\|_{K,n}\coloneq\sup_{z \in \overline{U_{n}(K)}}|f(z)|\e^{(1/n)|\re(z)|}
\]
and the spectral maps for $n,k\in\N$, $n\leq k$, be given by the restrictions
\[
\pi_{n,k}\colon \mathcal{O}_{n}(\overline{U_{n}(K)})\to\mathcal{O}_{k}(\overline{U_{k}(K)}), \;
\pi_{n,k}(f)\coloneq f_{\mid U_{k}(K)}.
\]
Then the space of rapidly decreasing holomorphic germs near $K\neq\varnothing$ given by the inductive limit 
\[
 \mathcal{P}_{\ast}(K)\coloneq\lim_{\substack{\longrightarrow\\n\in \N}}\mathcal{O}_{n}(\overline{U_{n}(K)})
\]
exists and is a DFS-space. 
If $K=\varnothing$, we set $\mathcal{P}_{\ast}(\varnothing)\coloneq 0$.
\end{prop}

The preceding proposition is a special case of \cite[Proposition 4 (a), p.\ 86]{kruse2019_2}. 
It is already mentioned in \cite[p.\ 469]{Kawai} resp.\ proved in 
\cite[1.11 Satz, p.\ 11]{J} and \cite[3.5 Theorem, p.\ 17]{ich} 
that $\mathcal{P}_{\ast}(K)$ is a DFS-space. In the literature the symbol  
$\utilde{\mathcal{O}}(K)$ is also used for $\mathcal{P}_{\ast}(K)$ and the symbol $\mathcal{P}_{\ast}$ 
for the special case $\mathcal{P}_{\ast}(\overline{\R})$ (see \cite[Definition 1.1.3, p.\ 468--469]{Kawai}).
If $\varnothing\neq K\subset\R$ is compact, then $\mathcal{P}_{\ast}(K)=\mathscr{A}(K)$.
The counterpart of the Silva--K\"othe--Grothendieck isomorphism \eqref{eq:koethe_iso} 
for vector-valued slowly increasing holomorphic functions 
outside a non-empty compact set $K\subset\overline{\R}$ reads as follows.

\begin{thm}[{\cite[Corollary 15, p.\ 102, Eq.\ (5), p.\ 93, Eq.\ (13), p.\ 101]{kruse2019_2}}]\label{thm:duality}
Let $K\subset\overline{\R}$ be a non-empty compact set and $E$ a sequentially complete $\C$-lcHs. 
Then the map 
\[
 H_{K}\colon \mathcal{O}^{exp}(\overline{\C}\setminus K,E)/\mathcal{O}^{exp}(\overline{\C},E)
 \to L_{b}(\mathcal{P}_{\ast}(K),E)
\]
given by 
\[
 H_{K}(f)(\varphi)\coloneq\int_{\gamma_{K,n,r}}F(\zeta)\varphi(\zeta)\d\zeta
\]
for $f=[F]\in \mathcal{O}^{exp}(\overline{\C}\setminus K,E)/\mathcal{O}^{exp}(\overline{\C},E)$ and 
$\varphi\in \mathcal{O}_{n}(\overline{U_{n}(K)})$, $n\in\N$, where the integral is a Pettis-integral and 
$\gamma_{K,n,r}$ a suitable path along $K$ in $\overline{U_{n}(K)}$, 
is a topological isomorphism. 
\begin{center}
\includegraphics[scale=0.85]{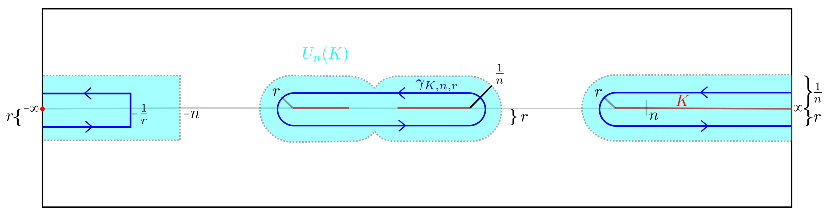}
\captionsetup{type=figure}
\caption{Path $\gamma_{K,n,r}$ for $\pm\infty\in K$ (cf.\ Fig.\ 3 \cite[p.\ 89]{kruse2019_2})}\label{fig:path_gamma}
\end{center}
Its inverse 
\[
 H_{K}^{-1}\colon  L_{b}(\mathcal{P}_{\ast}(K),E)
 \to\mathcal{O}^{exp}(\overline{\C}\setminus K,E)/\mathcal{O}^{exp}(\overline{\C},E)
\]
is given by  
\[
 H_{K}^{-1}(T)\coloneq\bigl[\C\setminus K \ni z\longmapsto \frac{1}{2\pi\mathsf{i}}\langle T,\frac{\e^{-(z-\cdot)^2}}{z-\cdot}\rangle\bigr],
 \quad T\in L_{b}(\mathcal{P}_{\ast}(K),E), 
\] 
In addition, for all non-empty compact sets $K_{1}\subset K$ it holds that
\begin{equation}\label{eq:unabh.H}
{H_{K}}_{\mid\mathcal{O}^{exp}(\overline{\C}\setminus K_{1},E)/\mathcal{O}^{exp}(\overline{\C},E)}=H_{K_{1}}
\end{equation}
on $\mathcal{P}_{\ast}(K)$ and
\begin{equation}\label{eq:unabh.H_inv}
H_{K}^{-1}(T)=H_{\overline{\R}}^{-1}(T),\quad T\in L(\mathcal{P}_{\ast}(K), E).
\end{equation}
\end{thm}

The topological isomorphism 
$\mathcal{O}^{exp}(\overline{\C}\setminus K,E)/\mathcal{O}^{exp}(\overline{\C},E) 
\cong L_{b}(\mathcal{P}_{\ast}(K),E)$ in \prettyref{thm:duality} is already known for special cases. 
For $E=\C$ it can be found in \cite[Theorem 3.2.1, p.\ 480]{Kawai}, and 
if $E$ is a Fr\'echet space in \cite[3.9 Satz, p.\ 41]{J} but the proof is different. 
For an interval $K=[a,\infty]$, $a\in\R$, and $E=\C$ the duality in \prettyref{thm:duality} was proved in
\cite[Theorem 3.3, p.\ 85--86]{Mori2} and was the starting point to prove \prettyref{thm:duality} for complete $E$ in
\cite[4.1 Theorem, p.\ 41]{ich}. The map $\Theta_{K}$ is also called (weighted) \emph{Cauchy transformation} 
(see \cite[p.\ 84]{Mori2}).

As nuclearity is inherited by quotient spaces, we derive from \prettyref{rem:eps_prod_nuclear} a), 
\prettyref{thm:duality} with $E=\C$ and the reflexivity of $\mathcal{P}_{\ast}(K)$ that 
$\mathcal{P}_{\ast}(K)$ is nuclear for every compact set $K\subset\overline{\R}$ 
(cf.\ \cite[1.11 Satz, p.\ 11]{J}).
By \cite[Theorem 2.2.1, p.\ 474]{Kawai} (see the correction
of its proof in \cite[Remark, p.\ 247--248]{saburi1985}) $\mathcal{P}_{\ast}(\overline{\R})$ is dense in $\mathcal{P}_{\ast}(K)$ 
for a non-empty compact set $K\subset\overline{\R}$. So for different compact sets $K, J \subset\overline{\R}$ 
we may identify elements of $L(\mathcal{P}_{\ast}(K), E)$ and $L(\mathcal{P}_{\ast}(J), E)$ 
by means of their restrictions to $\mathcal{P}_{\ast}(\overline{\R})$. 
Then the following result defining the support of a vector-valued $\mathcal{P}_{\ast}$-functional is valid, 
whose counterpart for compact subsets of $\R$ is given in \cite[Proposition 5.3, p.\ 1121]{D/L}.

\begin{prop}[{support}]\label{prop:traeger}
Let $K\subset\overline{\R}$ be compact and $E$ a sequentially complete $\C$-lcHs.
\begin{enumerate}
\item [a)] If $J\subset\overline{\R}$ is compact and $K\cap J\neq\varnothing$, then
\[
L(\mathcal{P}_{\ast}(K), E)\cap L(\mathcal{P}_{\ast}(J), E)=L(\mathcal{P}_{\ast}(K\cap J), E).
\]
\item [b)]
For any $T\in L(\mathcal{P}_{\ast}(K), E)$ there is a minimal compact set $J\subset K$ such that $T\in L(\mathcal{P}_{\ast}(J), E)$. 
The set $J$ is called the \emph{support} of $T$.
\end{enumerate}
\end{prop}
\begin{proof}
\begin{enumerate}
\item [a)]
``$\subset$'': Let $T\in L(\mathcal{P}_{\ast}(K), E)\cap L(\mathcal{P}_{\ast}(J), E)$. Then $H_{K}^{-1}(T)=H_{J}^{-1}(T)$ 
by \eqref{eq:unabh.H_inv} and
\begin{gather*}
 H_{K}^{-1}(T)\in \bigl(\mathcal{O}^{exp}(\overline{\C}\setminus K, E)/\mathcal{O}^{exp}(\overline{\C}, E)\bigr)
 \cap\bigl(\mathcal{O}^{exp}(\overline{\C}\setminus J, E)/\mathcal{O}^{exp}(\overline{\C}, E)\bigr) \\
=\mathcal{O}^{exp}(\overline{\C}\setminus(K\cap J), E)/\mathcal{O}^{exp}(\overline{\C}, E)
\end{gather*}
and $T=(H_{K\cap J}\circ H_{K}^{-1})(T)\in L(\mathcal{P}_{\ast}(K\cap J), E)$ by \prettyref{thm:duality} and \eqref{eq:unabh.H_inv}. 
The other inclusion is obvious.
\item [b)]
This is clear by \prettyref{thm:duality} since for any $f\in\mathcal{O}^{exp}(\overline{\C}\setminus K, E)$ 
there is a minimal compact set $J$ such that $f\in\mathcal{O}^{exp}(\overline{\C}\setminus J, E)$.
\end{enumerate}
\end{proof}

\begin{rem}
Let $E$ be a $\C$-lcHs. \prettyref{prop:traeger} improves \cite[4.3 Proposition, p.\ 50]{ich} 
from complete $E$ to sequentially complete $E$.
If $K,J\subset\R$ are compact sets with $K\cap J\neq\varnothing$, then \prettyref{prop:traeger} is still valid for 
locally complete $E$ since \prettyref{thm:duality} holds in this case as well by 
\cite[Remark 12, p.\ 101]{kruse2019_2}.
\end{rem}

For $K=\overline{\R}$ we look at the duality \prettyref{thm:duality} once again, but from a different point of view. 
Let $f\in\mathcal{O}^{exp}(\overline{\C}\setminus \overline{\R},E)$. 
In the spirit of \cite{L2} and \cite[Chap.\ II, p.\ 77--97]{Schapira} 
we assign the boundary value
\[
  \langle R(f),\varphi\rangle
\coloneq\lim_{t,\;t'\searrow 0}\langle R_{t,t'}(f),\varphi\rangle
\coloneq\lim_{t,\;t'\searrow 0}\int_{\R}(f(x+\mathsf{i}t)-f(x-\mathsf{i}t'))\varphi(x)\d x,
  \quad\varphi\in\mathcal{P}_{\ast}(\overline{\R}),
\]
to this function, if the limit in $E$ of this Pettis-integral exists. Furthermore, we define the upper boundary value by
\[
  \langle R^{+}(f),\varphi\rangle
\coloneq\lim_{t\searrow 0}\langle R^{+}_{t}(f),\varphi\rangle
\coloneq\lim_{t\searrow 0}\int_{\R}f(x+\mathsf{i}t)\varphi(x)\d x,
  \quad\varphi\in\mathcal{P}_{\ast}(\overline{\R}),
\]
and the lower boundary value by
\[
  \langle R^{-}(f),\varphi\rangle
\coloneq\lim_{t\searrow 0}\langle R^{-}_{t}(f),\varphi\rangle
\coloneq\lim_{t\searrow 0}\int_{\R}f(x-\mathsf{i}t)\varphi(x)\d x,
  \quad\varphi\in\mathcal{P}_{\ast}(\overline{\R}),
\]
if the limit in $E$ of these Pettis-integrals exists.

\begin{thm}\label{thm:bv_int}
Let $E$ be a sequentially complete $\C$-lcHs.
\begin{enumerate}
\item [a)] $(R_{t,t'}(f))$, $(R_{t}^{+}(f))$ and $(R_{t}^{-}(f))$ converge 
to $R(f)$, $R^{+}(f)$ and $R^{-}(f)$ in $L_{b}(\mathcal{P}_{\ast}(\overline{\R}),E)$, 
respectively, and
\[
R(f)=R^{+}(f)-R^{-}(f)=-H_{\overline{\R}}([f])
\]
for all $f\in\mathcal{O}^{exp}(\overline{\C}\setminus \overline{\R},E)$.
\item [b)] The map $[f]\mapsto R(f)$ from 
$\mathcal{O}^{exp}(\overline{\C}\setminus \overline{\R},E)/\mathcal{O}^{exp}(\overline{\C},E)$ 
to $L_{b}(\mathcal{P}_{\ast}(\overline{\R}),E)$ is a topological isomorphism.
\end{enumerate}
\end{thm}
\begin{proof}
a) Let $f\in\mathcal{O}^{exp}(\overline{\C}\setminus \overline{\R},E)$. We fix $t>0$ and observe that 
$f(\cdot\pm \mathsf{i}t)\varphi$ is Pettis integrable for every $\varphi\in\mathcal{P}_{\ast}(\overline{\R})$ 
on every compact set $K\subset\R$ by \cite[4.7 Lemma, p.\ 369]{kruse2018_1}, i.e.\ there is $e_{K}^{\pm}=e_{K}^{\pm}(\pm t,f,\varphi)\in E$ such that
\[
 \langle e', e^{\pm}_{K} \rangle=\int_{K}\langle e',f(x\pm \mathsf{i}t)\varphi(x)\rangle \d x,\quad e'\in E'.
\]
We set $a_{k,-}^{-}\coloneq e^{-}_{[-k,0]}$ and $a_{k,+}^{-}\coloneq e^{-}_{[0,k]}$ for $k\in\N$. Let $\alpha\in\mathfrak{A}$, $n\in\N$ and 
$\varphi\in\mathcal{O}_{n}(U_{n}(\overline{\R}))$. We choose $m\in\N$ with $m>2\max(n,t)$ and 
observe that for $k,p\in\N$, $k>p$,
\begin{align*}
     p_{\alpha}(a_{k,-}^{-}-a_{p,-}^{-})
&\leq\int_{[-k,-p]}p_{\alpha}(f(x\pm \mathsf{i}t))|\varphi(x)|\d x
 \leq |f|_{\overline{\R},m,\alpha}\|\varphi\|_{\overline{\R},n}\int^{-p}_{-k}\e^{\frac{1}{m}|x-\mathsf{i} t|-\frac{1}{n}|x|}\d x\\
&\leq \e^{\frac{1}{2n}t}|f|_{\overline{\R},m,\alpha}\|\varphi\|_{\overline{\R},n}\int^{-p}_{-k}\e^{-\frac{1}{2n}|x|}\d x.
\end{align*}
We derive that $(a_{k,-}^{-})$ and, analogously, $(a_{k,+}^{-})$ are Cauchy sequences in the sequentially complete space $E$. 
Hence they have limits $a^{-}_{-}$ resp.\ $a^{-}_{+}$ in $E$ and it is easy to check that 
\[
 \langle e', a_{-}^{-}+a_{+}^{-} \rangle=\int_{\R}\langle e',f(x-\mathsf{i}t)\varphi(x)\rangle \d x,\quad e'\in E',
\]
which means that $f(\cdot-\mathsf{i}t)\varphi$ is Pettis-integrable on $\R$ with $R^{-}_{t}(f)(\varphi)=a_{-}^{-}+a_{+}^{-}$. 
In the same way, it follows that $f(\cdot+\mathsf{i}t)\varphi$ is Pettis-integrable on $\R$.
Further, we obtain
\[
     p_{\alpha}(R^{\pm}_{t}(f)(\varphi))
\leq 2\e^{\frac{1}{2n}t}|f|_{\overline{\R},m,\alpha}\|\varphi\|_{\overline{\R},n}\int^{\infty}_{0}\e^{-\frac{1}{2n}x}\d x
= 4n\e^{\frac{1}{2n}t}|f|_{\overline{\R},m,\alpha}\|\varphi\|_{\overline{\R},n}<\infty.
\]
Thus $R^{\pm}_{t}(f)\in L(\mathcal{O}_{n}(U_{n}(\overline{\R})),E)$ for every $n\in\N$, 
implying $R^{\pm}_{t}(f)\in L(\mathcal{P}_{\ast}(\overline{\R}),E)$.

Now, we set $\varphi^{\pm}_{t}(x)\coloneq\varphi(x\pm \mathsf{i}t)$. Then the functions
\begin{equation}\label{eq:bv_int_1}
t\mapsto R^{\pm}_{t}(f)(\varphi^{\pm}_{t})=\int_{\R}f(x\pm \mathsf{i}t)\varphi(x\pm \mathsf{i}t)\d x
\end{equation}
are defined for $\varphi\in\mathcal{O}_{n}(U_{n}(\overline{\R}))$, $n\in\N$, on $(0,\tfrac{1}{n})$ 
and constant by Cauchy's integral theorem (see the proof of \cite[Proposition 8 (c), p.\ 89]{kruse2019_2}). 
Thus the limits $\lim_{t\searrow 0}R^{\pm}_{t}(f)(\varphi^{\pm}_{t})$ exist in $E$ 
for every $\varphi\in\mathcal{P}_{\ast}(\overline{\R})$.

Let $\alpha\in\mathfrak{A}$, $n\in\N$, and $\varphi\in\mathcal{O}_{n}(U_{n}(\overline{\R}))$. 
For $0<t<\tfrac{1}{3n}$ and $z\in\overline{U_{3n}(K)}$ we have
\begin{flalign*}
&\hspace{0.35cm}|\varphi(z)-\varphi(z\pm \mathsf{i}t)|\e^{\frac{1}{3n}|\re(z)|}\\
&=|\int_{[z\pm \mathsf{i}t,z]}{\varphi'(w)\d w}|\e^{\frac{1}{3n}|\re(z)|}
  \leq t\sup_{w\in[z\pm \mathsf{i}t,z]}|\varphi'(w)|\e^{\frac{1}{3n}|\re(z)|}\\
&\leq t\sup_{w\in[z\pm \mathsf{i}t,z]} 6n\max_{|\zeta-w|
 =\frac{1}{6n}}|\varphi(\zeta)|\e^{\frac{1}{3n}|\re(z)|}\\
&\leq 6n\e^{\frac{1}{18n^{2}}}t\sup_{w\in[z\pm \mathsf{i}t,z]}
 \max_{|\zeta-w|=\frac{1}{6n}}|\varphi(\zeta)|\e^{\frac{1}{3n}|\re(\zeta)|}
 \leq 6n\e^{\frac{1}{18n^{2}}}\|\varphi\|_{\overline{\R},n}t
\end{flalign*}
by Cauchy's inequality where $[z\pm \mathsf{i}t,z]$ is
the line segment from $z\pm \mathsf{i}t$ to $z$. Hence we get
\begin{equation}\label{letzter-satz}
\|\varphi-\varphi(\cdot\pm \mathsf{i}t)\|_{\overline{\R},3n}\leq 6n\e^{\frac{1}{18n^{2}}}\|\varphi\|_{\overline{\R},n}t.
\end{equation}
Further, we have for $0<t<\tfrac{1}{3n}$ and $x\in\R$
\[
|\im(x\pm\mathsf{i}\tfrac{1}{3n})|=\tfrac{1}{3n}\quad\text{plus}\quad 
6n>\tfrac{1}{n}>|\im(x\pm t\pm\mathsf{i}\tfrac{1}{3n})|=t+\tfrac{1}{3n}>\tfrac{1}{6n}.
\]
Due to Cauchy's integral theorem we obtain for all $0<t<\tfrac{1}{3n} $
\begin{align*}
  p_{\alpha}(R^{\pm}_{t}(f)(\varphi)-R^{\pm}_{t}(f)(\varphi^{\pm}_{t}))
&=p_{\alpha}\bigl(\int_{\R}f(x\pm \mathsf{i}t)(\varphi(x)-\varphi(x\pm \mathsf{i}t))\d x\bigr)\\
&=p_{\alpha}\bigl(\int_{\R}f(x\pm \mathsf{i}t \pm \mathsf{i}\tfrac{1}{3n})(\varphi(x\pm \mathsf{i}\tfrac{1}{3n})
  -\varphi(x\pm \mathsf{i}t\pm \mathsf{i}\tfrac{1}{3n}))\d x\bigr)\\
&\leq|f|_{\overline{\R},6n,\alpha}\|\varphi-\varphi(\cdot \pm \mathsf{i}t)\|_{\overline{\R},3n}
 \int^{\infty}_{-\infty}\e^{\frac{1}{6n}|x\pm \mathsf{i}t|-\frac{1}{3n}|x|}\d x\\
&\leq 12n\e^{\frac{1}{6n}t}|f|_{\overline{\R},6n,\alpha}\|\varphi-\varphi(\cdot\pm \mathsf{i}t)\|_{\overline{\R},3n}\\ 
&\underset{\mathclap{\eqref{letzter-satz}}}{\leq} (72n^{2}\e^{\frac{1}{18n^{2}}}\|\varphi\|_{\overline{\R},n}
 |f|_{\overline{\R},6n,\alpha})\e^{\frac{1}{6n}t}t\underset{t \searrow 0}{\to}0.
\end{align*}
Since the limits $\lim_{t\searrow 0}R^{\pm}_{t}(f)(\varphi^{\pm}_{t})$ exist in $E$ 
for every $\varphi\in\mathcal{P}_{\ast}(\overline{\R})$, 
we deduce that the limits $\langle R^{\pm}(f),\varphi\rangle=\lim_{t\searrow 0}R^{\pm}_{t}(f)(\varphi)$ exist in $E$, more precisely,
\[
\langle R^{\pm}(f),\varphi\rangle=\lim_{t\searrow 0}R^{\pm}_{t}(f)(\varphi)=\lim_{t\searrow 0}R^{\pm}_{t}(f)(\varphi^{\pm}_{t}).
\]
The space $\mathcal{P}_{\ast}(\overline{\R})$ is a DFS-space by \prettyref{prop:DFS} and hence a Montel space. 
Thus it is barrelled and by the Banach--Steinhaus theorem \cite[10.3.4 Satz, p.\ 53]{F/W/Buch} we obtain 
that $R^{\pm}(f)\in L_{b}(\mathcal{P}_{\ast}(\overline{\R}),E)$ and 
$(R^{\pm}_{t}(f))$ converges to $R^{\pm}(f)$ in $L_{b}(\mathcal{P}_{\ast}(\overline{\R}),E)$ as $t\searrow0$.
Furthermore, we get
\begin{align}\label{eq:bv_int_2}
  \langle R(f),\varphi\rangle 
&=\lim_{t,\;t'\searrow 0}(R^{+}_{t}(f)(\varphi)-R^{-}_{t'}(f)(\varphi))
 =\lim_{t\searrow 0}R^{+}_{t}(f)(\varphi)-\lim_{t\searrow 0}R^{-}_{t}(f)(\varphi)\nonumber\\
&=\langle R^{+}(f),\varphi\rangle-\langle R^{-}(f),\varphi\rangle
 =\lim_{t\searrow 0}(R^{+}_{t}(f)(\varphi^{+}_{t})- R^{-}_{t}(f)(\varphi^{-}_{t}))\nonumber\\
&=\lim_{t\searrow 0}\bigl(\int_{\R}f(x+\mathsf{i}t)\varphi(x+\mathsf{i}t)\d x-\int_{\R}f(x-\mathsf{i}t)
  \varphi(x-\mathsf{i}t)\d x\bigr)
 =-H_{\overline{\R}}([f])(\varphi)
\end{align}
for every $\varphi\in\mathcal{P}_{\ast}(\overline{\R})$ by \prettyref{thm:duality} and \cite[Proposition 8 (c), p.\ 89]{kruse2019_2}. 
In particular, this means that $R_{t,t'}(f)$ converges to $R(f)$ in $L_{b}(\mathcal{P}_{\ast}(\overline{\R}),E)$ as $t,t'\searrow 0$
for every $f\in \mathcal{O}^{exp}(\overline{\C}\setminus \overline{\R},E)$.

b) By the first part the considered map coincides with $-H$ and the statement follows directly by \prettyref{thm:duality}.
\end{proof}

\prettyref{thm:bv_int} improves \cite[4.5 Theorem, p.\ 51]{ich} from complete $E$ to sequentially complete $E$. 
In particular, this theorem contains, at least in one variable, \cite[Theorem 3.2.9, p.\ 483--484]{Kawai} for $E=\C$ and 
\cite[Satz 3.13, p.\ 44]{J} for Fr\'echet spaces $E$, where it is stated that the map
\[
\widetilde{R}\colon\mathcal{O}^{exp}(\overline{\C}\setminus \overline{\R},E)/\mathcal{O}^{exp}(\overline{\C},E)
\to L_{b}(\mathcal{P}_{\ast}(\overline{\R}),E),
\]
defined by
\[ 
\widetilde{R}([f])(\varphi)\coloneq R^{+}_{t}(f)(\varphi^{+}_{t})-R^{-}_{t}(f)(\varphi^{-}_{t})
\]
for $f\in\mathcal{O}^{exp}(\overline{\C}\setminus\overline{\R},E)$ and $\varphi\in\mathcal{P}_{\ast}(\overline{\R})$ 
and fixed $t$ small enough, is a topological isomorphism. 
This result is contained since the functions in \eqref{eq:bv_int_1} are constant and due to \eqref{eq:bv_int_2}. 

Finally, we define the Fourier transformation on $L_{b}(\mathcal{P}_{\ast}(\overline{\R}),E)$. 
By \cite[Proposition 3.2.4, p.\ 483]{Kawai} (cf.\ \cite[Proposition 8.2.2, p.\ 376]{Kan}) 
the Fourier transformation $\mathscr{F}\colon \mathcal{P}_{\ast}(\overline{\R})\to\mathcal{P}_{\ast}(\overline{\R})$ 
defined by
\[
  \mathscr{F}(\varphi)(\zeta)
\coloneq\widehat{\varphi}(\zeta)
\coloneq\int_{\R}\varphi(x)\e^{\mathsf{i}x\zeta}\d x,\quad\varphi\in\mathcal{O}_{n}(U_{n}(\overline{\R})),\,
 \zeta\in U_{k}(\overline{\R}),\,k>n,
\]
is a topological isomorphism. 
The Fourier transformation on $L_{b}(\mathcal{P}_{\ast}(\overline{\R}),E)$ is now defined by transposition 
(see e.g.\ \cite[3.14 Folgerung, p.\ 45]{J}, \cite[Definition 3.2.5, p.\ 483]{Kawai}, 
\cite[4.6 Theorem, p.\ 53]{ich}).

\begin{cor}\label{cor:Fourier-Trafo}
Let $E$ be a $\C$-lcHs. The Fourier transformation 
\begin{align*}
&\mathscr{F}_{\star}\colon L_{b}(\mathcal{P}_{\ast}(\overline{\R}),E)\to L_{b}(\mathcal{P}_{\ast}(\overline{\R}),E),\\ 
&\mathscr{F}_{\star}(T)(\varphi)\coloneq\langle T,\mathscr{F}(\varphi)\rangle,
\quad T\in L_{b}(\mathcal{P}_{\ast}(\overline{\R}),E),\;\varphi\in\mathcal{P}_{\ast}(\overline{\R}),
\end{align*}
is a topological isomorphism with inverse given by $\mathscr{F}_{\star}^{-1}(T)(\varphi)\coloneq\langle T,\mathscr{F}^{-1}(\varphi)\rangle$, 
for $T\in L_{b}(\mathcal{P}_{\ast}(\overline{\R}),E)$, $\varphi\in\mathcal{P}_{\ast}(\overline{\R})$.
\end{cor}

This follows directly from the fact that $\mathscr{F}$ is a topological isomorphism.

\section{Strict admissibility}
\label{sect:strictly_admissible}

In this section we recall some results on the notion of strict admissibility from the introduction.

\begin{defn}[{(strictly) admissible, \cite[p.\ 55]{ich}}]
Let $E$ be a $\C$-lcHs. We call $E$ \emph{admissible} if the Cauchy--Riemann operator
\[
\overline{\partial}\coloneq\tfrac{1}{2}(\partial^{e_{1}}+\mathsf{i}\partial^{e_{2}})
\colon \mathcal{E}^{exp}(\overline{\C}\setminus K,E)\to \mathcal{E}^{exp}(\overline{\C}\setminus K,E)
\]
is surjective for every compact set $K\subset\overline{\R}$.
We call $E$ \emph{strictly admissible} if $E$ is admissible and 
\[
\overline{\partial}\colon \mathcal{C}^{\infty}(U,E)\to\mathcal{C}^{\infty}(U,E)
\]
is surjective for every open set $U\subset\C$.
\end{defn}

Using that $E=\C$ is admissible (see e.g.\ \cite[5.16 Theorem, p.\ 80]{ich} or \cite[Corollary 5.6, Example 5.7 a), p.\ 2702--2703]{kruse2018_5}), 
it follows from Grothendieck's classical theory of tensor products \cite{Gro} and 
\prettyref{rem:eps_prod_nuclear} that Fr\'echet spaces $E$ are 
admissible, from Vogt's splitting theory for Fr\'echet spaces that $E\coloneq F_{b}'$, where $F$ is a Fr\'{e}chet space 
satisfying the condition $(DN)$, is admissible by \cite[Theorem 2.6, p.\ 174]{vogt1983}, and from 
Bonet's and Doma\'nski's splitting theory for PLS-spaces that an ultrabornological PLS-space $E$ having the property $(PA)$ 
is admissible by \cite[Corollary 3.9, p.\ 1112]{D/L} since 
$\mathcal{O}^{exp}(\overline{\C}\setminus K)=\operatorname{ker}\overline{\partial}$ 
in $\mathcal{E}^{exp}(\overline{\C}\setminus K)$ has the property $(\Omega)$ 
(see \cite[Chap.\ 29, Definition, p.\ 367]{meisevogt1997}) by 
\cite[5.20 Theorem, p.\ 85]{ich} or \cite[Corollary 18, p.\ 21]{kruse2019_1} if $K=\varnothing$ and 
\cite[5.22 Theorem, p.\ 92]{ich} or \cite[Corollary 19 (ii), p.\ 109]{kruse2019_2} if $K\neq\varnothing$. 

We recall that a Fr\'echet space $(F,(\vertiii{\cdot}_{k})_{k\in\N})$ satisfies $(DN)$
by \cite[Chap.\ 29, Definition, p.\ 359]{meisevogt1997} if
\[
\exists\;p\in\N\;\forall\;k\in\N\;\exists\;n\in\N,\,C>0\;\forall\;x\in F:\;
\vertiii{x}^{2}_{k}\leq C\vertiii{x}_{p}\vertiii{x}_{n}.
\]
A \emph{PLS-space} is a projective limit $X=\lim\limits_{\substack{\longleftarrow\\N\in\N}}X_{N}$, where the
inductive limits $X_{N}=\lim\limits_{\substack{\longrightarrow\\n\in \N}}(X_{N,n},\vertiii{\cdot}_{N,n})$ 
are DFS-spaces, and it satisfies $(PA)$ if
\begin{gather*}
\forall\;N\;\exists\;M\;\forall\;K\;\exists\;n\;\forall\;m\;\forall\;\eta >0\;\exists\;k,C,r_0 >0\;\forall\;r>r_0\; 
\forall\; x'\in X'_{N}:\\
     \vertiii{x'\circ i^{M}_{N}}^{\ast}_{M,m}
\leq C\bigl(r^{\eta}\vertiii{x'\circ i^{K}_{N}}^{\ast}_{K,k}+\frac{1}{r}\vertiii{x'}^{\ast}_{N,n}\bigr)
\end{gather*}
where $\vertiii{\cdot}^{\ast}$ denotes the dual norm of $\vertiii{\cdot}$ (see \cite[Section 4, Eq.\ (24), p.\ 577]{Dom1}).
By \cite[Remark 2, p.\ 6]{kruse2019_1} a Fr\'echet-Schwartz space $F$ satisfies $(DN)$ if and
only if the DFS-space $E\coloneq F_{b}'$ satisfies $(PA)$.

\begin{thm}[{\cite[Example 5.7 a), p.\ 2703]{kruse2018_5}, \cite[Theorem 20 (ii), p.\ 110]{kruse2019_2}, 
\cite[Corollary 18, p.\ 21]{kruse2019_1}}]\label{thm:surj_CR}
Let $K\subset\overline{\R}$ be a compact set. If
\begin{enumerate}
\item [a)] $E$ is a Fr\'echet space over $\C$, or if
\item [b)] $E\coloneq F_{b}'$ where $F$ is a Fr\'echet space over $\C$ satisfying $(DN)$, or if
\item [c)] $E$ is an ultrabornological PLS-space over $\C$ satisfying $(PA)$, 
\end{enumerate}
then $E$ is admissible, i.e.\
\[
\overline{\partial}\colon \mathcal{E}^{exp}(\overline{\C}\setminus K,E)\to\mathcal{E}^{exp}(\overline{\C}\setminus K,E)
\]
is surjective for every compact set $K\subset\overline{\R}$.
\end{thm}

This result can also be found in \cite[5.17 Theorem, p.\ 82]{ich} and \cite[5.24 Theorem, p.\ 85]{ich}. 
In the non-weighted case the Cauchy--Riemann operator
\begin{equation}\label{eq:CR_unweighted}
\overline{\partial}\colon\mathcal{C}^{\infty}(U,E)\to\mathcal{C}^{\infty}(U,E)
\end{equation}
is surjective for every open set $U\subset\C$ if $E=\C$ by \cite[Theorem 1.4.4, p.\ 12]{H3}. 
Furthermore, $\mathcal{O}(U)$ and $\mathcal{C}^{\infty}(U)$, 
both equipped with the topology of uniform convergence on compact subsets 
(of partial derivatives of any order in the latter case), 
are nuclear Fr\'echet spaces by \cite[Examples 5.18 (3), (4), p.\ 42]{meisevogt1997}, 
\cite[Examples 28.9 (1), (4), p.\ 349--350]{meisevogt1997} 
and we have the topological isomorphisms
\[
\mathcal{C}^{\infty}(U,E)\cong\mathcal{C}^{\infty}(U)\varepsilon E\cong\mathcal{C}^{\infty}(U)\widehat{\otimes}_{\pi}E
\]
plus
\[
\mathcal{O}(U,E)\cong \mathcal{O}(U)\varepsilon E\cong\mathcal{O}(U)\widehat{\otimes}_{\pi}E
\]
by \cite[Theorem 44.1, p.\ 449]{Treves} resp.\ \cite[16.7.5 Corollary, p.\ 366]{Jarchow} 
for open $U\subset\C$ and complete $E$. 
By \cite[Theorem 10.10, p.\ 240]{Kaballo} the $\overline{\partial}$-operator in \eqref{eq:CR_unweighted} is surjective 
if $E$ is a Fr\'echet space. If $E\coloneq F_{b}'$ where $F$ is a Fr\'echet space satisfying $(DN)$ 
or $E$ is an ultrabornological PLS-space satisfying $(PA)$, this holds due to \cite[2.6 Theorem, p.\ 174]{vogt1983} 
resp.\ \cite[Corollary 3.9, p.\ 1112]{D/L} as well because
$\mathcal{O}(U)=\operatorname{ker}\overline{\partial}$ in $\mathcal{C}^{\infty}(U)$ has the property $(\Omega)$
by \cite[Proposition 2.5 (b), p.\ 173]{vogt1983} for every open $U\subset\C$. 
In combination with \prettyref{thm:surj_CR} this means:

\begin{thm}[{\cite[5.25 Theorem, p.\ 98]{ich}}]\label{thm:examples_strictly_admiss}
If
\begin{enumerate}
\item [a)] $E$ is a Fr\'echet space over $\C$, or if
\item [b)] $E\coloneq F_{b}'$ where $F$ is a Fr\'echet space over $\C$ satisfying $(DN)$, or if
\item [c)] $E$ is an ultrabornological PLS-space over $\C$ satisfying $(PA)$, 
\end{enumerate}
then $E$ is strictly admissible.
\end{thm}

To close this section we provide some examples of ultrabornological PLS-spaces satisfying $(PA)$ and 
spaces of the form $E\coloneq F_{b}'$ where $F$ is a Fr\'echet space satisfying $(DN)$.
Most of them are already contained in \cite[Corollary 4.8, p.\ 1116]{D/L} (see \cite{Dom1} and \cite{vogt1983} as well).

\begin{exa}[{\cite[Example 3, p.\ 7]{kruse2019_1}}]\label{ex:PLS_PA_DF_DN}
a) The following spaces are ultrabornological PLS-spaces with property $(PA)$ 
and also strong duals of a Fr\'echet space satisfying $(DN)$, hence are strictly admissible:
\begin{itemize}
 \item the strong dual of a power series space of inifinite type $\Lambda_{\infty}(\alpha)_{b}'$,
 \item the strong dual of any space of holomorphic functions $\mathcal{O}(U)_{b}'$ 
 where $U$ is a Stein manifold with the strong Liouville property (for instance, for $U=\C^{d}$),
 \item the space of germs of holomorphic functions $\mathcal{O}(K)$ where $K$ is a completely pluripolar 
 compact subset of a Stein manifold (for instance $K$ consists of one point),
 \item the space of tempered distributions $\mathcal{S}(\R^{d})_{b}'$ and 
 the space of Fourier ultra-hyperfunctions $\mathcal{P}'_{\ast\ast}$ (with the strong dual topology),
 \item the weighted distribution spaces $(K\{pM\})_{b}'$ of Gelfand and Shilov if the weight $M$ satisfies
 \[
  \sup_{|y|\leq 1}M(x+y)\leq C\inf_{|y|\leq 1}M(x+y),\quad x\in\R^{d},
 \]
 \item $\mathcal{D}(K)_{b}'$ for any compact set $K\subset\R^{d}$ with non-empty interior,
 \item $\mathcal{C}^{\infty}(\overline{U})_{b}'$ for any non-empty open bounded set $U\subset\R^{d}$ with $\mathcal{C}^{1}$-boundary.
\end{itemize}
b) The following spaces are ultrabornological PLS-spaces with property $(PA)$:
\begin{itemize}
 \item an arbitrary Fr\'echet-Schwartz space,
 \item a PLS-type power series space $\Lambda_{r,s}(\alpha,\beta)$ whenever $s=\infty$ 
 or $\Lambda_{r,s}(\alpha,\beta)$ is a Fr\'echet space,
 \item the spaces of distributions $\mathcal{D}(U)_{b}'$ and ultradistributions of Beurling type $\mathcal{D}_{(\omega)}(U)_{b}'$ 
 for any open set $U\subset\R^{d}$,
 \item the kernel of any linear partial differential operator with constant coefficients in $\mathcal{D}(U)_{b}'$ 
 or in $\mathcal{D}_{(\omega)}(U)_{b}'$ when $U\subset\R^{d}$ is open and convex,
 \item the space $L_{b}(X,Y)$ where $X$ has $(DN)$, $Y$ has $(\Omega)$ and both are nuclear Fr\'echet spaces. 
 In particular, $L_{b}(\Lambda_{\infty}(\alpha),\Lambda_{\infty}(\beta))$ if both spaces are nuclear.
\end{itemize}
c) The following spaces are strong duals of a Fr\'echet space satisfying $(DN)$, hence are strictly admissible:
\begin{itemize}
 \item the strong dual $F_{b}'$ of any Banach space $F$,
 \item the strong dual $\lambda^{2}(A)_{b}'$ of the K\"othe space $\lambda^{2}(A)$ 
 with a K\"othe matrix $A=(a_{j,k})_{j,k\in\N_{0}}$ satisfying 
 \[
  \exists\;p\in\N_{0}\;\forall\;k\in\N_{0}\;\exists\;n\in\N_{0},C>0:\;a_{j,k}^{2}\leq Ca_{j,p}a_{j,n}.
 \]
\end{itemize}
\end{exa} 

\begin{exa}[{\cite[Corollary 4.9, p.\ 1117]{D/L}}]\label{ex:PLS_non_PA}
a) The following ultrabornological PLS-spaces do not have $(PA)$:
\begin{itemize}
 \item the strong dual of power series space of finite type $\Lambda_{0}(\alpha)_{b}'$,
 \item the space of ultradifferentiable functions of Roumieu type $\mathcal{E}_{\{\omega\}}(U)$ 
 where $\omega$ is a non-quasianalytic weight and $U\subset\R^{d}$ is an arbitrary open set,
 \item the strong dual of any space of holomorphic functions $\mathcal{O}(U)_{b}'$ 
 where $U$ is a Stein manifold which does not have the strong Liouville property 
 (for instance, $U=\D^{d}$ the polydisc, $U=\mathbb{B}_{d}$ the unit ball etc.),
 \item the space of germs of holomorphic functions $\mathcal{O}(K)$ where $K$ is compact 
 and not completely pluripolar (for instance, $K=\overline{\D}^{d}$ or $K=\overline{\mathbb{B}}_{d}$),
 \item the space of distributions $\mathcal{E}'(U)$ and ultradistributions 
 $\mathcal{E}_{(\omega)}'(U)$ (with the strong dual topology) with compact support for $U\subset\R^{d}$ open, 
 \item the space of real analytic functions $\mathscr{A}(U)$ for any open set $U\subset\R^{d}$.
\end{itemize}
b) For the following LFS-spaces $E$ the map \eqref{eq:CR_unweighted} is not surjective and hence $E$ is not strictly admissible:
\begin{itemize}
 \item the space of test functions $\mathcal{D}(U)$ (with its inductive limit topology) where $U\subset\R^{d}$ is any open set,
 \item the space of test functions for ultradistributions $\mathcal{D}_{(\omega)}(U)$, 
 the space of ultradistributions of Roumieu type with compact support $\mathcal{E}_{\{\omega\}}(U)_{b}'$ 
 where $\omega$ is a non-quasianalytic weight, $U\subset\R^{d}$ is any open set,
 \item the strong dual $\mathscr{A}(U)_{b}'$ for any open set $U\subset\R^{d}$.
\end{itemize}
\end{exa}

\section{Duality method}
\label{sect:duality_method}

In this section we present the main results of \cite[Chapter 6]{ich}. 
We construct $E$-valued Fourier hyperfunctions in one variable as the sheaf 
generated by equivalence classes of compactly supported $E$-valued $\mathcal{P}_{\ast}$-functionals and 
show that they form a flabby sheaf under the condition that $E$ is strictly admissible and sequentially complete. 
This construction relies on the Silva--K\"othe--Grothendieck duality \prettyref{thm:duality} 
and the method, which goes back to Martineau \cite{martineau1961}, is sometimes called \emph{duality method} 
(see \cite{D/L} and \cite{Ito1998}). Furthermore, a description of $E$-valued Fourier hyperfunctions 
as boundary values of slowly increasing holomorphic functions is provided 
and finally the necessity of the conditions that are used for the construction of 
vector-valued Fourier hyperfunctions will be examined.

\begin{defn}[{Fourier hyperfunctions}]
For an non-empty open set $\Omega\subset\overline{\R}$ and $\C$-lcHs $E$ 
we define the space of $E$-valued \emph{Fourier hyperfunctions} on $\Omega$ by
\[
\mathcal{R}(\Omega,E)\coloneq L(\mathcal{P}_{\ast}(\overline{\Omega}),E)/L(\mathcal{P}_{\ast}(\partial\Omega),E)
\]
and $\mathcal{R}(\varnothing,E)\coloneq 0$. 
For $T\in L(\mathcal{P}_{\ast}(\overline{\Omega}),E)$ we denote by $[T]$ the corresponding element of $\mathcal{R}(\Omega,E)$. 
If the set $\Omega$ is equipped with an index, then we sometimes do the same with the corresponding equivalence class 
in order to distinguish between different classes. Further, we use the notation $\mathcal{R}(\Omega)\coloneq\mathcal{R}(\Omega,\C)$.
\end{defn}

We observe that $L(\mathcal{P}_{\ast}(\varnothing),E)=L(0,E)=0$
and hence $\mathcal{R}(\overline{\R},E)=L(\mathcal{P}_{\ast}(\overline{\R}),E)$ 
(more precisely, we identify $L(\mathcal{P}_{\ast}(\overline{\R}),E)$ and $\{\{T\}\;|\;T\in L(\mathcal{P}_{\ast}(\overline{\R}),E)\}$). 
Thus there is a reasonable locally convex Hausdorff topology on $\mathcal{R}(\overline{\R},E)$.
For $\Omega\neq\overline{\R}$ there is no reasonable locally convex Hausdorff topology on $\mathcal{R}(\Omega,E)$ by 
\cite[3.10 Bemerkung, p.\ 41--42]{J}.

Let us first take a look at the scalar case. Let $\Omega\subset\Omega_{1}\subset\overline{\R}$ be open. 
It is straightforward to prove that the 
canonical injective (by \prettyref{prop:traeger} a)) linear map
\[
I\colon \mathcal{P}_{\ast}(\overline{\Omega})'/\mathcal{P}_{\ast}(\partial\Omega)' \to
\mathcal{P}_{\ast}(\overline{\Omega}_{1})'/\mathcal{P}_{\ast}(\overline{\Omega}_{1}\setminus\Omega)'.
\]
is surjective (see \cite[p.\ 101--102]{ich}), thus an algebraic isomorphism. 
Therefore the restrictions and a sheaf structure may be defined on 
$\mathcal{R}_{\Omega_{1}}\coloneq\{\mathcal{R}_{\Omega}\;|\;\Omega\subset\Omega_{1}\;\text{open}\}$ 
like in \prettyref{def:restrictions_op_sheaf}. It is not known whether the corresponding map $I$ 
in the vector-valued case is always an algebraic isomorphism 
(see \prettyref{rem:itos_luecke_1}). But this holds if we additionally assume that
\[
\overline{\partial}\colon \mathcal{E}^{exp}(\overline{\C}\setminus K,E)\to \mathcal{E}^{exp}(\overline{\C}\setminus K,E)
\]
is surjective for any compact set $K\subset\overline{\R}$, i.e.\ that $E$ is \emph{admissible}. 
Let us turn to the already indicated statement whose counterpart for hyperfunctions is given in \cite[Lemma 6.2, p.\ 1122]{D/L}.

\begin{lem}\label{lem:I_isom_op}
Let $\Omega_{2}\subset\Omega_{1}\subset\overline{\R}$ be open, $\Omega_{2}\neq\varnothing$, and $E$ sequentially complete 
and admissible. Then the canonical map
\begin{gather*}
I\colon L(\mathcal{P}_{\ast}(\overline{\Omega}_{2}), E)/L(\mathcal{P}_{\ast}(\partial\Omega_{2}), E) \to
L(\mathcal{P}_{\ast}(\overline{\Omega}_{1}), E)/L(\mathcal{P}_{\ast}(\overline{\Omega}_{1}\setminus\Omega_{2}), E),\\
[T]_{2}\mapsto[T],
\end{gather*}
is an algebraic isomorphism.
\end{lem}
\begin{proof}
This map is well-defined, in particular, independent of the choice of the representative 
since $\mathcal{P}_{\ast}(\overline{\Omega}_{1})$ is continuously and densely embedded in $\mathcal{P}_{\ast}(\overline{\Omega}_{2})$ 
(see the remark right above \prettyref{prop:traeger}) and thus the embedding of 
$L(\mathcal{P}_{\ast}(\overline{\Omega}_{2}), E)$ into $L(\mathcal{P}_{\ast}(\overline{\Omega}_{1}),E)$ 
is defined as well as the map of $L(\mathcal{P}_{\ast}(\partial\Omega_{2}), E)$ 
into $L(\mathcal{P}_{\ast}(\overline{\Omega}_{1}\setminus\Omega_{2}), E)$ in this manner.

If $\R\subset\Omega_{2}$, then $\overline{\Omega}_{2}=\overline{\Omega}_{1}=\overline{\R}$ and 
therefore $\overline{\Omega}_{1}\setminus\Omega_{2}=\partial\Omega_{2}$. Hence the statement is obviously true. 

Now, let $\R\not\subset\Omega_{2}$. Let $T\in L(\mathcal{P}_{\ast}(\overline{\Omega}_{2}), E)$ 
with $[T]=0$. Then we get by \prettyref{prop:traeger} a)
\[
T\in L(\mathcal{P}_{\ast}(\overline{\Omega}_{2}), E)
 \cap L(\mathcal{P}_{\ast}(\overline{\Omega}_{1}\setminus\Omega_{2}), E)
 = L(\mathcal{P}_{\ast}(\overline{\Omega}_{2}\cap(\overline{\Omega}_{1}\setminus\Omega_{2})),E)
 = L(\mathcal{P}_{\ast}(\partial\Omega_{2}), E)
\]
and thus $[T]_{2}=0$, implying the injectivity of $I$.

The surjectivity of $I$ is equivalent to the surjectivity of the map
\[
I_{0}\colon L(\mathcal{P}_{\ast}(\overline{\Omega}_{1}\setminus\Omega_{2}), E) \times L(\mathcal{P}_{\ast}(\overline{\Omega}_{2}), E)\to
            L(\mathcal{P}_{\ast}(\overline{\Omega}_{1}), E),\; (T_{1},T_{2})\mapsto T_{1}+T_{2}.
\]
By \prettyref{thm:duality} the surjectivity of $I_{0}$ is equivalent to the surjectivity of
\begin{gather*}
I_{1}\colon \mathcal{O}^{exp}(\overline{\C}\setminus( \overline{\Omega}_{1}\setminus\Omega_{2}), E)/\mathcal{O}^{exp}(\overline{\C}, E)
\times \mathcal{O}^{exp}(\overline{\C}\setminus \overline{\Omega}_{2}, E)/\mathcal{O}^{exp}(\overline{\C}, E)\\ 
\to \mathcal{O}^{exp}(\overline{\C}\setminus \overline{\Omega}_{1}, E)/\mathcal{O}^{exp}(\overline{\C}, E), \;
(f_{1},f_{2})\mapsto f_{1}+f_{2},
\end{gather*}
and thus to the surjectivity of
\[
I_{2}\colon \mathcal{O}^{exp}(\overline{\C}\setminus (\overline{\Omega}_{1}\setminus\Omega_{2}), E)
\times \mathcal{O}^{exp}(\overline{\C}\setminus \overline{\Omega}_{2}, E)
\to \mathcal{O}^{exp}(\overline{\C}\setminus \overline{\Omega}_{1}, E), \; (f_{1},f_{2})\mapsto f_{1}+f_{2}.
\]
The proof is now done in several steps, beginning with the construction of a cut-off function. 
We restrict to the case that $\pm\infty\in\overline{\Omega}_{1}$, $-\infty\in\Omega_{2}$ and $\infty\notin\overline{\Omega}_{2}$. 
For the similar treatment of the other cases we refer to the proof of \cite[6.2 Lemma, p.\ 103]{ich}.

(i) There is $x_{0}\in\R$ such that $[x_{0},\infty]\subset \overline{\Omega}_{2}^{C}$ since $\overline{\Omega}_{2}^{C}\subset\overline{\R}$ is open 
and $\infty\notin\Omega_{2}$. Further, there is $\widetilde{x}_{1}\in\R$ such that 
$[-\infty, \widetilde{x}_{1}]\subset \Omega_{2}$, since $\Omega_{2}$ is open and $-\infty\in\Omega_{2}$,
and thus $[-\infty, \widetilde{x}_{1}]\subset (\overline{\Omega}_{1}\setminus\Omega_{2})^{C}$. We define the sets
\[
F_{0}\coloneq (\R\setminus \Omega_{2}) \cup ([x_{0}+2,\infty)\times\left[-1,1\right])\;\;\text{and}\;\;
F_{1}\coloneq (\R\cap\overline{\Omega}_{2}) \cup ((-\infty,\widetilde{x}_{1}-2]\times[-1,1]).
\]
The sets $F_{0}$ and $F_{1}$ are non-empty and closed in $\R^{2}$, $F_{0}\cap\R=\R\setminus\Omega_{2}$, 
$F_{1}\cap\R=\overline{\Omega}_{2}\cap\R$ and $F_{0}\cap F_{1}=\partial\Omega_{2}\cap\R$. 
By \cite[Corollary 1.4.11, p.\ 31]{H1} there exists 
$\varphi\in\mathcal{C}^{\infty}((F_{0}\cap F_{1})^{C})=\mathcal{C}^{\infty}(\R^{2}\setminus\partial\Omega_{2})$, 
$0\leq\varphi\leq 1$, such that $\varphi=0$ on $V_{0}$ and $\varphi=1$ on $V_{1}$ 
where $V_{0}$, $V_{1}\subset\R^{2}$ are open and
\[
V_{0}\supset F_{0}\setminus (F_{0}\cap F_{1})=F_{0}\setminus \partial\Omega_{2}\supset(\R\setminus\overline{\Omega}_{2})
\;\;\text{and}\;\;
V_{1}\supset F_{1}\setminus (F_{0}\cap F_{1})=F_{1}\setminus \partial\Omega_{2}\supset(\R\cap\Omega_{2}).
\]
\begin{center}
\includegraphics[scale=0.85]{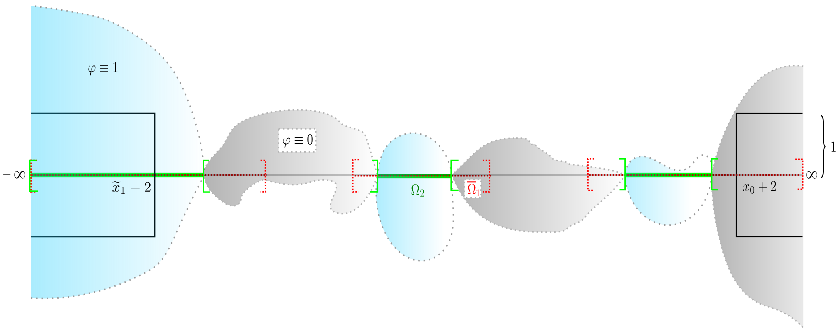}
\captionsetup{type=figure}
\caption{case: $\pm\infty\in\overline{\Omega}_{1}$, $-\infty\in\Omega_{2}$, $\infty\notin\overline{\Omega}_{2}$}
\end{center}
Furthermore,
\begin{equation}\label{eq:cut-off}
|\varphi^{(k)}(z;y^{1},\cdots,y^{k})|\leq C^{k}|y^{1}|\cdots|y^{k}|\frac{\d(z)^{-k}}{d_{1}\cdots d_{k}}
\end{equation}
for all $z\in\R^{2}\setminus \partial \Omega_{2}$ and all $y^{i}\in\R^{2}$, $1\leq i\leq k$, $k\in\N$, 
where $\varphi^{(k)}$ denotes the differential of order $k$ of $\varphi$, $C>0$ is a constant independent of $z$, $y^{i}$ and $k$,
\[
\d(z)\coloneq\max\bigl(\d(z,F_{0}),\d(z,F_{1})\bigr)=\max\bigl(\min_{x\in F_{0}}|z-x|,\min_{x\in F_{1}}|z-x| \bigr)
\]
and $(d_{n})_{n\in\N}$ is any decreasing sequence with $\sum^{\infty}_{n=1}d_{n}=1$, e.g.\ $d_{n}\coloneq (\tfrac{1}{2})^{n}$. 
We observe that for $\beta=(\beta_{1},\beta_{2})\in\N^{2}_{0}$ the relation
\[
\partial^{\beta}\varphi(z)=\varphi^{(|\beta|)}(z;\underbrace{e_{1},\cdots,e_{1}}_{\beta_{1}-\text{times}},
\underbrace{e_{2},\cdots,e_{2}}_{\beta_{2}-\text{times}})
\]
holds between the differential of order $|\beta|$ and the $\beta$th partial derivative where $e_{j}$, $j=1,2$, 
is the $j$th unit vector in $\R^{2}$.
Thus we obtain from \eqref{eq:cut-off} the estimate
\begin{equation}\label{eq:cut-off_partial}
|\partial^{\beta}\varphi(z)|\leq C^{|\beta|}\frac{\d(z)^{-|\beta|}}{d_{1}\cdots d_{|\beta|}},
\quad z\in\R^{2}\setminus \partial \Omega_{2},\,\beta\in\N_{0}^{2},
\end{equation}
where we set $d_{1}\cdots d_{|0|}\coloneq 1$ which is consistent with $|\partial^{0}\varphi|=\varphi\leq 1$.
Let us take a closer look at the right-hand side of this inequality. For $z\in\R^{2}\setminus\partial\Omega_{2}$ 
there is $z_{i}\in F_{i}$ such that $\d(z,F_{i})=|z-z_{i}|$, $i=0,1,$ since $F_{0}$ and $F_{1}$ are closed. 
Let $n\in\N$, $n\geq 2$. We claim that
\[
\d(z)\geq \frac{1}{n}, \quad \;z\in S_{n}(\partial\Omega_{2}).
\]
Let $z\in S_{n}(\partial\Omega_{2})$.

\emph{case} $z_{0},\;z_{1}\in\R$:
Let us assume that $\d(z)<\tfrac{1}{n}$. The definition of the set $S_{n}(\partial\Omega_{2})$ 
implies $z_{i}\notin\partial\Omega_{2}\cap\R$, $i=0,1$. Thus we get by the definition of the sets $F_{i}$ 
that $z_{0}\in\R\setminus\overline{\Omega}_{2}$ and 
$z_{1}\in\R\cap\Omega_{2}$, in particular, $z_{0}\neq z_{1}$. W.l.o.g.\ $z_{0}<z_{1}$. 
Then $O_{0}\coloneq (z_{0},z_{1})\cap(\R\setminus\overline{\Omega}_{2})$ and
$O_{1}\coloneq (z_{0},z_{1})\cap(\R\cap\Omega_{2})$ are disjoint, open sets in $\R$. 
Assume that there is no $\widetilde{z}\in\partial\Omega_{2}\cap\R$ with $z_{0}<\widetilde{z}<z_{1}$.
Due to this assumption, we obtain
\[
  O_{0}\cup O_{1}
=(z_{0},z_{1})\cap\bigl[(\R\setminus\overline{\Omega}_{2})\cup(\R\cap\Omega_{2})\bigr]
=(z_{0},z_{1})\cap(\R\setminus\partial\Omega_{2})=(z_{0},z_{1})
\]
and hence, as $(z_{0},z_{1})$ is connected, $(z_{0},z_{1})\subset O_{0}$ or $(z_{0},z_{1})\subset O_{1}$. 
If $(z_{0},z_{1})\subset O_{0}$, we get $z_{1}\notin\R\cap\Omega_{2}$, 
and if $(z_{0},z_{1})\subset O_{1}$, we have $z_{0}\notin \R\setminus\overline{\Omega}_{2}$, which is a contradiction. 
So there must be a $\widetilde{z}\in\partial\Omega_{2}\cap\R$ with $z_{0}<\widetilde{z}<z_{1}$. 
The convexity of $\D_{\d(z)}(z)$ implies $\widetilde{z}\in(z_{0},z_{1})\subset \D_{\d(z)}(z)$, 
but then the following is valid
\[
     \frac{1}{n}<|z-\widetilde{z}| 
\leq \max(|z-z_{0}|,|z-z_{1}|)
=    \d(z)<\frac{1}{n},
\]
which is again a contradiction.

\emph{case} $(z_{0}\notin\R$, $z_{1}\in\R)$ or $(z_{0}\in\R$, $z_{1}\notin\R)$:
We only consider the first case, the latter one is analogous. 
We have $z_{1}<x_{0}$ and $\re(z_{0})\geq x_{0}+2$. Therefore, we get
\[
|z_{1}-\re(z_{0})|\geq |x_{0}-(x_{0}+2)|=2.
\]
If $|z-z_{0}|<\tfrac{1}{n}$, we obtain by the estimate above
\begin{align*}
 \d(z)
&\geq|z-z_{1}|\geq|\re(z)-z_{1}|\geq|z_{1}-\re(z_{0})|-|\re(z_{0})-\re(z)|>|z_{1}-\re(z_{0})|-\frac{1}{n}\\ 
&\geq 2-\frac{1}{n}\geq\frac{1}{n}.
\end{align*}

\emph{case} $z_{i}\notin\R$, $i=0,1$: 
If $|z-z_{0}|<\tfrac{1}{n}$, we have 
\[
\re(z_{1})\leq \widetilde{x}_{1}-2<\widetilde{x}_{1}<x_{0}<x_{0}+2-\frac{1}{n}\leq\re(z)
\]
and thus we get
\[
\d(z)\geq|z-z_{1}|\geq|\re(z)-\re(z_{1})|\geq 4-\frac{1}{n}>\frac{1}{n}.
\]

Hence the claim is proved and via \eqref{eq:cut-off_partial} we obtain
\begin{equation}\label{eq:cut-off_partial_dist}
|\partial^{\beta}\varphi(z)|\leq C^{|\beta|}\frac{n^{|\beta|}}{d_{1}\cdots d_{|\beta|}},\quad z\in S_{n}(\partial \Omega_{2}).
\end{equation}

(ii) Let $f\in \mathcal{O}^{exp}(\overline{\C}\setminus \overline{\Omega}_{1}, E)$. 
Due to the choice of $\varphi$, the function $\overline{\partial}(\varphi f)$ may be regarded 
as an element of $\mathcal{C}^{\infty}(\R^{2}\setminus\partial\Omega_{2},E)$ 
by $\mathcal{C}^{\infty}$-extension via $\overline{\partial}(\varphi f)\coloneq 0$ on $\R\setminus\partial\Omega_{2}$. 
Furthermore,
\[
\overline{\partial}(\varphi f)(z)=
\begin{cases}
0\; &,\;z\in V_{0}\cup V_{1},\\
(\overline{\partial}\varphi)(z)f(z) &,\;\text{else},
\end{cases}
\]
is valid.

Let $n\in\N$, $n\geq 2$, $m\in\N_{0}$ and $\alpha\in\mathfrak{A}$. 
We define the set $S(n)\coloneq S_{n}(\partial\Omega_{2})\setminus(V_{0}\cup V_{1})$. 
By applying the Leibniz rule (see e.g.\ \cite[Proposition 3.13, p.\ 242]{kruse2018_2}), we obtain
\begin{flalign}\label{eq:estimate_molified_f}
&\hspace{0.35cm}|\overline{\partial}(\varphi f)|_{\partial\Omega_{2},n,m,\alpha}\nonumber\\
&=\sup_{\substack{z\in S_{n}(\partial\Omega_{2})\\ \beta\in\N^{2}_{0},|\beta|\leq m}}
  p_{\alpha}(\partial^{\beta}\overline{\partial}(\varphi f)(z))\e^{-\frac{1}{n}|\re(z)|}\nonumber\\
&\leq\sup_{\substack{z\in S_{n}(\partial\Omega_{2})\setminus(V_{0}\cup V_{1})\\ \beta\in\N^{2}_{0},|\beta|\leq m}}
     \sum_{\gamma\leq\beta}\dbinom{\beta}{\gamma}|\partial^{\beta-\gamma}\overline{\partial}\varphi(z)|p_{\alpha}(\partial^{\gamma}f(z))
     \e^{-\frac{1}{n}|\re(z)|}\nonumber\\
&\underset{\mathclap{\eqref{eq:real.compl.part.deriv.1}}}{\leq}\;(m!)^{2}\sum_{|\gamma|\leq m+1}\sup_{z\in S(n)}|\partial^{\gamma}\varphi(z)|
     \underbrace{\sup_{\substack{z\in S(n)\\ \beta\in\N^{2}_{0},|\beta|\leq m}}
     p_{\alpha}(\partial_{\C}^{|\beta|}f(z))\e^{-\frac{1}{n}|\re(z)|}}_{\eqqcolon C(f)}\nonumber\\
&\underset{\mathclap{\eqref{eq:cut-off_partial_dist}}}{\leq}\;(m!)^{2}C(f)
 \sum_{|\gamma|\leq m+1}C^{|\gamma|}\frac{n^{|\gamma|}}{d_{1}\cdots d_{|\gamma|}}\nonumber\\
&\leq (m!)^{2}\frac{[\max(C,1)]^{m+1}}{d_{1}\cdots d_{m+1}}C(f)\sum_{|\gamma|\leq m+1}n^{|\gamma|}
\end{flalign}
where we used the properties of $(d_{j})$, which imply $0<d_{j}<1$ for all $j\in\N$, in the last estimate.

Now, we have to take a closer look at $C(f)$. We decompose the set $S(n)$ in the following manner:
\[
 S(n)
=\underbrace{[S(n)\cap\{z\in\C\;|\;|\im(z)|>\tfrac{1}{2n}\}]}_{\subset S_{2n}(\overline{\Omega}_{1})}
\cup\underbrace{[S(n)\setminus\{z\in\C\;|\;|\im(z)|>\tfrac{1}{2n}\}]}_{\eqqcolon M}
\]
\begin{center}
\includegraphics[scale=0.85]{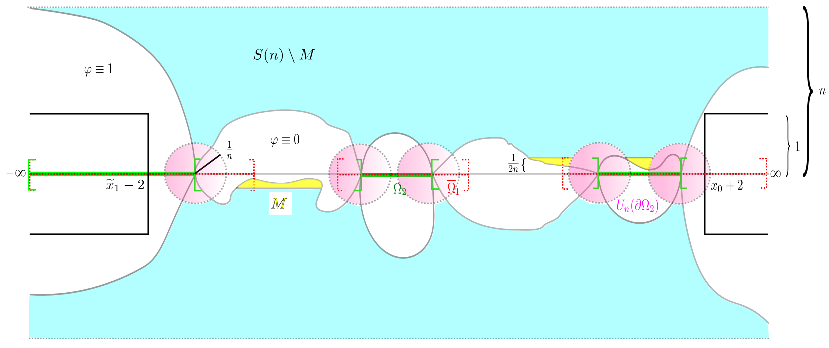}
\captionsetup{type=figure}
\caption{case: $\pm\infty\in\overline{\Omega}_{1}$, $-\infty\in\Omega_{2}$, $\infty\notin\overline{\Omega}_{2}$}
\end{center}
Due to \prettyref{prop:Cauchy_estimates}, we get for $r\coloneq\tfrac{1}{2}(\tfrac{1}{2n}-\tfrac{1}{3n})$
\begin{align}\label{eq:estimate_C(f)} 
C(f)&\leq \sup_{\substack{z\in S_{2n}(\overline{\Omega}_{1})\\ \beta\in\N^{2}_{0},|\beta|\leq m}}
 p_{\alpha}(\partial_{\C}^{|\beta|}f(z))\e^{-\frac{1}{n}|\re(z)|}+ \sup_{\substack{z\in M\\ \beta\in\N^{2}_{0},|\beta|\leq m}}
 p_{\alpha}(\partial_{\C}^{|\beta|}f(z))\e^{-\frac{1}{n}|\re(z)|}\nonumber\\
&\leq \e^{\frac{r}{n}}\frac{m!}{r^{m}}|f|_{\overline{\Omega}_{1},3n,\alpha}
 +\sup_{\substack{z\in M\\ \beta\in\N^{2}_{0},|\beta|\leq m}}
 p_{\alpha}(\partial_{\C}^{|\beta|}f(z))\e^{-\frac{1}{n}|\re(z)|}.
\end{align}
Let us turn our attention to the set $M$. First, we observe that
\[
\R\subset\bigl[\underbrace{V_{0}\cup V_{1}}_{\supset \R\setminus \partial\Omega_{2}} 
               \cup \bigcup_{x\in \partial\Omega_{2}\cap\R}\D_{1/n}(x)\bigr]\eqqcolon V.
\]
$V\subset\R^{2}$ is open as the union of open sets and so we get by the definition of $M$ that
\begin{equation}\label{eq:closure_M}
\overline{M}\subset\overline{V^{C}}=V^{C}\subset(\R^{2}\setminus\R).
\end{equation}
We claim that $M$ is bounded. As $|\im(z)|\leq \tfrac{1}{2n}$ for every $z\in M$, 
it suffices to prove that there is $C_{1}>0$ such that $|\re(z)|\leq C_{1}$ for every $z\in M$. 
The choice of the sets $F_{0}$ and $F_{1}$ gives $C_{1}\coloneq\max(|\widetilde{x}_{1}-2|,|x_{0}+2|)$.
Hence $\overline{M}$ is compact and we have by \eqref{eq:closure_M} and the continuity of $\partial_{\C}^{|\beta|}f$ 
on $\R^{2}\setminus\R$ for all $\beta\in\N^{2}_{0}$ that
\[
    \sup_{\substack{z\in M\\\beta\in\N^{2}_{0},|\beta|\leq m}}p_{\alpha}(\partial_{\C}^{|\beta|}f(z))\e^{-\frac{1}{n}|\re(z)|}
\leq\sup_{\substack{z\in \overline{M}\\ \beta\in\N^{2}_{0},|\beta|\leq m}}
    p_{\alpha}(\partial_{\C}^{|\beta|}f(z))\e^{-\frac{1}{n}|\re(z)|}
<\infty .
\]
Thus $C(f)<\infty$ by \eqref{eq:estimate_C(f)} and therefore $|\overline{\partial}(\varphi f)|_{\partial\Omega_{2},n,m,\alpha}<\infty$ 
for all $n\in\N$, $n\geq 2$, $m\in\N_{0}$ and $\alpha\in\mathfrak{A}$ by \eqref{eq:estimate_molified_f}, implying
$\overline{\partial}(\varphi f)\in\mathcal{E}^{exp}(\overline{\C}\setminus \partial\Omega_{2},E)$. 
As $E$ is admissible, there is $g\in\mathcal{E}^{exp}(\overline{\C}\setminus \partial\Omega_{2},E)$ 
such that
\begin{equation}\label{eq:solution_CR_molified}
\overline{\partial}g=\overline{\partial}(\varphi f).
\end{equation}

(iii) We set $f_{1}\coloneq (1-\varphi)f+g$ and $f_{2}\coloneq \varphi f-g.$ It remains to be proved that
$f_{1}\in\mathcal{O}^{exp}(\overline{\C}\setminus(\overline{\Omega}_{1}\setminus\Omega_{2}),E)$ and
$f_{2}\in\mathcal{O}^{exp}(\overline{\C}\setminus \overline{\Omega}_{2}, E)$. The proof is quite similar to part (ii).
$f_{1}$ is defined on $\C\setminus(\overline{\Omega}_{1}\setminus\Omega_{2})$ (by setting $(1-\varphi)f\coloneq 0$ on $\Omega_{2}\cap\R$) 
and can be regarded as an element of $\mathcal{O}(\C\setminus(\overline{\Omega}_{1}\setminus\Omega_{2}),E)$ 
due to \eqref{eq:solution_CR_molified}. 

Let $n\in\N$, $n\geq 2$, and set $S(n)\coloneq S_{n}(\overline{\Omega}_{1}\setminus\Omega_{2})\setminus V_{1}$. 
Remark that $S_{n}(\overline{\Omega}_{1}\setminus\Omega_{2})\subset S_{n}(\partial\Omega_{2})$ and
\[
S(n)
=\underbrace{[S(n)\cap\{z\in\C\;|\;|\im(z)|>\tfrac{1}{2n}\}]}_{\subset S_{2n}(\overline{\Omega}_{1})}
\cup\underbrace{[S(n)\setminus\{z\in\C\;|\;|\im(z)|>\tfrac{1}{2n}\}]}_{\eqqcolon M}.
\]
\begin{center}
\includegraphics[scale=0.85]{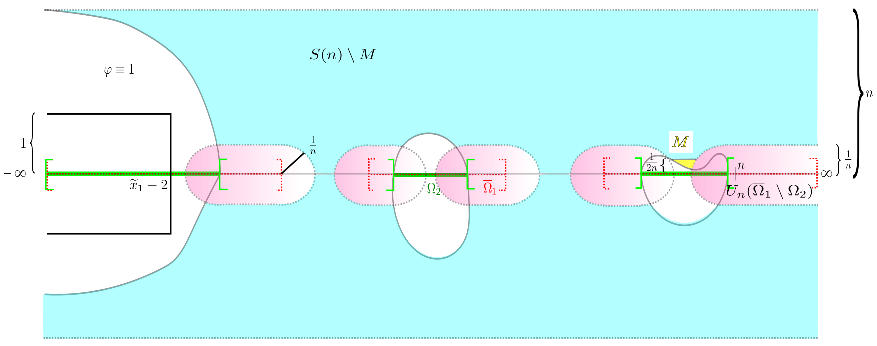}
\captionsetup{type=figure}
\caption{case: $\pm\infty\in\overline{\Omega}_{1}$, $-\infty\in\Omega_{2}$, $\infty\notin\overline{\Omega}_{2}$}
\end{center}
For $\alpha\in\mathfrak{A}$ we have by the choice of $\varphi$
\begin{align}\label{eq:estimate_f_1}
  |f_{1}|_{\overline{\Omega}_{1}\setminus\Omega_{2},n,\alpha}
&=\sup_{z\in S_{n}(\overline{\Omega}_{1}\setminus\Omega_{2})}p_{\alpha}(f_{1}(z))\e^{-\frac{1}{n}|\re(z)|}\nonumber\\
&\leq \underbrace{\sup_{z\in S_{n}(\partial\Omega_{2})}p_{\alpha}(g(z))\e^{-\frac{1}{n}|\re(z)|}}_{=|g|_{\partial\Omega_{2},n,0,\alpha}}
 +\sup_{z\in S_{n}(\overline{\Omega}_{1}\setminus\Omega_{2})}p_{\alpha}((1-\varphi)f(z))\e^{-\frac{1}{n}|\re(z)|}\nonumber\\
&=|g|_{\partial\Omega_{2},n,0,\alpha}+\sup_{z\in S(n)}\underbrace{|1-\varphi(z)|}_{\leq 1}
  p_{\alpha}(f(z))\e^{-\frac{1}{n}|\re(z)|}\nonumber\\
&\leq |g|_{\partial\Omega_{2},n,0,\alpha}+\underbrace{\sup_{z\in S_{2n}(\overline{\Omega}_{1})}
  p_{\alpha}(f(z))\e^{-\frac{1}{n}|\re(z)|}}_{=|f|_{\overline{\Omega}_{1},2n,\alpha}}
 +\sup_{z\in M}p_{\alpha}(f(z))\e^{-\frac{1}{n}|\re(z)|}\nonumber\\
&=|g|_{\partial\Omega_{2},n,0,\alpha}+|f|_{\overline{\Omega}_{1},2n,\alpha}+\sup_{z\in M} p_{\alpha}(f(z))\e^{-\frac{1}{n}|\re(z)|}.
\end{align}
Again, we have to take a closer look at the set $M$. First, we observe that
\[
\R\cap\overline{\Omega}_{1}\subset \bigl[ \underbrace{ V_{1}}_{\supset \R\cap\Omega_{2}} \cup 
                                   \bigcup_{x\in (\overline{\Omega}_{1}\setminus\Omega_{2})\cap\R}\D_{\tfrac{1}{n}}(x)\bigr] \eqqcolon V.
\]
$V\subset\R^{2}$ is open and so we get by the definition of the set $M$
\[
\overline{M}\subset\overline{ V^{C}}=V^{C}\subset (\R^{2}\setminus\overline{\Omega}_{1}).
\]
Like in part (ii) the set $M$ is bounded because the real part is bounded with 
$|\re(z)|\leq \max(|\widetilde{x}_{1}-2|,n)$ for all $z\in M$.
Since $f$ is continuous on $\R^{2}\setminus\overline{\Omega}_{1}$, we obtain
\[
\sup_{z\in M}p_{\alpha}(f(z))\e^{-\frac{1}{n}|\re(z)|}<\infty.
\]
Thus we get $|f_{1}|_{\overline{\Omega}_{1}\setminus\Omega_{2},n,\alpha}<\infty$ for every $n\in\N$ and $\alpha\in\mathfrak{A}$ 
by \eqref{eq:estimate_f_1}, implying $f_{1}\in\mathcal{O}^{exp}(\overline{\C}\setminus(\overline{\Omega}_{1}\setminus\Omega_{2}),E)$.

$f_{2}$ is defined on $\C\setminus\overline{\Omega}_{2}$ (by setting $\varphi f\coloneq 0$ 
on $\overline{\Omega}_{1}\setminus\overline{\Omega}_{2}$) 
and can be regarded as an element of $\mathcal{O}(\C\setminus\overline{\Omega}_{2},E)$ due to \eqref{eq:solution_CR_molified}. 
Let $n\in\N$, $n\geq 2$. We set $S(n)\coloneq S_{n}(\overline{\Omega}_{2})\setminus V_{0}$ and remark that 
$S_{n}(\overline{\Omega}_{2})\subset S_{n}(\partial\Omega_{2})$ as well as
\[
  S(n)
=\underbrace{\bigl[S(n)\cap\{z\in\C\;|\;|\im(z)|>\tfrac{1}{2n}\}\bigr]}_{\subset S_{2n}(\overline{\Omega}_{1})}
 \cup\underbrace{\bigl[S(n)\setminus\{z\in\C\;|\;|\im(z)|>\tfrac{1}{2n}\}\bigr]}_{\eqqcolon M}.
\]
\begin{center}
\includegraphics[scale=0.85]{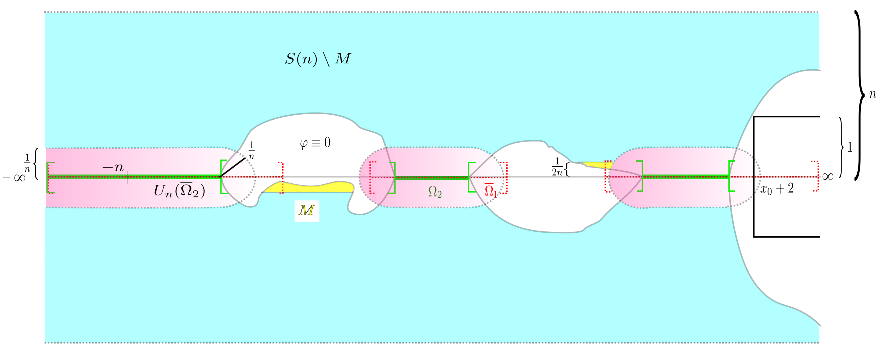}
\captionsetup{type=figure}
\caption{case: $\pm\infty\in\overline{\Omega}_{1}$, $-\infty\in\Omega_{2}$, $\infty\notin\overline{\Omega}_{2}$}
\end{center}
For $\alpha\in\mathfrak{A}$ we have by the choice of $\varphi$
\begin{align}\label{eq:estimate_f_2}
  |f_{2}|_{\overline{\Omega}_{2},n,\alpha}
&=\sup_{z\in S_{n}(\overline{\Omega}_{2})}p_{\alpha}(f_{2}(z))\e^{-\frac{1}{n}|\re(z)|}\nonumber\\
&\leq \underbrace{\sup_{z\in S_{n}(\partial\Omega_{2})}p_{\alpha}(g(z))\e^{-\frac{1}{n}|\re(z)|}}_{=|g|_{\partial\Omega_{2},n,0,\alpha}}
 +\sup_{z\in S_{n}(\overline{\Omega}_{2})}p_{\alpha}(\varphi f(z))\e^{-\frac{1}{n}|\re(z)|}\nonumber\\
&=|g|_{\partial\Omega_{2},n,0,\alpha}+\sup_{z\in S(n)}\underbrace{|\varphi(z)|}_{\leq 1}
  p_{\alpha}(f(z))\e^{-\frac{1}{n}|\re(z)|}\nonumber\\
&\leq |g|_{\partial\Omega_{2},n,0,\alpha}+\underbrace{\sup_{z\in S_{2n}(\overline{\Omega}_{1})}
  p_{\alpha}(f(z))\e^{-\frac{1}{n}|\re(z)|}}_{=|f|_{\overline{\Omega}_{1},2n,\alpha}}
 +\sup_{z\in M}p_{\alpha}(f(z))\e^{-\frac{1}{n}|\re(z)|}\nonumber\\
&=|g|_{\partial\Omega_{2},n,0,\alpha}+|f|_{\overline{\Omega}_{1},2n,\alpha}+\sup_{z\in M} p_{\alpha}(f(z))\e^{-\frac{1}{n}|\re(z)|}.
\end{align}
Again, we have to take a closer look at the set $M$ and observe that
\[
\R\subset \bigl[\underbrace{ V_{0}}_{\supset\R\setminus\overline{\Omega}_{2}}\cup
                \bigcup_{x\in \overline{\Omega}_{2}\cap\R}\D_{\tfrac{1}{n}}(x)\bigr] \eqqcolon V.
\]
$V\subset\R^{2}$ is open and so we get by the definition of the set $M$
\[
\overline{M}\subset\overline{ V^{C}}=V^{C}\subset(\R^{2}\setminus\R).
\]
Like before the set $M$ is bounded because the real part is bounded with $|\re(z)|\leq \max(|-n|,|x_{0}+2|)$ for all $z\in M$.
Again, we gain
\[
\sup_{z\in M}p_{\alpha}(f(z))\e^{-\frac{1}{n}|\re(z)|}<\infty
\]
and thus get $|f_{2}|_{\overline{\Omega}_{2},n,\alpha}<\infty$ for every $n\in\N$, $n\geq 2$, 
and $\alpha\in\mathfrak{A}$ by \eqref{eq:estimate_f_2}, 
implying $f_{2}\in\mathcal{O}^{exp}(\overline{\C}\setminus \overline{\Omega}_{2},E)$.
Obviously $f_{1}+f_{2}=f$, completing the proof by part (i).
\end{proof}

Ito (see \cite[p.\ 15, l.\ 16]{Ito2002}) states that \prettyref{lem:I_isom_op} 
is valid for any $\C$-lcHs E, but he does not prove that $I$ is surjective. 
Nevertheless, he states as an open problem (see \cite[Problem A, p.\ 17]{Ito2002}) 
whether for two compact sets $K_{1},K_{2}\subset \overline{\R}$ the map
\[
\mathcal{L}\colon L(\mathcal{P}_{\ast}(K_{1}),E)\times L(\mathcal{P}_{\ast}(K_{2}),E)\to L(\mathcal{P}_{\ast}(K_{1}\cup K_{2}),E),
\]
given by $\mathcal{L}(T_{1},T_{2})\coloneq T_{1}-T_{2}$, is surjective for non-Fr\'echet spaces $E$. 

\begin{rem}\label{rem:itos_luecke_1}
Let $\Omega_{2}\subset\Omega_{1}\subset\overline{\R}$ be open and $E$ a sequentially complete $\C$-lcHs. 
Then the following assertions are equivalent:
\begin{enumerate}
\item[a)] The canonical map 
\[
I\colon  L(\mathcal{P}_{\ast}(\overline{\Omega}_{2}), E)/L(\mathcal{P}_{\ast}(\partial\Omega_{2}), E) \to
L(\mathcal{P}_{\ast}(\overline{\Omega}_{1}), E)/L(\mathcal{P}_{\ast}(\overline{\Omega}_{1}\setminus\Omega_{2}),E)
\]
is an algebraic isomorphism.
\item[b)] The map
\[
\mathcal{L}\colon L(\mathcal{P}_{\ast}(\overline{\Omega}_{1}\setminus\Omega_{2}),E)\times
L(\mathcal{P}_{\ast}(\overline{\Omega}_{2}),E)\to L(\mathcal{P}_{\ast}(\overline{\Omega}_{1}),E)
\]
is surjective.
\end{enumerate}
\end{rem}
\begin{proof}
$I$ is obviously surjective if and only if $\mathcal{L}$ is surjective. 
Moreover, $I$ is always linear and injective by \prettyref{prop:traeger} a).
\end{proof}

The corresponding issue in Ito's paper \cite{Ito1998} on vector-valued hyperfunctions 
(see \cite[p.\ 34, l.\ 2, Problem A, p.\ 35]{Ito1998})
was pointed out by Doma\'nski and Langenbruch in \cite[Remark 6.3, p.\ 1123]{D/L}. 
Using \prettyref{lem:I_isom_op}, we can define the restrictions on 
$\mathcal{R}(\Omega,E)$, if $E$ is sequentially complete and admissible, as follows.

\begin{defn}\label{def:restrictions_op_sheaf}
Let $E$ be sequentially complete and admissible. 
For open sets $\Omega_{2}\subset\Omega_{1}\subset\overline{\R}$, $\Omega_{2}\neq \varnothing$, we denote by
\[
q\colon L(\mathcal{P}_{\ast}(\overline{\Omega}_{1}),E)/L(\mathcal{P}_{\ast}(\partial\Omega_{1}),E) \to
L(\mathcal{P}_{\ast}(\overline{\Omega}_{1}),E)/L(\mathcal{P}_{\ast}(\overline{\Omega}_{1}\setminus\Omega_{2}),E)
\]
the canonical quotient map. We define the restriction maps via \prettyref{lem:I_isom_op} by
\[
R_{\Omega_{1},\Omega_{2}}\colon\mathcal{R}(\Omega_{1},E)\to\mathcal{R}(\Omega_{2},E),\;
R_{\Omega_{1},\Omega_{2}}([T])\coloneq [T]_{\mid\Omega_{2}}\coloneq I^{-1}\bigl(q([T])\bigr),
\]
and for an open set $\Omega_{1}\subset\overline{\R}$
\[
R_{\Omega_{1},\varnothing}\colon\mathcal{R}(\Omega_{1},E)\to \mathcal{R}(\varnothing,E),\; 
R_{\Omega_{1},\varnothing}([T])\coloneq [T]_{\mid\varnothing}\coloneq 0.
\]
\end{defn}

The next lemma is the counterpart of \cite[Lemma 6.5, p.\ 1124]{D/L}.

\begin{lem}\label{lem:presheaf_S1}
Let $\Omega\subset\overline{\R}$ be open, $E$ sequentially complete and admissible and set
$\mathcal{R}_{\Omega}(E)\coloneq\{\mathcal{R}(\omega,E)\;|\;\omega\subset\Omega\;\text{open}\}$. 
Then $\mathcal{R}_{\Omega}(E)$, equipped with the restrictions of \prettyref{def:restrictions_op_sheaf}, 
forms a presheaf on $\Omega$, satisfying the condition $(S1)$:

For every family of open sets $\{\omega_{j}\subset\Omega\;|\;j\in J\}$ with $\omega\coloneq\bigcup_{j\in J}\omega_{j}$ holds:
If $[T]\in\mathcal{R}(\omega,E)$ such that $R_{\omega,\omega_{j}}([T])=0$ for all $j\in J$, then $[T]=0$.
\end{lem}
\begin{proof}
(i) We begin with the proof that $\mathcal{R}_{\Omega}(E)$ with its restrictions is a presheaf. 
We clearly have $R_{\omega,\omega}=\id_{\mathcal{R}(\omega,E)}$. 
Let $\omega_{3}\subset\omega_{2}\subset\omega_{1}\subset\Omega$ be open. 
We have to show that $R_{\omega_{2},\omega_{3}}\circ R_{\omega_{1},\omega_{2}}=R_{\omega_{1},\omega_{3}}$ is valid. 
This is obvious if one of the sets is empty, so let them all be non-empty. 
Let $T\in L(\mathcal{P}_{\ast}(\overline{\omega}_{1}),E)$. Let $T_{0}\in L(\mathcal{P}_{\ast}(\overline{\omega}_{3}),E)$ 
be a representative of $R_{\omega_{1},\omega_{3}}([T]_{1})$, let $T_{1}\in L(\mathcal{P}_{\ast}(\overline{\omega}_{2}),E)$ 
be a representative of $R_{\omega_{1},\omega_{2}}([T]_{1})$ and $T_{2}\in L(\mathcal{P}_{\ast}(\overline{\omega}_{3}),E)$ 
a representative of $R_{\omega_{2},\omega_{3}}\circ R_{\omega_{1},\omega_{2}}([T]_{1})=R_{\omega_{2},\omega_{3}}([T_{1}]_{2})$. 
By the definition of the restrictions the following is true:
\begin{enumerate}
	\item [(1)] $T_{0}-T\in L(\mathcal{P}_{\ast}(\overline{\omega}_{1}\setminus \omega_{3}),E)$,
	\item [(2)] $T_{1}-T\in L(\mathcal{P}_{\ast}(\overline{\omega}_{1}\setminus \omega_{2}),E)$,
	\item [(3)] $T_{2}-T_{1}\in L(\mathcal{P}_{\ast}(\overline{\omega}_{2}\setminus \omega_{3}),E)$.
\end{enumerate}
First, we observe that
\begin{equation}\label{eq:diff_T02}
T_{0}-T_{2}\in L(\mathcal{P}_{\ast}(\overline{\omega}_{3}),E).
\end{equation}
It remains to be shown that $T_{0}-T_{2}\in L(\mathcal{P}_{\ast}(\partial\omega_{3}),E)$. The equality
\[
T_{0}-T_{2}=(T_{0}-T)+(T-T_{1})+(T_{1}-T_{2})
\]
holds on $\mathcal{P}_{\ast}(\overline{\R})$ and the right-hand side is an element of
\[
L(\mathcal{P}_{\ast}(\overline{\omega}_{1}\setminus \omega_{3})
  \cap\mathcal{P}_{\ast}(\overline{\omega}_{1}\setminus \omega_{2})
  \cap \mathcal{P}_{\ast}(\overline{\omega}_{2}\setminus \omega_{3}),E)
=L(\mathcal{P}_{\ast}(\overline{\omega}_{1}\setminus \omega_{3}),E)
\]
by $(1)$--$(3)$ and as $\omega_{3}\subset\omega_{2}\subset\omega_{1}$. So due to the remark above \prettyref{prop:traeger}, 
$T_{0}-T_{2}$ can also be regarded as an element of $L(\mathcal{P}_{\ast}(\overline{\omega}_{1}\setminus \omega_{3}),E)$ 
and thus we get by \prettyref{prop:traeger} a) and \eqref{eq:diff_T02}
\begin{align*}
T_{0}-T_{2}&\in L(\mathcal{P}_{\ast}(\overline{\omega}_{3}),E)\cap L(\mathcal{P}_{\ast}(\overline{\omega}_{1}\setminus \omega_{3}),E)
 =L(\mathcal{P}_{\ast}(\overline{\omega}_{3}\cap(\overline{\omega}_{1}\setminus \omega_{3})),E)\\
&=L(\mathcal{P}_{\ast}(\partial\omega_{3}),E).
\end{align*}

(ii) Let $T$ be like in $(S1)$ and $j\in J$. Then for a representative $T_{j}$ of $R_{\omega,\omega_{j}}([T])$ 
it holds $T_{j}\in L(\mathcal{P}_{\ast}(\partial\omega_{j}),E)$, since $R_{\omega,\omega_{j}}([T])=0$, 
and $T-T_{j}\in L(\mathcal{P}_{\ast}(\overline{\omega}\setminus \omega_{j}),E)$ by the definition of the restriction. 
Again, the equality
\[
T=(T-T_{j})+T_{j}
\]
holds on $\mathcal{P}_{\ast}(\overline{\R})$ and the right-hand side is an element of
\[
 L(\mathcal{P}_{\ast}(\overline{\omega}\setminus \omega_{j})\cap\mathcal{P}_{\ast}(\partial\omega_{j}),E)
=L(\mathcal{P}_{\ast}(\overline{\omega}\setminus \omega_{j}),E).
\]
By the same argument as in part (i), we can regard $T$ as an element of 
$L(\mathcal{P}_{\ast}(\overline{\omega}\setminus \omega_{j}),E)$ and get
\[
\operatorname{supp} T\subset \overline{\omega}\setminus \omega_{j}
\]
where the support $\operatorname{supp} T$ is meant in the sense of \prettyref{prop:traeger} b). 
Since this is valid for all $j\in J$, we obtain
\[
 \operatorname{supp} T\subset\bigcap_{j\in J}(\overline{\omega}\setminus \omega_{j})
=\overline{\omega}\setminus\bigcup_{j\in J}\omega_{j}
=\overline{\omega}\setminus\omega
=\partial\omega
\]
and thus $T\in L(\mathcal{P}_{\ast}(\partial\omega),E)$, i.e.\ $[T]=0$.
\end{proof}

For the special case $\Omega=\overline{\R}$ we use the notation $\mathcal{R}(E)\coloneq\mathcal{R}_{\overline{\R}}(E)$. 
We will see that the presheaf $\mathcal{R}_{\Omega}(E)$, which satisfies $(S1)$, is already a sheaf, so satisfies, in addition, 
the sheaf condition $(S2)$ if we assume that $E$ is not only admissible, but strictly admissible. 
For this purpose we introduce a boundary value representation of $\mathcal{R}(E)$ in the following way. 
Let $\Omega\subset \overline{\R}$, $\Omega\neq\varnothing$, be an open set and we define
\[
\mathcal{U}(\Omega)\coloneq
\{U\;|\; U\subset\overline{\C}\;\text{open},\; U\cap\overline{\R}=\Omega\}.
\]
Now, we define, similar to \prettyref{def:smooth_weighted_space}, spaces of vector-valued 
slowly increasing holomorphic functions on $U\setminus\overline{\R}$ resp.\ $U$ for $U\in\mathcal{U}(\Omega)$.

If $-\infty\in\Omega$ or $\infty\in\Omega$, we define
\[
  \mathcal{O}^{exp}(U\setminus\overline{\R}, E)
\coloneq \{f\in\mathcal{O}((U\setminus\overline{\R})\cap\C,E)\;|\;\forall\;n\in\N,\,n\geq 2,
     \,\alpha\in\mathfrak{A}:\;\vertiii{f}_{U^{\ast},n,\alpha}<\infty\}
\]
where
\[
\vertiii{f}_{U^{\ast},n,\alpha}\coloneq\sup_{z \in S_{n}(U)}p_{\alpha}(f(z))\e^{-\frac{1}{n}|\re(z)|}
\]
and
\begin{flalign*}
&\hspace{0.35cm}S_{n}(U)\\
&\coloneq
\begin{cases}
U\cap\{z\in\C\;|\;\tfrac{1}{n}<|\im(z)|< n,\; \re(z)>-n,\;\d(z,\C\cap\partial U)>\tfrac{1}{n}\} &,
 \; -\infty\not\in\Omega,\, \infty\in\Omega, \\
U\cap\{z\in\C\;|\;\tfrac{1}{n}<|\im(z)|< n,\; \re(z)<n,\;\d(z,\C\cap\partial U)>\tfrac{1}{n}\} &,
 \; -\infty\in\Omega,\, \infty\not\in\Omega,\\
U\cap\{z\in\C\;|\;\tfrac{1}{n}<|\im(z)|< n,\;\d(z,\C\cap\partial U)>\tfrac{1}{n}\} &,
 \; \pm\infty\in\Omega.
\end{cases}
\end{flalign*}
\begin{center}
\includegraphics[scale=0.85]{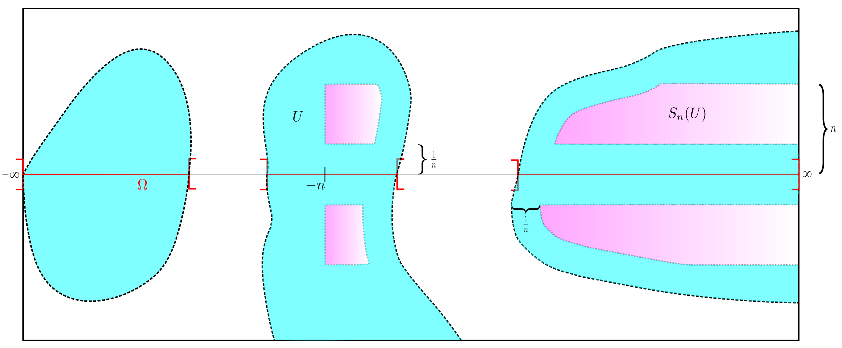}
\captionsetup{type=figure}
\caption{$S_{n}(U)$ for $\infty\in\Omega$, $-\infty\notin\Omega$}
\end{center}

If $\pm\infty\not\in\Omega$, we define
\[
\mathcal{O}^{exp}(U\setminus\overline{\R}, E)\coloneq \mathcal{O}((U\setminus\overline{\R})\cap\C, E).
\]

If $-\infty\in\Omega$ or $\infty\in\Omega$, we define
\[
\mathcal{O}^{exp}(U,E)\coloneq\{f\in\mathcal{O}(U\cap\C,E)\;|\;\forall\;n\in\N,\,n\geq 2,
                          \,\alpha\in\mathfrak{A}:\;\vertiii{f}_{U,n,\alpha}<\infty\}
\]
where
\[
\vertiii{f}_{U,n,\alpha}\coloneq\sup_{z\in T_{n}(U)}p_{\alpha}(f(z))\e^{-\frac{1}{n}|\re(z)|}
\]
and
\begin{align*}
&\hspace{0.35cm}T_{n}(U)\\
&\coloneq
\begin{cases}
U\cap\{z\in\C\;|\;|\im(z)|< n,\;\re(z)>-n,\;\d(z,\C\cap\partial U)>\tfrac{1}{n}\} &,\; -\infty\not\in\Omega,\, \infty\in\Omega, \\
U\cap\{z\in\C\;|\;|\im(z)|< n,\;\re(z)<n,\;\d(z,\C\cap\partial U)>\tfrac{1}{n}\} &,\; -\infty\in\Omega,\, \infty\not\in\Omega,\\
U\cap\{z\in\C\;|\;|\im(z)|< n,\;\d(z,\C\cap\partial U)>\tfrac{1}{n}\} &,\; \pm\infty\in\Omega.
\end{cases}
\end{align*}

If $\pm\infty\not\in\Omega$, we define
\[
\mathcal{O}^{exp}(U,E)\coloneq\mathcal{O}(U\cap\C,E).
\]

We remark that $\mathcal{O}^{exp}(U\setminus\overline{\R},E)$ and $\mathcal{O}^{exp}(U,E)$ for complete $E$ 
are complete $\C$-lcHs by \cite[Proposition 3.7, p.\ 240]{kruse2018_2} if $-\infty\in\Omega$ or $\infty\in\Omega$. 
If $\pm\infty\notin\Omega$, then this is obviously valid for the corresponding spaces as well if equipped 
with the topology of uniform convergence on compact subsets. Moreover, if $U=\overline{\C}$, so $\Omega=\overline{\R}$, 
then the definition of $\mathcal{O}^{exp}(\overline{\C}\setminus\overline{\R},E)$ and $\mathcal{O}^{exp}(\overline{\C},E)$ 
in the just introduced sense coincides with the one in the sense of \prettyref{def:smooth_weighted_space} 
(and therefore the spaces have the same symbol).

\begin{defn}\label{def:bv}
For an open set $\Omega\subset\overline{\R}$, $\Omega\neq\varnothing$, $U\in\mathcal{U}(\Omega)$ and an $\C$-lcHs $E$ 
we define the space of boundary values by
\[
bv(\Omega,E)\coloneq\mathcal{O}^{exp}(U\setminus\overline{\R},E)/\mathcal{O}^{exp}(U,E)
\]
and $bv(\varnothing,E)\coloneq 0$.
\end{defn}

The counterpart of the next observation in the context of vector-valued hyperfunctions can be found in \cite[Lemma 6.7, p.\ 1124]{D/L}.

\begin{lem}\label{lem:bv_indep_U}
Let $\Omega\subset\overline{\R}$ be non-empty and open and $E$ locally complete and admissible. 
The definition of $bv(\Omega,E)$ is independent 
of the choice of $U\in\mathcal{U}(\Omega)$ and for every $f\in bv(\Omega,E)$ 
there is $F\in\mathcal{O}^{exp}(\overline{\C}\setminus\overline{\Omega},E)$ such 
that $f=[F]$.
\end{lem}
\begin{proof}
Let $U,U_{1}\in\mathcal{U}(\Omega)$, w.l.o.g.\ $U_{1}\coloneq (\overline{\C}\setminus\overline{\R})\cup\Omega$. 
Then $U\subset U_{1}$. The canonical map
\[
J\colon \mathcal{O}^{exp}(U_{1}\setminus\overline{\R},E)/\mathcal{O}^{exp}(U_{1},E)
 \to\mathcal{O}^{exp}(U\setminus\overline{\R},E)/\mathcal{O}^{exp}(U, E),\;
 [f]\mapsto[f_{\mid(U\setminus\overline{\R})\cap\C}],
\]
is well-defined since $\mathcal{O}^{exp}(U_{1},E)\subset\mathcal{O}^{exp}(U,E)$.

Let $f\in\mathcal{O}^{exp}(U_{1}\setminus\overline{\R},E)$ with $[f_{\mid(U\setminus\overline{\R})\cap\C}]=0$, i.e.\ 
$f_{\mid(U\setminus\overline{\R})\cap\C}\in\mathcal{O}^{exp}(U,E)$. Then
\[
f\in\mathcal{O}^{exp}((U_{1}\setminus\overline{\R})\cup U,E)=\mathcal{O}^{exp}(U_{1},E)
\]
and therefore $[f]=0$, yielding the injectivity of $J$.

The proof of surjectivity resembles the one of \prettyref{lem:I_isom_op}, but it is sometimes necessary to use two cut-off functions. 
We restrict to the case that $\infty\in\Omega$ and $-\infty\in\partial\Omega$. 
For the similar treatment of the other cases we refer to the proof of \cite[6.8 Lemma, p.\ 118]{ich}.

(i) There are $\widetilde{x}_{0}\in\R$, w.l.o.g.\ $\widetilde{x}_{0}\geq 0$, and $\varepsilon_{0}>0$ 
such that $[\widetilde{x}_{0},\infty]\subset\Omega$ 
and $[\widetilde{x}_{0},\infty]\times[-\varepsilon_{0},\varepsilon_{0}]\subset U$ since $\infty\in\Omega$, 
$\Omega$ is open and $U\in\mathcal{U}(\Omega)$. We define the sets
\[
F_{0}\coloneq (U^{C}\cap\R^{2})\cup\bigr[\R\times\bigl(\R\setminus(-\tfrac{\varepsilon_{0}}{2},\tfrac{\varepsilon_{0}}{2})\bigr)\bigr]
\]
and
\[
F_{1}\coloneq (\R\cap\overline{\Omega})\cup([\widetilde{x}_{0}+2,\infty)\times[-\tfrac{\varepsilon_{0}}{4},\tfrac{\varepsilon_{0}}{4}]).
\]
The sets $F_{0}$ and $F_{1}$ are non-empty and closed in $\R^{2}$ and $F_{0}\cap F_{1}=\R\cap\partial\Omega$. 
By \cite[Corollary 1.4.11, p.\ 31]{H1} there exists 
$\varphi_{0}\in\mathcal{C}^{\infty}((F_{0}\cap F_{1})^{C})=\mathcal{C}^{\infty}(\R^{2}\setminus\partial\Omega)$, 
$0\leq\varphi_{0}\leq 1$, such that $\varphi_{0}=0$ on $V_{0}$ and $\varphi_{0}=1$ on $V_{1}$ 
where $V_{0}$, $V_{1}\subset\R^{2}$ are open and
\[
V_{0}\supset F_{0}\setminus(F_{0}\cap F_{1})=F_{0}\setminus\partial\Omega\supset(\R\setminus\overline{\Omega})
\]
and
\[
V_{1}\supset F_{1}\setminus(F_{0}\cap F_{1})=F_{1}\setminus\partial\Omega\supset(\R\cap\Omega)
\]
as well as
\begin{equation}\label{eq:estimate_molifier_phi0}
|\partial^{\beta}\varphi_{0}(z)|\leq C^{|\beta|}\frac{\d(z)^{-|\beta|}}{d_{1}\cdots d_{|\beta|}},
\quad z\in\R^{2}\setminus\partial\Omega,\,\beta\in\N^{2}_{0},
\end{equation}
where $C$, $\d$ and $(d_{n})$ with $d_{1}\cdots d_{|0|}\coloneq 1$ are like in part (i) of the proof of \prettyref{lem:I_isom_op}.
\begin{center}
\includegraphics[scale=0.8]{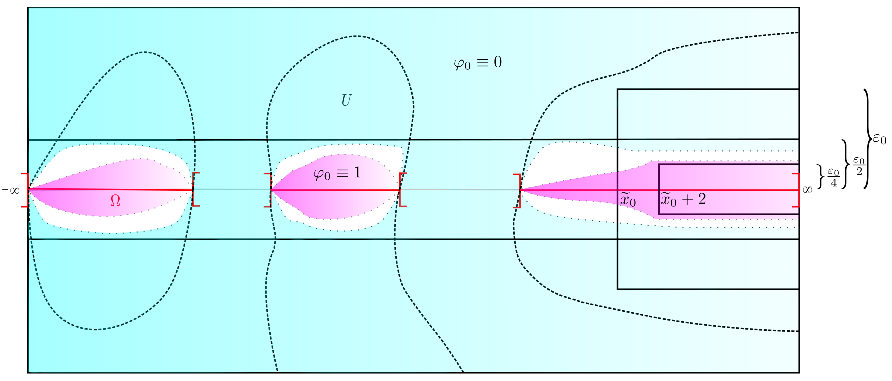}
\captionsetup{type=figure}
\caption{case: $\infty\in\Omega$, $-\infty\in\partial\Omega$}
\end{center}
Furthermore, we define the sets $K_{0}\coloneq\{(x,y)\in\R^{2}\;|\;y\leq -2\e^{-|x|}\;\text{or}\; y\geq 2\e^{-|x|}\}$ 
and $K_{1}\coloneq\{(x,y)\in\R^{2}\;|\; -\e^{-|x|}\leq y\leq \e^{-|x|}\}$ as well as
\[
\widetilde{F}_{0}\coloneq K_{0}\cup([0,\infty)\times[\R\setminus (-2,2)])
\quad\text{and}\quad
\widetilde{F}_{1}\coloneq K_{1}\cup([0,\infty)\times[-1,1]).
\]
The sets $\widetilde{F}_{0}$ and $\widetilde{F}_{1}$ are non-empty and closed in $\R^{2}$ 
and $\widetilde{F}_{0}\cap\widetilde{F}_{1}=\varnothing$. Like above there is 
$\varphi_{1}\in\mathcal{C}^{\infty}((\widetilde{F}_{0}\cap\widetilde{F}_{1})^{C})=\mathcal{C}^{\infty}(\R^{2})$, 
$0\leq\varphi_{1}\leq 1$, such that $\varphi_{1}=0$ on $W_{0}$ and $\varphi_{1}=1$ on $W_{1}$ 
where $W_{0}$, $W_{1}\subset\R^{2}$ are open and
\[
W_{0}\supset\widetilde{F}_{0}\setminus(\widetilde{F}_{0}\cap\widetilde{F}_{1})=\widetilde{F}_{0}
\quad\text{and}\quad
W_{1}\supset\widetilde{F}_{1}\setminus(\widetilde{F}_{0}\cap\widetilde{F}_{1})=\widetilde{F}_{1}
\]
as well as
\begin{equation}\label{eq:estimate_molifier_phi1}
|\partial^{\beta}\varphi_{1}(z)|\leq \widetilde{C}^{|\beta|}\frac{\widetilde{\d}(z)^{-|\beta|}}{d_{1}\cdots d_{|\beta|}},
\quad z\in\R^{2},\,\beta\in\N^{2}_{0},
\end{equation}
where $\widetilde{C}$, $\widetilde{\d}$ and $(d_{n})$ are like above. 
\begin{center}
\includegraphics[scale=0.8]{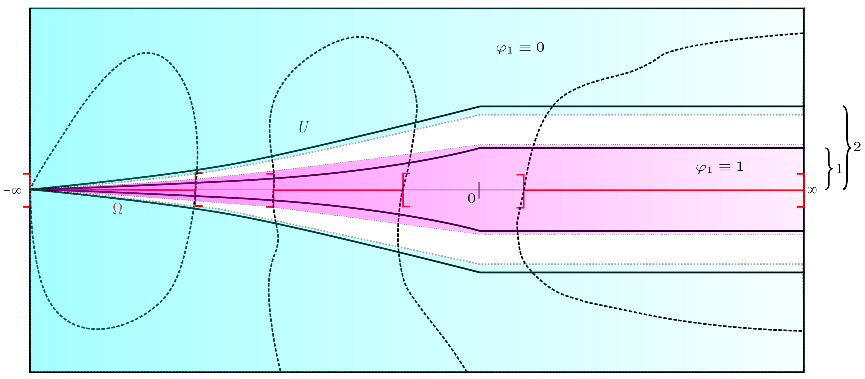}
\captionsetup{type=figure}
\caption{case: $\infty\in\Omega$, $-\infty\in\partial\Omega$}
\end{center}
Again, we take a closer look at the right-hand side of \eqref{eq:estimate_molifier_phi0} 
resp.\ \eqref{eq:estimate_molifier_phi1} and claim that 
\begin{equation}\label{eq:dist_d_estimate}
B\coloneq\inf_{z\in S_{n}(\partial\Omega)}\d(z)>0
\end{equation}
and
\begin{equation}\label{eq:dist_tilde_d_estimate} 
D\coloneq\inf_{z\in S_{n}(\partial\Omega)}\widetilde{\d}(z)>0
\end{equation}
for all $n\in\N$, $n\geq 2$. We begin with \eqref{eq:dist_d_estimate}.

(\ref{eq:dist_d_estimate}.1) For $z\in S_{n}(\partial\Omega)$ with $\re(z)\leq -n$ we have
\[
 \d(z)
=\max(\d(z,F_{0}),\d(z,F_{1}))\geq\d(z,F_{1})
\geq\min\bigl(n+\widetilde{x}+2-\frac{1}{n},\frac{1}{n}\bigr)\geq\frac{1}{n}.
\]
For $z\in S_{n}(\partial\Omega)$ with $\re(z)\geq\widetilde{x}_{0}+2$ we have
\begin{align*}
\d(z)&\geq
\left.\begin{cases}
 \frac{\varepsilon_{0}}{4} &,\;z\in F_{0},\\
 \min(\frac{1}{2}\frac{\varepsilon_{0}}{4},2) &,\;z\notin F_{0},\, z\notin F_{1},\\
 \min(\frac{\varepsilon_{0}}{4},2) &,\;z\in F_{1},
\end{cases}\right\}
\geq\min\bigl(\frac{\varepsilon_{0}}{8},2\bigr).
\end{align*}

(\ref{eq:dist_d_estimate}.2) For $z\in S_{n}(\partial\Omega)$ with $\re(z)\leq \widetilde{x}_{0}$ and $|\im(z)|\geq\tfrac{1}{n}$ we get
\[
\d(z)\geq\d(z,F_{1})\geq\min\bigl(2,\frac{1}{n}\bigr)=\frac{1}{n}.
\]

(\ref{eq:dist_d_estimate}.3) By \cite[Remark 6 (a), p.\ 87--88]{kruse2019_2} 
the set $U_{n}(\partial\Omega)$ has finitely many components $Z_{j}$, 
so there exists $k\in\N$ with $U_{n}(\partial\Omega)=\bigcup^{k}_{j=1}Z_{j}$. 
Since $-\infty\in\partial\Omega$ and $\infty\notin\partial\Omega$, all but one $Z_{j}$ are bounded. 
Denote by $Z_{1}$ the unbounded component and by $Z_{j}$, $2\leq j\leq k$, the bounded components if $k\geq 2$. 
Let $a_{j}\coloneq\min(Z_{j}\cap\partial\Omega)$ and $b_{j}\coloneq\max(Z_{j}\cap\partial\Omega)$, $2\leq j\leq k$, 
if $k\geq 2$. 
If $\max(Z_{1}\cap\partial\Omega)$ exists, we set $b_{1}\coloneq\max(Z_{1}\cap\partial\Omega)$, and 
if $\max(Z_{1}\cap\partial\Omega)$ does not exist, implying $Z_{1}=(-\infty,n]\times[-\tfrac{1}{n},\tfrac{1}{n}]$, 
then we set $b_{1}\coloneq -n-\tfrac{1}{n}$. We observe that $\max_{1\leq j\leq k}b_{j}<\widetilde{x}_{0}$. 

Let $k\geq 2$. W.l.o.g.\ $a_{j}<a_{j+1}$ for $2\leq j\leq k-1$ (otherwise renumber). 
Then we have $b_{k}=\max_{1\leq j\leq k}b_{j}$. Due to \cite[Remark 6 (b)(i), (iii), (v), p.\ 87--88]{kruse2019_2} 
there is $0<r_{j}<\tfrac{1}{n}$ such that $\{z\in\C\;|\;\d(z,(-\infty,b_{1}])\leq r_{1}\}\subset Z_{1}$ 
and $\{z\in\C\;|\;\d(z,[a_{j},b_{j}])\leq r_{j}\}\subset Z_{j}$ for all $2\leq j\leq k$ 
(see \prettyref{fig:path_gamma}). 
Since $k\geq 2$, we have $a_{1}\coloneq -n<b_{k}$. 
Let $z\in S_{n}(\partial\Omega)$ such that $a_{1}=-n<\re(z)<b_{k}$ and $|\im(z)|<\tfrac{1}{n}$.
If $a_{j}<b_{j}$ for some $2\leq j\leq k$, we obtain for $z$ with $a_{j}<\re(z)<b_{j}$
\[
\d(z)\geq\d(z,F_{1})\geq r_{j}.
\]
Now, we consider $z\in S_{n}(\partial\Omega)$ with $b_{j}<\re(z)<a_{j+1}$ for $1\leq j\leq k-1$ and $|\im(z)|<\tfrac{1}{n}$. 
If $\d(z)\leq\tfrac{1}{2n}$, 
we have with $N_{0}\coloneq \{w\in\C\;|\;|\im(w)|>\tfrac{3}{2n}\}$ and 
$N_{1}\coloneq\{w\in\C\;|\;\d(\R\cap\partial\Omega)<\tfrac{1}{3n}\}$ that
\[
\d(z,F_{1})=\d(z,F_{1}\setminus N_{1})=\d(z,\underbrace{([b_{j},a_{j+1}]\cap\overline{\Omega})\setminus N_{1}}_{\eqqcolon K_{1,j}})
\]
and
\begin{align*}
  \d(z,F_{0})
&=\d(z,F_{0}\setminus(N_{0}\cup N_{1}))\\
&=\d(z,\underbrace{(F_{0}\cap\{w\in\C\;|\;b_{j}\leq\re(w)\leq a_{j+1}\})
     \setminus(N_{0}\cup N_{1})}_{\eqqcolon K_{0,j}})
\end{align*}
because $\tfrac{1}{n}-\tfrac{1}{3n}>\tfrac{1}{2n}$ and $\tfrac{1}{n}+\tfrac{1}{2n}=\tfrac{3}{2n}$.
The sets $K_{0,j}$ and $K_{1,j}$ are bounded closed sets in $\R^{2}$, thus compact, and disjoint. 
Hence $c_{j}\coloneq\d(K_{0,j},K_{1,j})>0$, yielding to
\[
\d(z)=\max_{i\in\{0,1\}}\d(z,K_{i,j})\geq\frac{c_{j}}{2}>0
\]
for all $1\leq j\leq k-1$. Combining these results, we obtain
\[
\d(z)\geq\min\bigl(\min_{1\leq j\leq k,a_{j}\neq b_{j}}r_{j},\min_{1\leq j\leq k-1}\frac{c_{j}}{2}, \frac{1}{2n}\bigr)>0
\]
for $z\in S_{n}(\partial\Omega)$ with $|\im(z)|\leq\tfrac{1}{n}$ and $a_{1}<\re(z)<b_{k}$ with $k\geq 2$.

(\ref{eq:dist_d_estimate}.4) Let $k\in\N$, $z\in S_{n}(\partial\Omega)$ such that $b_{k}\leq\re(z)<\widetilde{x}_{0}+2$.
If $\d(z)\leq\tfrac{1}{2n}$, we get with $N_{2}\coloneq\{w\in\C\;|\;\re(w)>\widetilde{x}_{0}+2+\tfrac{1}{2n}\}$ and 
$N_{3}\coloneq\{w\in\C\;|\;|\im(w)|>n+\tfrac{1}{2n}\;\text{or}\;\re(w)<b_{k}\}$ that
\begin{align*}
  \d(z,F_{1})
&=\d(z,F_{1}\setminus(\D_{\tfrac{1}{3n}}(b_{k})\cup N_{2}))\\
&=\d(z,\underbrace{\{w\in F_{1}\;|\; \re(w)\geq b_{k}\}\setminus(\D_{\tfrac{1}{3n}}(b_{k})\cup N_{2})}_{\eqqcolon \widetilde{K}_{1}})
\end{align*}
as well as
\[
\d(z,F_{0})=\d(z,\underbrace{F_{0}\setminus(\D_{\tfrac{1}{3n}}(b_{k})\cup N_{2}\cup N_{3})}_{\eqqcolon \widetilde{K}_{0}}).
\]
$\widetilde{K}_{0}$ and $\widetilde{K}_{1}$ are compact and disjoint. Thus we have $c_{0}\coloneq\d(\widetilde{K}_{0},\widetilde{K}_{1})>0$, 
implying
\[
\d(z)=\max_{i\in\{0,1\}}\d(z,\widetilde{K}_{i})\geq\frac{c_{0}}{2}>0.
\]

(\ref{eq:dist_d_estimate}.5) Merging (\ref{eq:dist_d_estimate}.1)-(\ref{eq:dist_d_estimate}.4), we gain
\[
B=\inf_{z\in S_{n}(\partial\Omega)}\d(z)
 \geq\min\bigl(\frac{1}{n},\min\bigl(\frac{\varepsilon_{0}}{8},2\bigr),\min\bigl(\min_{j\in J}r_{j},
               \min_{1\leq j\leq k-1}\frac{c_{j}}{2},\frac{1}{2n}\bigr),\frac{c_{0}}{2}\bigr)>0
\]
if $k\geq 2$ and $J\coloneq\{j\in\N\;|\;j\leq k,\,a_{j}<b_{j}\}\neq\varnothing$. 
If $J=\varnothing$ resp.\ $k=1$, then the $\min_{j\in J}$-term resp.\ the $\min_{1\leq j\leq k-1}$-term 
does not appear in the estimate above.
 
Let us turn to \eqref{eq:dist_tilde_d_estimate}. 

(\ref{eq:dist_tilde_d_estimate}.1) For $z\in S_{n}(\partial\Omega)$ with $\re(z)\geq 1$ we have
\begin{align*}
  \widetilde{\d}(z)
&=\max(\d(z,\widetilde{F}_{0}),\d(z,\widetilde{F}_{1}))
 \geq
 \left.\begin{cases}
	1 &,\;z\in \widetilde{F}_{0},\\
	\min(\frac{1}{2},1) &,\;z\notin\widetilde{F}_{0},\,z\notin\widetilde{F}_{1},\\
	\min(1,1) &,\;z\in\widetilde{F}_{1},
	\end{cases}\right\}
 \geq\frac{1}{2}.
\end{align*}	
	
(\ref{eq:dist_tilde_d_estimate}.2) Let $z\in S_{n}(\partial\Omega)$ such that $0\leq\re(z)<1$. 
If $\widetilde{\d}(z)\leq\tfrac{1}{2n}$, then
\[
\d(z,\widetilde{F}_{0})=\d(z,\widetilde{F}_{0}\setminus (N_{0}\cup N_{1}))\quad\text{and}\quad
\d(z,\widetilde{F}_{1})=\d(z,\widetilde{F}_{1}\setminus  N_{1})
\]
where $N_{0}\coloneq\{w\in\C\;|\;|\im(w)|>n+\tfrac{1}{2n}\}$ and 
$N_{1}\coloneq\{w\in\C\;|\;\re(w)<-\tfrac{1}{n}\;\text{or}\;\re(w)>1+\tfrac{1}{n}\}$. 
The sets $\widetilde{F}_{0}\setminus(N_{0}\cup N_{1})$ and $\widetilde{F}_{1}\setminus N_{1}$ are compact and disjoint, 
thus we gain $c_{0}\coloneq\d(\widetilde{F}_{0}\setminus(N_{0}\cup N_{1}),\widetilde{F}_{1}\setminus N_{1})>0$ 
and therefore
\[
\widetilde{\d}(z)\geq\frac{c_{0}}{2}>0.
\]

(\ref{eq:dist_tilde_d_estimate}.3) Let $z\in S_{n}(\partial\Omega)$ with $ \re(z)<0$. 
If $\widetilde{\d}(z)\leq\tfrac{1}{2n}$, then
\[
\d(z,\widetilde{F}_{0})=\d(z,\widetilde{F}_{0}\setminus(N_{0}\cup N_{2}))\quad\text{and}\quad
\d(z,\widetilde{F}_{1})=\d(z,\widetilde{F}_{1}\setminus  N_{2})
\]
with $N_{0}$ from (\ref{eq:dist_tilde_d_estimate}.2) and 
$N_{2}\coloneq\{w\in\C\;|\;(|\im(w)|<\tfrac{1}{3n}\;\text{and}\;\re(w)<-n-\tfrac{1}{2n})\;\text{or}\;\re(w)>\tfrac{1}{n}\}$. 
The sets $\widetilde{F}_{0}\setminus(N_{0}\cup N_{2})$ and $\widetilde{F}_{1}\setminus N_{2}$ are compact and disjoint, 
so we obtain $c_{1}\coloneq\d(\widetilde{F}_{0}\setminus(N_{0}\cup N_{2}),\widetilde{F}_{1}\setminus N_{2})>0$ 
and hence $\widetilde{\d}(z)\geq\tfrac{c_{1}}{2}>0$.

(\ref{eq:dist_tilde_d_estimate}.4) By combining these results, we have
\[
D=\inf_{z\in S_{n}(\partial\Omega)}\widetilde{\d}(z)\geq\min\bigl(\frac{1}{2},\frac{1}{2n},\frac{c_{0}}{2},\frac{c_{1}}{2}\bigr)>0.
\]

(ii) Let $f\in\mathcal{O}^{exp}(U\setminus\overline{\R},E)$. By the choice of $\varphi_{0}$ and $\varphi_{1}$ 
the function $\overline{\partial}(\varphi_{1}\varphi_{0}f)$ may be regarded as an element of 
$\mathcal{C}^{\infty}(\R^{2}\setminus\partial\Omega,E)$ by $\mathcal{C}^{\infty}$-extension via 
$\overline{\partial}(\varphi_{1}\varphi_{0}f)\coloneq 0$ on $[(U^{C}\cap\R^{2})\cup\R]\setminus\partial\Omega$. 
Moreover, with the definition
\[
V\coloneq (V_{0}\cup W_{0})\cup(V_{1}\cap W_{1}),
\]
the equation
\[
\overline{\partial}(\varphi_{1}\varphi_{0} f)(z)=
\begin{cases}
0 &,\;z\in V,\\
[(\overline{\partial}\varphi_{1})\varphi_{0}f+(\overline{\partial}\varphi_{0})\varphi_{1}f](z) &,\;\text{else},
\end{cases}
\]
is valid.

The next step is similar to \eqref{eq:estimate_molified_f}. Let $n\in\N$, $n\geq 2$, $m\in\N_{0}$ and $\alpha\in\mathfrak{A}$.
We define the set $S(n)\coloneq S_{n}(\partial\Omega)\setminus V$ and the cardinality $C_{m}\coloneq |\{\gamma\in\N^{2}_{0}\;|\;|\gamma|\leq m\}|$.
By applying the Leibniz rule twice, we obtain 
\begin{flalign}\label{eq:estimate_CR_phi_0_1_f}
&\hspace{0.35cm}|\overline{\partial}(\varphi_{1}\varphi_{0}f)|_{\partial\Omega,n,m,\alpha}\nonumber\\
&=\sup_{\substack{z\in S_{n}(\partial\Omega)\\\beta\in\N^{2}_{0},|\beta|\leq m}}
  p_{\alpha}(\partial^{\beta}\overline{\partial}(\varphi_{1}\varphi_{0}  f)(z))\e^{-\frac{1}{n}|\re(z)|}\nonumber\\
&\underset{\mathclap{\eqref{eq:real.compl.part.deriv.1}}}{\leq}(m!)^{2}
 \hspace{-0.2cm}\sup_{\substack{z\in S(n)\\ \beta\in\N^{2}_{0},|\beta|\leq m}}
 \sum_{\gamma\leq\beta}|\partial^{\beta-\gamma}[(\overline{\partial}\varphi_{1})\varphi_{0}
 +(\overline{\partial}\varphi_{0})\varphi_{1}](z)
 |\underbrace{\sup_{\substack{z\in S(n)\\ \beta\in\N^{2}_{0},|\beta|\leq m}}
 p_{\alpha}(\partial_{\C}^{|\beta|}f(z))\e^{-\frac{1}{n}|\re(z)|}}_{\eqqcolon C(f)}\nonumber\\
&\leq(m!)^{4}C(f)\hspace{-0.2cm}\sup_{\substack{z\in S(n)\\ \beta\in\N^{2}_{0},|\beta|\leq m}}
 \sum_{\gamma\leq \beta}\sum_{\tau\leq\beta-\gamma}|\partial^{\tau}(\overline{\partial}\varphi_{1})(z)
 \partial^{\beta-\gamma-\tau}\varphi_{0}(z)
 +\partial^{\tau}(\overline{\partial}\varphi_{0})(z)\partial^{\beta-\gamma-\tau}\varphi_{1}(z)|\nonumber\\
&\leq(m!)^{4}C(f)\sum_{\substack{|\gamma|\leq m\\|\tau|\leq m+1}}\sup_{z\in S(n)}|\partial^{\tau}\varphi_{1}(z)|\;
 \sup_{\mathclap{\substack{z\in S(n)\\ \upsilon\in\N^{2}_{0},|\upsilon|\leq m}}}|\partial^{\upsilon}\varphi_{0}(z)|
 +\sup_{z\in S(n)}|\partial^{\tau}\varphi_{0}(z)|\;
  \sup_{\mathclap{\substack{z\in S(n)\\ \upsilon\in\N^{2}_{0},|\upsilon|\leq m}}}|\partial^{\upsilon}\varphi_{1}(z)|\nonumber\\
&\;\underset{\mathclap{\substack{\eqref{eq:estimate_molifier_phi0},\\\eqref{eq:estimate_molifier_phi1}}}}{\leq}\;(m!)^{4}C_{m}C(f)
 \sum_{|\tau|\leq m+1}\;\widetilde{C}^{|\tau|}\sup_{z\in S(n)}\frac{\widetilde{\d}(z)^{-|\tau|}}{d_{1}\cdots d_{|\tau|}}\;\;
 \sup_{\mathclap{\substack{z\in S(n)\\ \upsilon\in\N^{2}_{0},|\upsilon|\leq m}}}C^{|\upsilon|}
 \frac{\d(z)^{-|\upsilon|}}{d_{1}\cdots d_{|\upsilon|}}\nonumber\\
&\phantom{\;\underset{\mathclap{\substack{\eqref{eq:estimate_molifier_phi0},\\\eqref{eq:estimate_molifier_phi1}}}}{\leq}\;}
 +C^{|\tau|}\sup_{z\in S(n)}\frac{\d(z)^{-|\tau|}}{d_{1}\cdots d_{|\tau|}}\;\;
 \sup_{\mathclap{\substack{z\in S(n)\\ \upsilon\in\N^{2}_{0},|\upsilon|\leq m}}}
 \widetilde{C}^{|\upsilon|}\frac{\widetilde{\d}(z)^{-|\upsilon|}}{d_{1}\cdots d_{|\upsilon|}}\nonumber\\
&\;\;\;\underset{\mathclap{\substack{\eqref{eq:dist_d_estimate},\\\eqref{eq:dist_tilde_d_estimate}}}}{\leq}\;(m!)^{4}C_{m}
 \frac{[\max(C,\widetilde{C},1)]^{m+1}}{(d_{1}\cdots d_{m+1})^{2}}C(f)\sum_{|\tau|\leq m+1}D^{-|\tau|}
 \sup_{\substack{\upsilon\in\N^{2}_{0}\\ |\upsilon|\leq m}}B^{-|\upsilon|}+B^{-|\tau|}
 \sup_{\substack{\upsilon\in\N^{2}_{0}\\ |\upsilon|\leq m}}D^{-|\upsilon|}.
\end{flalign}
Now, we have to take a closer look at $C(f)$. First of all, we remark that
\begin{align*}
  [(U^{C}\cup\overline{\R})\cap\R^{2}]
&=([(U^{C}\cap\R^{2})\cup(\R\cap\overline{\Omega})]\setminus\partial\Omega)\cup(\partial\Omega\cap\R)\\
&\subset\bigl[(V_{0}\cup W_{0})\cup(V_{1}\cap W_{1})\cup\bigcup_{x\in\R\cap\partial\Omega}\D_{\tfrac{1}{n}}(x)\bigr]\\
&=V\cup\bigcup_{x\in\R\cap\partial\Omega}\D_{\tfrac{1}{n}}(x)\eqqcolon W.
\end{align*}
$W$ is an open set in $\R^{2}$ as the union of open sets and we get
\begin{equation}\label{eq:M_0_inclusion}
\overline{S(n)}=\overline{[S_{n}(\partial\Omega)\setminus V]}\subset\overline{W^{C}}=W^{C}
\subset[(U\setminus\overline{\R})\cap\R^{2}].
\end{equation}
In the following we prove that there are $k\in\N$, $k\geq 2$, $M_{0}\subset S(n)$ bounded and $M_{1}\subset S_{k}(U)$ 
such that
\[
S(n)\subset(M_{0}\cup M_{1}).
\]
As $|\im(z)|\leq\tfrac{1}{n}$ for every $z\in S(n)$, it suffices to prove that there is $C_{1}>0$ 
such that $|\re(z)|\leq C_{1}$ for every $z\in M_{0}$. 
We define the set $M\coloneq\{z\in\C\;|\; \re(z)>\widetilde{x}_{0}+2\}$
and decompose
\[
S(n)
=\underbrace{[S(n)\setminus M]}_{\eqqcolon M_{0}}\cup\underbrace{[S(n)\cap M]}_{\eqqcolon M_{1}}.
\]
\begin{center}
\includegraphics[scale=0.8]{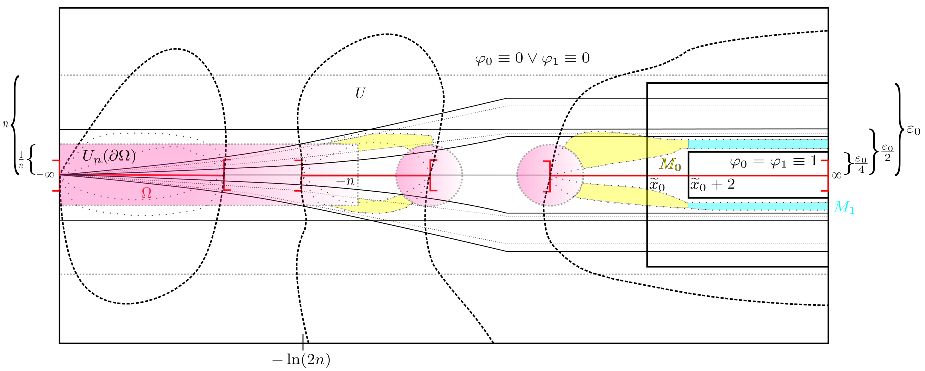}
\captionsetup{type=figure}
\caption{case: $\infty\in\Omega$, $-\infty\in\partial\Omega$}
\end{center}
We observe that the inequality $\tfrac{1}{n}\geq 2e^{-|x|}$ is equivalent to $\ln(2n)\leq|x|$ for all $x\in\R$. 
Hence $M_{0}$ is bounded since
\[
|\re(z)|\leq\max(|-n|,|-\ln(2n)|,|\widetilde{x}_{0}+2|)
\]
for all $z\in M_{0}$. Let 
\[
r\coloneq\frac{1}{2}\min\bigl(2,\frac{\varepsilon_{0}}{2},\frac{\varepsilon_{0}}{4}\bigr)=\min\bigl(1,\frac{\varepsilon_{0}}{8}\bigr),
\]
choose $k\in\N$ with $k>\max(n,\varepsilon_{0})$ and $\frac{1}{k}<\tfrac{\varepsilon_{0}}{8}$ and $-k<\widetilde{x}_{0}$.
Then
\[
\frac{\varepsilon_{0}}{8}\leq\frac{\varepsilon_{0}}{4}-r\quad\text{and}\quad
\frac{1}{k}<\min\bigl(\frac{1}{n},\frac{\varepsilon_{0}}{8}\bigr)
\]
is valid and thus we have for all $z\in M_{1}$
\begin{align}\label{eq:disc_in_S_k}
 \overline{\D_{r}(z)}
&\subset\{w\in\C\;|\;\d(w,M_{1})\leq r\}\nonumber\\
&\subset([\widetilde{x}_{0}+2-r,\infty)\times[-\tfrac{\varepsilon_{0}}{2}-r,\tfrac{\varepsilon_{0}}{2}+r])
         \setminus\{w\in\C\;|\;|\im(w)|<\tfrac{\varepsilon_{0}}{4}-r\}\nonumber\\
&\subset([\widetilde{x}_{0}+1,\infty)\times[-\tfrac{5\varepsilon_{0}}{8},\tfrac{5\varepsilon_{0}}{8}])
         \setminus\{w\in\C\;|\;|\im(w)|<\tfrac{\varepsilon_{0}}{8}\}\subset S_{k}(U).
\end{align}
From \prettyref{prop:Cauchy_estimates} follows that
\begin{align*}
\sup_{\substack{z\in M_{1}\\ \beta\in\N^{2}_{0},|\beta|\leq m}}
 p_{\alpha}(\partial_{\C}^{|\beta|}f(z))\e^{-\frac{1}{n}|\re(z)|}
&\leq \e^{\frac{r}{n}}\frac{m!}{r^{m}}\sup_{\zeta\in S_{k}(U)}p_{\alpha}(f(\zeta))\e^{-\frac{1}{k}|\re(\zeta)|}\\
&=\e^{\frac{r}{n}}\frac{m!}{r^{m}}\vertiii{f}_{U^{\ast},k,\alpha}.
\end{align*}
Since $\overline{M}_{0}\subset[(U\setminus\overline{\R})]\cap\R^{2}$ by \eqref{eq:M_0_inclusion}, $\overline{M}_{0}$ is compact 
and $f\in\mathcal{O}^{exp}(U\setminus\overline{\R},E)$, we have
\[
C(f)\leq\sup_{\substack{z\in \overline{M}_{0}\\ \beta\in\N^{2}_{0},|\beta|\leq m}}
p_{\alpha}(\partial_{\C}^{|\beta|}f(z))\e^{-\frac{1}{n}|\re(z)|}+\e^{\frac{r}{n}}\frac{m!}{r^{m}}\vertiii{f}_{U^{\ast},k,\alpha}<\infty.
\]
Due to \eqref{eq:estimate_CR_phi_0_1_f} this implies that
$|\overline{\partial}(\varphi_{1}\varphi_{0} f)|_{\partial\Omega,n,m,\alpha}<\infty$ 
for all $n\in\N$, $n\geq 2$, $m\in\N_{0}$ and $\alpha\in\mathfrak{A}$ and thus 
$\overline{\partial}(\varphi_{1}\varphi_{0} f)\in\mathcal{E}^{exp}(\overline{\C}\setminus \partial\Omega,E)$. 
As $E$ is admissible, there exists $g\in\mathcal{E}^{exp}(\overline{\C}\setminus\partial\Omega,E)$ such that
\begin{equation}\label{eq:CR_solution}
\overline{\partial}g=\overline{\partial}(\varphi_{1}\varphi_{0} f).
\end{equation}

(iii) We set $F\coloneq\varphi_{1}\varphi_{0} f-g$. The next step is to show that 
$F\in\mathcal{O}^{exp}(\overline{\C}\setminus\overline{\Omega},E)$, 
which implies $F\in\mathcal{O}^{exp}(U_{1}\setminus\overline{\R},E)$ since 
$\mathcal{O}^{exp}(\overline{\C}\setminus \overline{\Omega},E)\subset\mathcal{O}^{exp}(U_{1}\setminus\overline{\R},E)$, 
and that $f-F\in\mathcal{O}^{exp}(U,E)$.
$F$ is defined on $\C\setminus\overline{\Omega}$ (by setting $\varphi_{1}\varphi_{0}f\coloneq 0$ on 
$[(U^{C}\cup\overline{\Omega})\setminus\partial\Omega]\cap\C$) and can be regarded as an element of 
$\mathcal{O}(\C\setminus\overline{\Omega},E)$ due to \eqref{eq:CR_solution}.
Let $n\in\N$, $n\geq 2$. We set $V\coloneq V_{0}\cup W_{0}$, $S(n)\coloneq S_{n}(\overline{\Omega})\setminus V$ 
and remark that $S_{n}(\overline{\Omega})\subset S_{n}(\partial\Omega)$.
For $\alpha\in\mathfrak{A}$ we have by the choice of $\varphi_{i}$, $i=1,2$,
\begin{align}\label{eq:estimate_F_clos_Omega}
  |F|_{\overline{\Omega},n,\alpha}
&=\sup_{z\in S_{n}(\overline{\Omega})}p_{\alpha}(F(z)\big)\e^{-\frac{1}{n}|\re(z)|}\nonumber\\
&\leq \underbrace{\sup_{z\in S_{n}(\partial\Omega)}p_{\alpha}(g(z))\e^{-\frac{1}{n}|\re(z)|}}_{=|g|_{\partial\Omega,n,0,\alpha}}
 +\sup_{z\in S_{n}(\overline{\Omega})}p_{\alpha}(\varphi_{1}\varphi_{0}f(z))\e^{-\frac{1}{n}|\re(z)|}\nonumber\\
&=|g|_{\partial\Omega,n,0,\alpha}+\sup_{z\in S(n)}\underbrace{|(\varphi_{1}\varphi_{0})(z)|}_{\leq 1}
 p_{\alpha}(f(z))\e^{-\frac{1}{n}|\re(z)|}\nonumber\\
&\leq |g|_{\partial\Omega,n,0,\alpha}+\sup_{z\in S(n)}p_{\alpha}(f(z))\e^{-\frac{1}{n}|\re(z)|}.
\end{align}
First, we observe that
\[
(U^{C}\cup\overline{\R})\cap\C\subset\bigl[V_{0}\cup\bigcup_{x\in \overline{\Omega}\cap\R}\D_{\tfrac{1}{n}}(x)\bigr]\eqqcolon W.
\]
$W\subset\C$ is open and so we get by the definition of the set $S(n)$ that
\[
\overline{S(n)}\subset\overline{W^{C}}=W^{C}\subset(U\setminus\overline{\R})\cap\C.
\]
Again, we claim that there are $M_{0}\subset S(n)$ bounded, $k\in\N$, $k\geq 2$, 
and $M_{1}\subset S_{k}(U)$ such that $S(n)=M_{0}\cup M_{1}$. 
For the boundedness we just have to prove that there is $C_{1}>0$ such that $|\re(z)|\leq C_{1}$ for every $z\in M_{0}$. 
We choose $k\in\N$ such that $k>n$, $\tfrac{1}{k}<\tfrac{\varepsilon_{0}}{2}<k$ and $-k<\widetilde{x}_{0}+2$.
Then we decompose the set $S(n)$ as follows
\[
S(n)=\underbrace{[S(n)\setminus S_{k}(U)]}_{\eqqcolon M_{0}}\cup\underbrace{[S(n)\cap S_{k}(U)]}_{\eqqcolon  M_{1}}.
\]
\begin{center}
\includegraphics[scale=0.8]{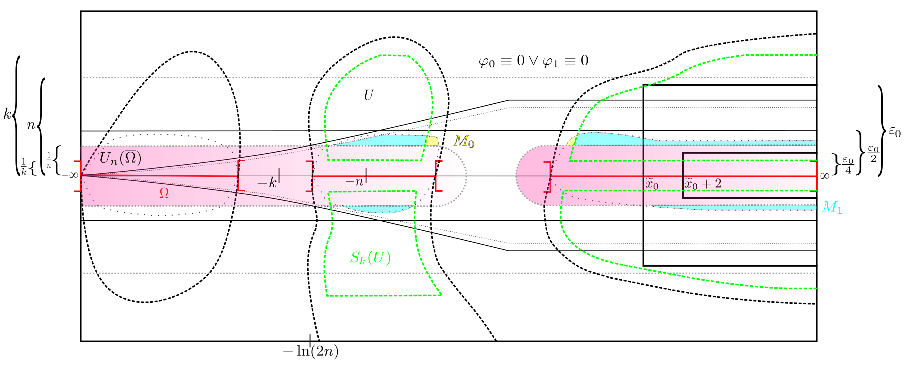}
\captionsetup{type=figure}
\caption{case: $\infty\in\Omega$, $-\infty\in\partial\Omega$}
\end{center}
Obviously $M_{1}\subset S_{k}(U)$ and $\overline{M}_{0}\subset\overline{S(n)}\subset(U\setminus\overline{\R})\cap\C$.
By the choice of $V_{0}$ we have
\begin{equation}\label{eq:inlcusion_M0_im}
M_{0}=[S(n)\setminus S_{k}(U)]\subset(S_{n}(\overline{\Omega})\setminus V_{0})
\subset\bigl\{z\in\C\;|\;|\im(z)|<\frac{\varepsilon_{0}}{2}\bigr\}
\end{equation}
and by the choice of $W_{0}$ 
\begin{equation}\label{eq:inlcusion_M0_re}
M_{0}\subset(S_{n}(\overline{\Omega})\setminus W_{0})\subset\{z\in\C\;|\;\re(z)>\min(-n,-\ln(2n))\}.
\end{equation}
Let $z\in S(n)$ with $|\im(z)|<\frac{\varepsilon_{0}}{2}$ and $\re(z)\geq\widetilde{x}_{0}+2$. Then
\[
z\in\bigl([\widetilde{x}_{0}+2,\infty)\times\bigl[-\frac{\varepsilon_{0}}{2},\frac{\varepsilon_{0}}{2}\bigr]\bigr)
\subset([\widetilde{x}_{0},\infty)\times[-\varepsilon_{0},\varepsilon_{0}])\subset U
\]
and therefore
\[
\d(z,\C\cap\partial U)\geq\min\bigl(2,\frac{\varepsilon_{0}}{2}\bigr)>\frac{1}{k}
\]
by the choice of $k$. Furthermore,
\[
k>n>|\im(z)|>\frac{1}{n}>\frac{1}{k}
\]
as $[\widetilde{x}_{0},\infty]\subset \Omega$ and due to the choice of $k$. 
In addition, $\re(z)\geq \widetilde{x}_{0}+2>-k$ and $z\in U$ by the choice of $k$ and since $z\in S(n)\subset U$. 
Hence we obtain $z\in S_{k}(U)$. So it follows from \eqref{eq:inlcusion_M0_im} that
\[
M_{0}=[S(n)\setminus S_{k}(U)]\subset\{z\in\C\;|\;\re(z)<\widetilde{x}_{0}+2\}
\]
and due to \eqref{eq:inlcusion_M0_re} we gain the claim with $C_{1}\coloneq\max(n,\ln(2n),|\widetilde{x}_{0}+2|)$.
By the same arguments as in part (ii) we get $\sup_{z\in S(n)} p_{\alpha}(f(z))e^{-\frac{1}{n}|\re(z)|}<\infty$ 
and by \eqref{eq:estimate_F_clos_Omega} that $F\in \mathcal{O}^{exp}(\overline{\C}\setminus \overline{\Omega},E)$.

(iv) $f-F$ is defined on $U\cap\C$ (by the setting in the beginning of part (iii)) 
and can be regarded as an element of $\mathcal{O}(U\cap\C,E)$ due to \eqref{eq:CR_solution}. 
Let $n\in\N$, $n\geq 2$. We set $V\coloneq V_{1}\cap W_{1}$ and  $T(n)\coloneq T_{n}(U)\setminus V$. 
With $R\coloneq\{z\in\C\;|\;\re(z)\leq -n\}$ we have
\[
       [\overline{U_{n}(\partial\Omega)}\cup\{z\in\C\;|\;|\im(z)|\geq n\}]
\subset\bigl[ R\cup\{z\in\C\;|\;|\im(z)|\geq n\}\cup\bigcup_{x\in\C\cap\partial U}\overline{\D_{\tfrac{1}{n}}(x)}\bigr]
\eqqcolon \widetilde{R}
\]
and thus
\begin{equation}\label{eq:T_nU}
       T_{n}(U)
\subset[(U\cap\C)\setminus\widetilde{R}]
\subset(\C\setminus[\overline{U_{n}(\partial\Omega)}\cup\{z\in\C\;|\;|\im(z)|\geq n\}])
=S_{n}(\partial\Omega).
\end{equation}
For $\alpha\in\mathfrak{A}$ we have by the choice of $\varphi_{i}$, $i=1,2$,
\begin{align}\label{eq:estimate_f_minus_F}
  \vertiii{f-F}_{U,n,\alpha}
&=\sup_{z\in T_{n}(U)}p_{\alpha}([(1-\varphi_{1}\varphi_{0})f+g](z))\e^{-\frac{1}{n}|\re(z)|}\nonumber\\
&\underset{\mathclap{\eqref{eq:T_nU}}}{\leq}\; 
 \underbrace{\sup_{z\in S_{n}(\partial\Omega)}p_{\alpha}(g(z))\e^{-\frac{1}{n}|\re(z)|}}_{=|g|_{\partial\Omega,n,0,\alpha}}
 +\sup_{z\in T_{n}(U)}p_{\alpha}((1-\varphi_{1}\varphi_{0})f(z))\e^{-\frac{1}{n}|\re(z)|}\nonumber\\
&=|g|_{\partial\Omega,n,0,\alpha}+\sup_{z\in T(n)}\underbrace{|1-(\varphi_{1}\varphi_{0})(z)|}_{\leq 1}
 p_{\alpha}(f(z))\e^{-\frac{1}{n}|\re(z)|}\nonumber\\
&\leq |g|_{\partial\Omega,n,0,\alpha}+\sup_{z\in T(n)}p_{\alpha}(f(z))\e^{-\frac{1}{n}|\re(z)|}.
\end{align}
We choose $k\in\N$ such that $\frac{1}{k}<\min(\tfrac{1}{n},\tfrac{\varepsilon_{0}}{4})$.
First, we observe that
\[
(U^{C}\cup\overline{\R})\cap\C \subset \bigl[ V\cup \bigcup_{x\in U^{C}\cap\C}\D_{\tfrac{1}{n}}(x)\bigr] \eqqcolon W.
\]
The set $W\subset\C$ is open and thus we get by the definition of the set $T(n)$
\[
 T(n)
=T_{n}(U)\setminus V
=\underbrace{\bigr[T_{n}(U)\setminus\bigl(\bigcup_{x\in U^{C}\cap\C}\D_{\tfrac{1}{n}}(x)\bigr)\bigr]}_{=T_{n}(U)}\setminus V
\subset W^{C}
\]
and so
\begin{equation}\label{eq:Tn_without_Sk}
\overline{T(n)\setminus S_{k}(U)}\subset \overline{T(n)}
\subset\overline{W^{C}}=W^{C}\subset(U\setminus\overline{\R})\cap\C.
\end{equation}
Then we can decompose the set $T(n)$ in the following manner
\[
T(n)=\underbrace{[T(n)\setminus S_{k}(U)]}_{\eqqcolon  M_{0}}\cup\underbrace{[T(n)\cap S_{k}(U)]}_{\eqqcolon M_{1}}.
\]
\begin{center}
\includegraphics[scale=0.8]{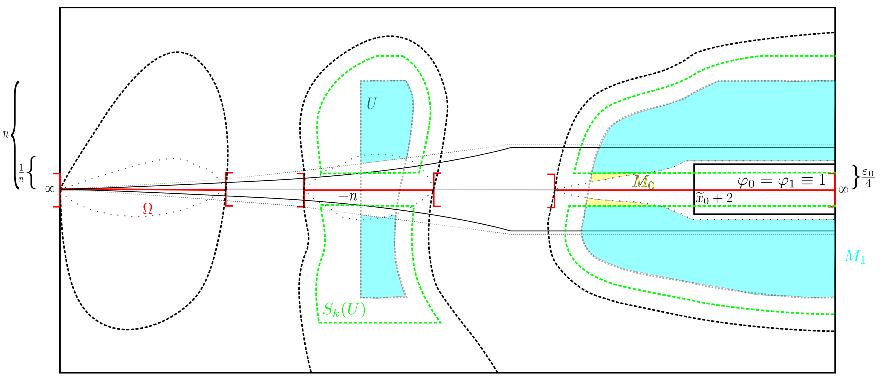}
\captionsetup{type=figure}
\caption{case: $\infty\in\Omega$, $-\infty\in\partial\Omega$}
\end{center}
We claim that the set $M_{0}$ is bounded. Again, we just have to prove that there is $C_{1}>0$ 
such that $|\re(z)|\leq C_{1}$ for every $z\in M_{0}$. 
By the choice of $k$ and the definition of $V_{1}$ and $W_{1}$ we have
$\re(z)\in[-n,\max(0,\widetilde{x}_{0}+2)]$ for every $z\in M_{0}$, proving the claim. 
Therefore, $\overline{M}_{0}$ is compact and by \eqref{eq:Tn_without_Sk} 
we get $\overline{M}_{0}\subset(U\setminus\overline{\R})\cap\C$. 
Then
\begin{align*}
 \sup_{z\in T(n)}p_{\alpha}(f(z))\e^{-\frac{1}{n}|\re(z)|}
&\leq \sup_{z\in \overline{M}_{0}}p_{\alpha}(f(z))\e^{-\frac{1}{n}|\re(z)|}
 +\sup_{z\in M_{1}}p_{\alpha}(f(z))\e^{-\frac{1}{n}|\re(z)|}\\
&\leq \sup_{z\in \overline{M}_{0}}p_{\alpha}(f(z))\e^{-\frac{1}{n}|\re(z)|}+\vertiii{f}_{U^{\ast},k,\alpha}<\infty
\end{align*}
for all $n\in\N$, $n\geq 2$, and $\alpha\in\mathfrak{A}$ since $f\in\mathcal{O}^{exp}(U\setminus\overline{\R},E)$. 
Hence we obtain by \eqref{eq:estimate_f_minus_F} that $f-F\in\mathcal{O}^{exp}(U,E)$.

So we have found 
$F\in\mathcal{O}^{exp}(\overline{\C}\setminus \overline{\Omega},E)\subset\mathcal{O}^{exp}(U_{1}\setminus\overline{\R},E)$ 
such that $[F_{\mid(U\setminus\overline{\R})\cap\C}]=[f]$, proving the surjectivity of $J$. 
For arbitrary $U$, $U_{0}\in\mathcal{U}(\Omega)$ we have, with $U_{1}$ from the proof,
\begin{align*}
      \mathcal{O}^{exp}(U\setminus\overline{\R},E)/\mathcal{O}^{exp}(U,E)
&\cong\mathcal{O}^{exp}(U_{1}\setminus\overline{\R},E)/\mathcal{O}^{exp}(U_{1},E)\\
&\cong\mathcal{O}^{exp}(U_{0}\setminus\overline{\R},E)/\mathcal{O}^{exp}(U_{0},E)
\end{align*}
algebraically, yielding the general statement.
\end{proof}

By virtue of \prettyref{lem:bv_indep_U} we may define restrictions in $bv(\Omega,E)$ in the following manner.

\begin{defn}\label{def:restrictions_bv}
Let $\Omega,\Omega_{1}\subset\overline{\R}$, $\Omega_{1}\subset\Omega$, be open and $E$ locally complete and admissible. 
For $\Omega_{1}\neq\varnothing$ let $[f]\in bv(\Omega,E)=\mathcal{O}^{exp}(U\setminus\overline{\R},E)/\mathcal{O}^{exp}(U,E)$ 
where $U\in\mathcal{U}(\Omega)$. Setting $U_{1}\coloneq U\cap(\Omega_{1}\times\R)$, we may define the restriction maps by
\[
R_{\Omega,\Omega_{1}}([f])\coloneq [f]_{\mid\Omega_{1}}\coloneq [f_{\mid(U_{1}\setminus\overline{\R})\cap\C}]
\in\mathcal{O}^{exp}(U_{1}\setminus\overline{\R},E)/\mathcal{O}^{exp}(U_{1},E)=bv(\Omega_{1},E).
\]
In addition, we define for an open set $\Omega\subset\overline{\R}$
\[
R_{\Omega,\varnothing}\colon bv(\Omega,E)\to bv(\varnothing,E),\; R_{\Omega,\varnothing}([f])\coloneq [f]_{\mid\varnothing}\coloneq 0.
\]
We denote the family $\{bv(\Omega,E)\;|\;\Omega\subset\overline{\R}\;\text{open}\}$ by $bv(E)$.
\end{defn}

\begin{thm}\label{thm:sheaf_flabby}
Let $E$ be locally complete and strictly admissible.
\begin{enumerate}
	\item [a)] $bv(E)$, equipped with the restrictions of \prettyref{def:restrictions_bv}, is a sheaf on $\overline{\R}$.
	\item [b)] $bv(E)$ is flabby, i.e.\ $R_{\overline{\R},\Omega}$ is surjective for any open $\Omega\subset\overline{\R}$.
	\item [c)] If $E$ is sequentially complete, then $bv(E)$ is isomorphic to $\mathcal{R}(E)=\mathcal{R}_{\overline{\R}}(E)$. 
	In particular, $\mathcal{R}(E)$ is a sheaf in this case.
\end{enumerate}
\end{thm}
\begin{proof}
a)(i) For open $\Omega\subset\overline{\R}$ the map $R_{\Omega,\Omega}$ can be regarded 
as $\id_{bv(\Omega,E)}$ by \prettyref{lem:bv_indep_U}. 
Let $\Omega_{3}\subset\Omega_{2}\subset\Omega_{1}\subset\overline{\R}$ be open. 
We have to prove that $R_{\Omega_{2},\Omega_{3}}\circ R_{\Omega_{1},\Omega_{2}}=R_{\Omega_{1},\Omega_{3}}$ holds. 
This is obviously true if one of the sets is empty, so let them all be non-empty. 
Let $[f]\in bv(\Omega_{1},E)=\mathcal{O}^{exp}(U_{1}\setminus\overline{\R},E)/\mathcal{O}^{exp}(U_{1},E)$ 
where $U_{1}\in\mathcal{U}(\Omega_{1})$. With $U_{2}\coloneq U_{1}\cap(\Omega_{2}\times\R)$ and
\begin{equation}\label{eq:composition_restr}	 
  U_{3}
\coloneq U_{2}\cap(\Omega_{3}\times\R)=[U_{1}\cap(\Omega_{2}\times\R)]\cap(\Omega_{3}\times\R)
\underset{\Omega_{3}\subset\Omega_{2}}{=}U_{1}\cap(\Omega_{3}\times\R)
\end{equation}
we get
\[
 R_{\Omega_{2},\Omega_{3}}\circ R_{\Omega_{1},\Omega_{2}}([f])
=R_{\Omega_{2},\Omega_{3}}([f_{\mid(U_{2}\setminus \overline{\R})\cap\C}])
\underset{U_{3}\subset U_{2}}{=}[f_{\mid(U_{3}\setminus \overline{\R})\cap\C}]
\underset{\eqref{eq:composition_restr}}{=}R_{\Omega_{1},\Omega_{3}}([f]).
\]

(ii) $(S1)$: Let $\{\Omega_{j}\subset\overline{\R}\;|\;j\in J\}$ be a family of open sets and 
$\Omega\coloneq\bigcup_{j\in J} \Omega_{j}$. Let $[f]\in bv(\Omega,E)=\mathcal{O}^{exp}(U\setminus\overline{\R},E)/\mathcal{O}^{exp}(U,E)$, 
where $U\in\mathcal{U}(\Omega)$, such that $R_{\Omega,\Omega_{j}}([f])=0$ for all $j\in J$. 
The assumption $R_{\Omega,\Omega_{j}}([f])=0$ is equivalent to $f\in\mathcal{O}^{exp}(U_{j},E)$ for every $j\in J$ 
where $U_{j}\coloneq U\cap(\Omega_{j}\times\R)$. Thus we obtain
\[
 f\in\mathcal{O}^{exp}([U\setminus\overline{\R}]\cup\bigcup_{j\in J}\Omega_{j}, E)
=\mathcal{O}^{exp}([U\setminus\overline{\R}]\cup\Omega,E)
\underset{U\in\mathcal{U}(\Omega)}{=}\mathcal{O}^{exp}(U,E)
\]
and hence $[f]=0$.

(iii) $(S2)$: Let $(\Omega_{j})_{j\in J}$ and $\Omega$ be like in part (ii). 
Let $[f_{j}]\in bv(\Omega_{j},E)=\mathcal{O}^{exp}(U_{j}\setminus\overline{\R},E)/\mathcal{O}^{exp}(U_{j},E)$, 
where $U_{j}\in\mathcal{U}(\Omega_{j})$, such that $[f_{j}]_{\mid\Omega_{j}\cap\Omega_{k}}=[f_{k}]_{\mid\Omega_{j}\cap\Omega_{k}}$. 
Hence we have, using that $bv(\Omega_{j}\cap\Omega_{k},E)$ does not depend on the choice of the open neighbourhood in $\C$ of 
$\Omega_{j}\cap\Omega_{k}$ by \prettyref{lem:bv_indep_U}, that 
\[
g_{jk}\coloneq {f_{j}}_{\mid[(U_{j}\cap U_{k})\setminus\overline{\R}]\cap\C}-{f_{k}}_{\mid[(U_{j}\cap U_{k})\setminus\overline{\R}]\cap\C}
\in\mathcal{O}^{exp}(U_{j}\cap U_{k},E)
\]
and $g_{jk}=-g_{kj}$ as well as $g_{jk}+g_{kl}+g_{lj}=0$ on $U_{j}\cap U_{k}\cap U_{l}$ by a simple calculation.

(iii.1) If $\pm\infty\notin\Omega$ and thus $\pm\infty\notin\Omega_{j}$, then exactly like in \cite[Theorem 1.4.5, p.\ 13]{H3}, 
where one uses that $E$ is \emph{strictly} admissible instead of \cite[Theorem 1.4.4, p.\ 12]{H3}, 
there are $g_{j}\in\mathcal{O}(U_{j}\cap\C,E)$ such that $g_{jk}=g_{k}-g_{j}$ on $U_{j}\cap U_{k}\cap\C$ 
(here the adjunct \emph{strictly} is needed). The setting $F_{j}\coloneq f_{j}+g_{j}$ defines a function 
$F\in\mathcal{O}((U\setminus\overline{\R})\cap\C,E)=\mathcal{O}^{exp}(U\setminus\overline{\R},E)$ 
since
\[
F_{j}-F_{k}=f_{j}+g_{j}-f_{k}-g_{k}=f_{j}-f_{k}+g_{j}-g_{k}=g_{jk}-g_{jk}=0
\]
on $U_{j}\cap U_{k}\cap\C$ such that $[F]_{\mid\Omega_{j}}=[f_{j}]$ for any $j\in J$.

(iii.2) Now, let $-\infty\in\Omega$ or $\infty\in\Omega$, i.e.\ there exists $j\in J$ 
such that $-\infty\in\Omega_{j}$ or $\infty\in\Omega_{j}$. 
We only consider the case that there are $j_{0},j_{1}\in J$ such that $-\infty\in\Omega_{j_{0}}$ and $\infty\in\Omega_{j_{1}}$. 
For the other two cases the proof is analogous. Then there are $x_{0},x_{1}\in\R$ and $\varepsilon_{0},\varepsilon_{1}>0$ 
such that $[-\infty,x_{0}]\times[-\varepsilon_{0},\varepsilon_{0}]\subset U_{j_{0}}$ and 
$[x_{1},\infty]\times[-\varepsilon_{1},\varepsilon_{1}]\subset U_{j_{1}}$. 
Now, let $x\coloneq\max(|x_{0}|,|x_{1}|)$ and $\varepsilon\coloneq\min(\varepsilon_{0},\varepsilon_{1})$. 
We define the sets 
\[
G_{0}\coloneq [(-\infty,-x-1)\times(-\tfrac{\varepsilon}{2},\tfrac{\varepsilon}{2})]^{C},\quad
H_{0}\coloneq (-\infty,-x-2]\times[-\tfrac{\varepsilon}{4},\tfrac{\varepsilon}{4}]
\] 
as well as 
\[
G_{1}\coloneq [(x+1,\infty)\times(-\tfrac{\varepsilon}{2},\tfrac{\varepsilon}{2})]^{C},\quad
H_{1}\coloneq [x+2,\infty)\times[-\tfrac{\varepsilon}{4},\tfrac{\varepsilon}{4}].
\]
By the proof of \cite[Theorem 1.4.1, p.\ 25]{H1} there are $\varphi_{i}\in\mathcal{C}^{\infty}(\R^{2})$, $i=0,1$, 
such that $0\leq\varphi_{i}\leq 1$ and $\varphi_{i}=0$ near $G_{i}$ plus $\varphi_{i}=1$ near $H_{i}$ 
as well as $|\partial^{\beta}\varphi_{i}|\leq C_{i,\beta} \widetilde{\varepsilon}^{\;-|\beta|}$ for all $\beta\in\N^{2}_{0}$ 
where $\widetilde{\varepsilon}\coloneq\frac{1}{4}\min(\tfrac{\varepsilon}{4},1)$ and $C_{i,\beta}>0$.
\begin{center}
\includegraphics[scale=0.8]{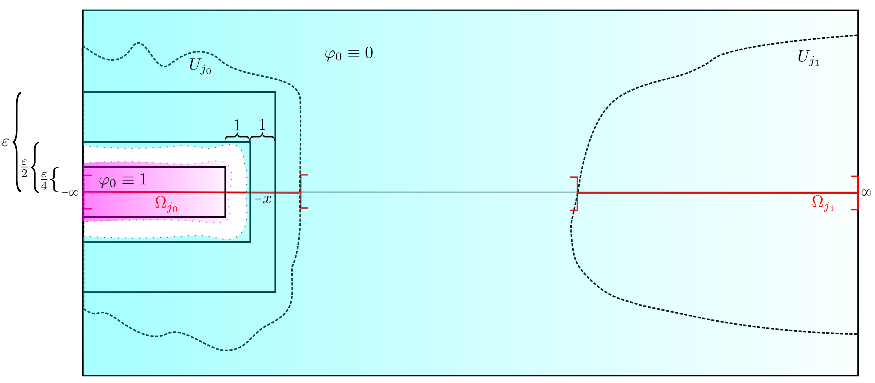}
\captionsetup{type=figure}
\caption{case: $-\infty\in\Omega_{j_{0}}$, $\infty\in\Omega_{j_{1}}$}
\end{center}
\begin{center}
\includegraphics[scale=0.8]{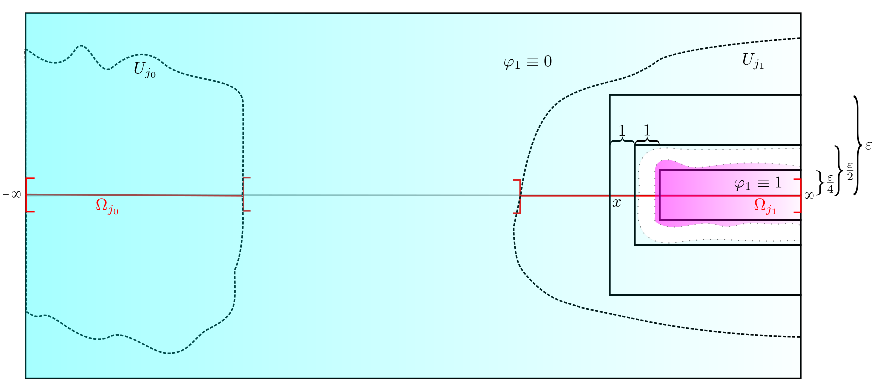}
\captionsetup{type=figure}
\caption{case: $-\infty\in\Omega_{j_{0}}$, $\infty\in\Omega_{j_{1}}$}
\end{center}
Due to case (iii.1) there is $F\in\mathcal{O}((U\setminus\overline{\R})\cap\C,E)$ such that 
$[F]_{\mid\Omega_{j}\cap\R}=[f_{j}]_{\mid\Omega_{j}\cap\R}$ for every $j\in J$. 
By \prettyref{lem:bv_indep_U} there exists $\widetilde{F}\in\mathcal{O}^{exp}(\overline{\C}\setminus\overline{\Omega},E)$ 
with $F-\widetilde{F}\in\mathcal{O}(U\cap\C,E)$. Thus we obtain
\begin{equation}\label{eq:f_j_minus_tilde_F_holom}
f_{j}-\widetilde{F}=\underbrace{(f_{j}-F)}_{\in\mathcal{O}(U_{j}\cap\C,E)}+\underbrace{(F-\widetilde{F})}_{\in\mathcal{O}(U\cap\C,E)}
\in\mathcal{O}(U_{j}\cap\C,E)
\end{equation}
for all $j\in J$. So by the choice of $\varphi_{i}$ we can regard 
$\overline{\partial}(\varphi_{0}(f_{j_{0}}-\widetilde{F})+\varphi_{1}(f_{j_{1}}-\widetilde{F}))$ 
as an element of $\mathcal{C}^{\infty}(\R^{2},E)$ (set $\varphi_{i}(f_{j_{i}}-\widetilde{F})\coloneq 0$ outside $U_{j_{i}}$). 
Let $n\in\N$, $n\geq 2$, $m\in\N_{0}$ and $\alpha\in\mathfrak{A}$. 
Then we obtain by applying the Leibniz rule and the choice of $\varphi_{i}$ like in \eqref{eq:estimate_molified_f} 
resp.\ \eqref{eq:estimate_CR_phi_0_1_f}
\begin{flalign}\label{eq:estimate_molified_S2}
&\hspace{0.35cm}
 |\overline{\partial}(\varphi_{0}(f_{j_{0}}-\widetilde{F})+\varphi_{1}(f_{j_{1}}-\widetilde{F}))|_{\varnothing,n,m,\alpha}\nonumber\\
&=\sup_{\substack{z\in S_{n}(\varnothing)\\ \beta\in\N^{2}_{0},|\beta|\leq m}}
 p_{\alpha}(\partial^{\beta}\overline{\partial}(\varphi_{0}(f_{j_{0}}-\widetilde{F})+\varphi_{1}(f_{j_{1}}-\widetilde{F}))(z))
 \e^{-\frac{1}{n}|\re(z)|}\nonumber\\
&\leq(m!)^{2}\sum_{i=0,1}\sup_{\substack{z\in S_{n}(\varnothing)\setminus(G_{i}\cup H_{i})\\ \beta\in\N^{2}_{0},|\beta|\leq m}}
 \sum_{\gamma\leq\beta}|\partial^{\beta-\gamma}(\overline{\partial}\varphi_{i})(z)|
 p_{\alpha}(\partial^{\gamma}(f_{j_{i}}-\widetilde{F})(z))
 \e^{-\frac{1}{n}|\re(z)|}\nonumber\\
&\underset{\mathclap{\eqref{eq:real.compl.part.deriv.1}}}{\leq} (m!)^{2}\hspace{-0.25cm}
 \sum_{\substack{i=0,1\\|\gamma|\leq m+1}}\underbrace{\sup_{z\in S_{n}(\varnothing)\setminus(G_{i}\cup H_{i})}
 |\partial^{\gamma}\varphi_{i}(z)|}_{\leq C_{i,\gamma}\widetilde{\varepsilon}^{\,-|\gamma|}}
 \underbrace{\sup_{\substack{z\in S_{n}(\varnothing)\setminus(G_{i}\cup H_{i})\\ \beta\in\N^{2}_{0},|\beta|\leq m}}
 p_{\alpha}(\partial_{\C}^{|\beta|}(f_{j_{i}}-\widetilde{F})(z))\e^{-\frac{1}{n}|\re(z)|}}_{\eqqcolon C(f_{j_{i}}-\widetilde{F})}\nonumber\\
&\leq(m!)^{2}(C(f_{j_{0}}-\widetilde{F})+C(f_{j_{1}}-\widetilde{F}))
 \sum_{|\gamma|\leq m+1}(C_{0,\gamma}+C_{1,\gamma})\widetilde{\varepsilon}^{\;-|\gamma|}.
\end{flalign}
Now, we have to take a closer look at $C(f_{j_{i}}-\widetilde{F})$. By the choice of the sets $G_{i}$ and $H_{i}$
\begin{align*}
 S_{n}(\varnothing)\setminus(G_{0}\cup H_{0})
&\subset\phantom{\cup}\underbrace{\{z\in\C\;|\;|\im(z)|<\tfrac{\varepsilon}{2},\;-x-2\leq\re(z)<-x-1\}}_{\eqqcolon N_{0}}\\
&\phantom{\subset}\cup\underbrace{\{z\in\C\;|\;\tfrac{\varepsilon}{4}<|\im(z)|<\tfrac{\varepsilon}{2},\;\re(z)<-x-2\}}_{\eqqcolon M_{0}}
\end{align*}
and
\begin{align*}
 S_{n}(\varnothing)\setminus(G_{1}\cup H_{1})
&\subset\phantom{\cup}\underbrace{\{z\in\C\;|\;|\im(z)|<\tfrac{\varepsilon}{2},\;x+1<\re(z)\leq x+2\}}_{\eqqcolon N_{1}}\\
&\phantom{\subset}\cup\underbrace{\{z\in\C\;|\;\tfrac{\varepsilon}{4}<|\im(z)|<\tfrac{\varepsilon}{2},\;\re(z)>x+2\}}_{\eqqcolon  M_{1}}
\end{align*}
is valid. 
\begin{center}
\includegraphics[scale=0.8]{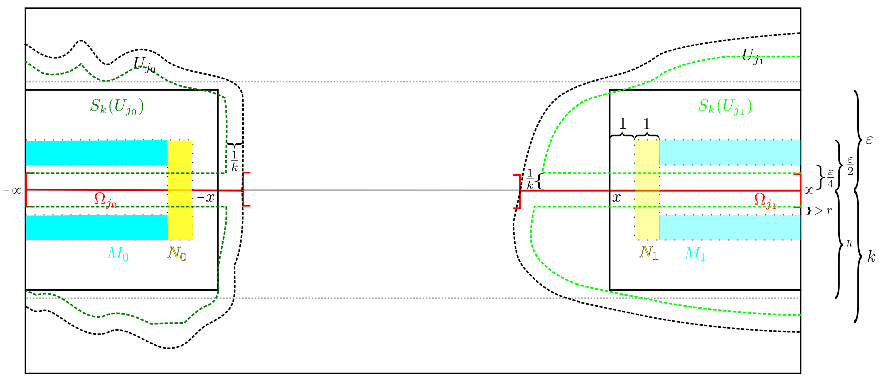}
\captionsetup{type=figure}
\caption{case: $-\infty\in\Omega_{j_{0}}$, $\infty\in\Omega_{j_{1}}$}
\end{center}
The sets $N_{i}$ are clearly bounded and $\overline{N}_{0}\subset U_{j_{0}}$ 
as well as $\overline{N}_{1}\subset U_{j_{1}}$. This implies
\begin{equation}\label{eq:estimate_C_ast_1}
 \sup_{\substack{z\in N_{i}\\ \beta\in\N^{2}_{0},|\beta|\leq m}}
 p_{\alpha}(\partial_{\C}^{|\beta|}(f_{j_{i}}-\widetilde{F})(z))\e^{-\frac{1}{n}|\re(z)|}
<\infty,\quad i=0,1,
\end{equation}
by \eqref{eq:f_j_minus_tilde_F_holom}. If we set
\[
r\coloneq\frac{1}{2}\min\bigl(2,\frac{\varepsilon}{2},\frac{\varepsilon}{4}\bigr)=\min\bigl(1,\frac{\varepsilon}{8}\bigr)
\]
and choose $k\in\N$ with $k>\max(n,\varepsilon)$ and $\tfrac{1}{k}<\tfrac{\varepsilon}{8}$
and, in addition, $-k<x$, if $\infty\notin\Omega_{j_{0}}$ resp.\ $-\infty\notin \Omega_{j_{1}}$, then
\[
\overline{\D}_{r}(z)\subset S_{k}(U_{j_{i}})\subset S_{k}(\overline{\Omega}),\quad i=0,1,
\]
holds for all $z\in M_{i}$ like in \eqref{eq:disc_in_S_k}. 
Due to \prettyref{prop:Cauchy_estimates} we have for $i=0,1$
\begin{equation}\label{eq:estimate_C_ast_2}
     \sup_{\substack{z\in M_{i}\\ \beta\in\N^{2}_{0},|\beta|\leq m}}
     p_{\alpha}(\partial_{\C}^{|\beta|}(f_{j_{i}}-\widetilde{F})(z))\e^{-\frac{1}{n}|\re(z)|}
\leq \e^{\frac{r}{n}}\frac{m!}{r^{m}}(\vertiii{f_{j_{i}}}_{U^{\ast}_{j_{i}},k,\alpha}
     +|\widetilde{F}|_{\overline{\Omega},k,\alpha})<\infty.
\end{equation}
So we get $C(f_{j_{i}}-\widetilde{F})<\infty$, $i=0,1$, by \eqref{eq:estimate_C_ast_1} and \eqref{eq:estimate_C_ast_2}, implying 
$\overline{\partial}(\varphi_{0}(f_{j_{0}}-\widetilde{F})+\varphi_{1}(f_{j_{1}}-\widetilde{F}))\in\mathcal{E}^{exp}(\overline{\C},E)$ 
by virtue of \eqref{eq:estimate_molified_S2}. 
Since $E$ is admissible, there is $g\in\mathcal{E}^{exp}(\overline{\C},E)$ such that
\begin{equation}\label{eq:CR_solution_S2}
\overline{\partial}g=\overline{\partial}(\varphi_{0}(f_{j_{0}}-\widetilde{F})+\varphi_{1}(f_{j_{1}}-\widetilde{F})).
\end{equation}

(iii.3) We set $h\coloneq\varphi_{0}(f_{j_{0}}-\widetilde{F})+\varphi_{1}(f_{j_{1}}-\widetilde{F})-g$. 
Then $h\in\mathcal{O}(\C,E)$ by \eqref{eq:CR_solution_S2} and for all $n\in\N$, $n\geq 2$, and $\alpha\in\mathfrak{A}$ we have
\begin{flalign}\label{eq:estimate_G}
&\hspace{0.35cm} |h|_{\{\pm\infty\},n,\alpha}\nonumber\\
&\leq\sum_{i=0,1}\sup_{z\in S_{n}(\{\pm\infty\})\setminus G_{i}}
 p_{\alpha}(\varphi_{i}(f_{j_{i}}-\widetilde{F})(z))\e^{-\frac{1}{n}|\re(z)|}
 +\underbrace{\sup_{z\in S_{n}(\varnothing)}p_{\alpha}(g(z))\e^{-\frac{1}{n}|\re(z)|}}_{=|g|_{\varnothing,n,\alpha}}\nonumber\\
&\leq\sum_{i=0,1}\sup_{z\in S_{n}(\{\pm\infty\})\setminus G_{i}}p_{\alpha}((f_{j_{i}}-\widetilde{F})(z))\e^{-\frac{1}{n}|\re(z)|}
 +|g|_{\varnothing,n,\alpha}.
\end{flalign}
Furthermore, if we choose $k\in\N$ such that $k>n$ and $\tfrac{1}{k}<\min(1,\tfrac{\varepsilon}{2})$ and, 
in addition, $-k<x+1$, if $\infty\notin\Omega_{j_{0}}$ resp.\ $-\infty\notin\Omega_{j_{1}}$, 
then $[S_{n}(\{\pm\infty\})\setminus G_{i}]\subset [M_{i}\cup S_{k}(U_{j_{i}})]$, $i=0,1$, where 
\[
M_{i}\coloneq
\begin{cases}
\varnothing &,\;n\leq x+1,\\
\{z\in\C\;|\;-n<\re(z)<-x-1,\,|\im(z)|\leq\tfrac{1}{n}\} &,\;n> x+1,\, i=0,\\
\{z\in\C\;|\;x+1<\re(z)< n,\,|\im(z)|\leq\tfrac{1}{n}\} &,\;n> x+1,\, i=1,
\end{cases}
\]
and its closure $\overline{M}_{i}$ is a compact subset of $U_{j_{i}}\cap\C$. 
\begin{center}
\includegraphics[scale=0.8]{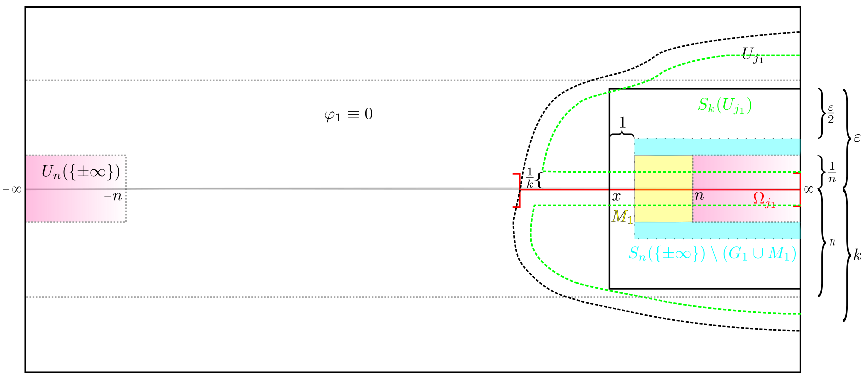}
\captionsetup{type=figure}
\caption{case: $-\infty\in\Omega_{j_{0}},$ $\infty\in\Omega_{j_{1}},$ $n>x+1,$ $i=1$}
\end{center}
In addition, $S_{k}(U_{j_{i}})\subset S_{k}(\overline{\Omega})$ and hence, keeping \eqref{eq:f_j_minus_tilde_F_holom} in mind,
\begin{flalign*}
&\hspace{0.35cm}\sup_{z\in S_{n}(\{\pm\infty\})\setminus G_{i}}p_{\alpha}((f_{j_{i}}-\widetilde{F})(z))\e^{-\frac{1}{n}|\re(z)|}\\
&\leq \sum_{i=0,1}\sup_{z\in M_{i}}p_{\alpha}((f_{j_{i}}-\widetilde{F})(z))\e^{-\frac{1}{n}|\re(z)|}
     +\underbrace{\sup_{z\in S_{k}(U_{j_{i}})}p_{\alpha}((f_{j_{i}})(z))
      \e^{-\frac{1}{n}|\re(z)|}}_{=\vertiii{f_{j_{i}}}_{U^{\ast}_{j_{i}},k,\alpha}}\nonumber\\
&\phantom{\leq}+2\underbrace{\sup_{z\in S_{k}(\overline{\Omega})}p_{\alpha}(\widetilde{F}(z))
               \e^{-\frac{1}{n}|\re(z)|}}_{=|\widetilde{F}|_{\overline{\Omega},n,\alpha}}\nonumber\\
&\leq 2|\widetilde{F}|_{\overline{\Omega},n,\alpha}+\sum_{i=0,1}\vertiii{f_{j_{i}}}_{U^{\ast}_{j_{i}},k,\alpha}
 +\sum_{i=0,1}\sup_{z\in \overline{M}_{i}}p_{\alpha}((f_{j_{i}}-\widetilde{F})(z))\e^{-\frac{1}{n}|\re(z)|}<\infty.
\end{flalign*}
So we gain $h\in\mathcal{O}^{exp}(\overline{\C}\setminus\{\pm\infty\},E)$ by \eqref{eq:estimate_G}.

(iii.4) Now, we define the function $F^{\ast}\coloneq\widetilde{F}+h$. Then we have
\[
F^{\ast}=\widetilde{F}+h\in\mathcal{O}^{exp}(\overline{\C}\setminus\overline{\Omega},E)
        \subset\mathcal{O}^{exp}(U\setminus\overline{\R},E).
\]
The last step is to prove that $F^{\ast}$ has the desired property, i.e.\ $[F^{\ast}]_{\mid\Omega_{j}}=[f_{j}]$ for all $j\in J$. 
If $j\in J$ with $\pm\infty\notin\Omega_{j}$, then
\[
f_{j}-F^{\ast}=(f_{j}-\widetilde{F})-h\in\mathcal{O}(U_{j}\cap\C,E)
\]
by \eqref{eq:f_j_minus_tilde_F_holom} and since $\mathcal{O}^{exp}(\overline{\C}\setminus\{\pm\infty\},E)\subset\mathcal{O}(\C,E)$. 
Thus we have $[F^{\ast}]_{\mid\Omega_{j}}=[f_{j}]$.

Let $j\in J$ such that $-\infty\in\Omega_{j}$ or $\infty\in\Omega_{j}$. 
Then we have for $n\in\N$, $n\geq 2$, and $\alpha\in\mathfrak{A}$
\begin{flalign}\label{eq:estimate_f_j_minus_F_ast}
&\hspace{0.35cm}\vertiii{f_{j}-F^{\ast}}_{U_{j},n,\alpha}\nonumber\\
&=\sup_{z\in T_{n}(U_{j})}p_{\alpha}((f_{j}-\widetilde{F}-\varphi_{0}(f_{j_{0}}
  -\widetilde{F})-\varphi_{1}(f_{j_{1}}-\widetilde{F})+g)(z))
  \e^{-\frac{1}{n}|\re(z)|}\nonumber\\
&\leq \sum_{i=0,1}\sup_{z\in T_{n}(U_{j})\cap H_{i}}p_{\alpha}((f_{j}-f_{j_{i}})(z))\e^{-\frac{1}{n}|\re(z)|}
 +\underbrace{\sup_{z\in S_{n}(\varnothing)}p_{\alpha}(g(z))\e^{-\frac{1}{n}|\re(z)|}}_{=|g|_{\varnothing,n,\alpha}}\nonumber\\
&\phantom{\leq}+\sup_{z\in T_{n}(U_{j})\setminus(H_{0}\cup H_{1})}
 p_{\alpha}((f_{j}-\widetilde{F}-\varphi_{0}(f_{j_{0}}-\widetilde{F})-\varphi_{1}(f_{j_{1}}-\widetilde{F}))(z))\e^{-\frac{1}{n}|\re(z)|}
\end{flalign}
where we used $T_{n}(U_{j})\subset S_{n}(\varnothing)$ plus
\begin{equation}\label{thm27.9}
H_{0}\subset G_{1} \quad\text{and}\quad H_{1}\subset G_{0}.
\end{equation}
Moreover, the following estimate holds
\begin{flalign}\label{thm27.10}
&\hspace{0.35cm}\sup_{z\in T_{n}(U_{j})\setminus(H_{0}\cup H_{1})}
 p_{\alpha}((f_{j}-\widetilde{F}-\varphi_{0}(f_{j_{0}}-\widetilde{F})
 -\varphi_{1}(f_{j_{1}}-\widetilde{F}))(z))\e^{-\frac{1}{n}|\re(z)|}\nonumber\\
&\leq \sup_{z\in T_{n}(U_{j})\setminus(H_{0}\cup H_{1})}p_{\alpha}((f_{j}-\widetilde{F})(z))\e^{-\frac{1}{n}|\re(z)|}\nonumber\\
&\phantom{\leq}+\sum_{i=0,1}\sup_{z\in T_{n}(U_{j})\setminus(H_{i}\cup G_{i})}
 p_{\alpha}((f_{j_{i}}-\widetilde{F})(z))\e^{-\frac{1}{n}|\re(z)|}
\end{flalign}
by \eqref{thm27.9} and the properties of $\varphi_{i}$. Choose $k\in\N$ such that $k>\max(n,\tfrac{\varepsilon}{2})$ 
and $\tfrac{1}{k}<\tfrac{\varepsilon}{4}$ and, in addition, $-k<x+1,$ 
if $\infty\notin \Omega_{j_{0}}$ resp.\ $-\infty\notin \Omega_{j_{1}}$. 
We remark that
\begin{align*}
T_{n}(U_{j})\setminus(H_{i}\cup G_{i})&\subset
\begin{cases}
[(-\infty,-x-1)\times(-\tfrac{\varepsilon}{2},\tfrac{\varepsilon}{2})]\setminus 
 ((-\infty,-x-2]\times[-\tfrac{\varepsilon}{4},\tfrac{\varepsilon}{4}])&,\;i=0,\\
[(x+1,\infty)\times(-\tfrac{\varepsilon}{2},\tfrac{\varepsilon}{2})]\setminus
 ([x+2,\infty)\times[-\tfrac{\varepsilon}{4},\tfrac{\varepsilon}{4}])&,\;i=1,
\end{cases}\\
&\subset S_{k}(U_{j_{i}})\cup M_{i},\quad i=0,1,
\end{align*}
with
\[
M_{i}\coloneq
\begin{cases}
\{z\in\C\;|\;-x-2<\re(z)<-x-1,\,|\im(z)|\leq\tfrac{1}{k}\} &,\;i=0,\\
\{z\in\C\;|\;x+1<\re(z)<x+2,\,|\im(z)|\leq\tfrac{1}{k}\}   &,\;i=1,
\end{cases}
\]
by the choice of $k$.
\begin{center}
\includegraphics[scale=0.8]{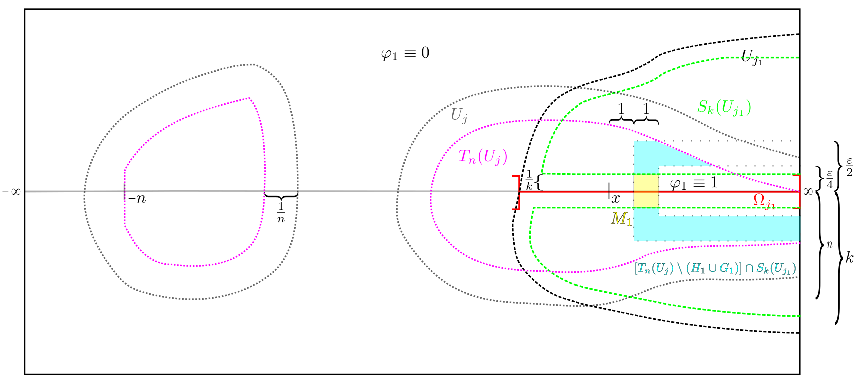}
\captionsetup{type=figure}
\caption{case $i=1$: $\infty\in\Omega_{j}$, $-\infty\notin\Omega_{j}$, $\infty\in\Omega_{j_{1}}$, $-\infty\notin\Omega_{j_{1}}$}
\end{center}
The sets $M_{i}$, $i=0,1$, are obviously bounded and $\overline{M}_{i}\subset(U_{j_{i}}\cap\C)$. 
Further, we define the set
\[
M_{2}\coloneq [T_{n}(U_{j})\setminus(H_{0}\cup H_{1})]\setminus S_{k}(U_{j})
\]
which is bounded, since $M_{2}\subset\{z\in\C\;|\;-x-2<\re(z)<x+2,\,|\im(z)|\leq\tfrac{1}{k}\}$ 
due to the choice of $k$, and we have $\overline{M}_{2}\subset\overline{T_{n}(U_{j})}\subset(U_{j}\cap\C)$.
\begin{center}
\includegraphics[scale=0.8]{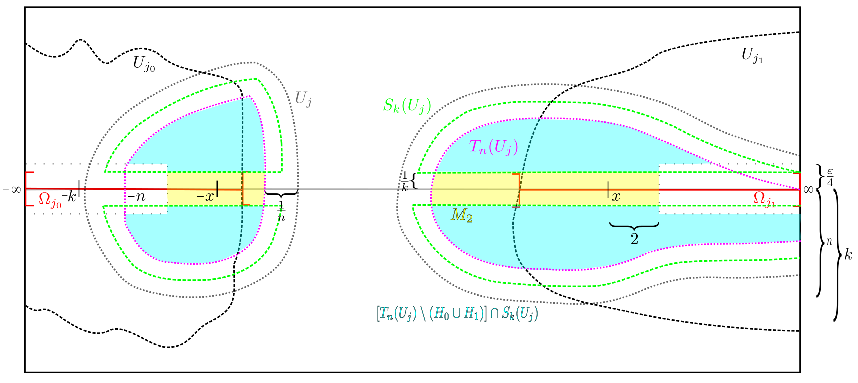}
\captionsetup{type=figure}
\caption{case: $\infty\in\Omega_{j}$, $-\infty\notin\Omega_{j}$, $-\infty\in\Omega_{j_{0}}$, $\infty\in\Omega_{j_{1}}$}
\end{center}
These results yield to
\begin{flalign*}
&\hspace{0.35cm}\sup_{z\in T_{n}(U_{j})\setminus(H_{i}\cup G_{i})}p_{\alpha}((f_{j_{i}}-\widetilde{F})(z))\e^{-\frac{1}{n}|\re(z)|}\\
&\leq \vertiii{f_{j_{i}}}_{U^{\ast}_{j_{i}},k,\alpha}+|\widetilde{F}|_{\overline{\Omega},k,\alpha}
 +\sup_{z\in \overline{M}_{i}}p_{\alpha}((f_{j_{i}}-\widetilde{F})(z))\e^{-\frac{1}{n}|\re(z)|}<\infty
\end{flalign*}
for $i=0,1$ and
\begin{flalign*}
&\hspace{0.35cm}\sup_{z\in T_{n}(U_{j})\setminus(H_{0}\cup H_{1})}p_{\alpha}((f_{j}-\widetilde{F})(z))\e^{-\frac{1}{n}|\re(z)|}\\
&\leq \vertiii{f_{j}}_{U^{\ast}_{j},k,\alpha}+|\widetilde{F}|_{\overline{\Omega},k,\alpha}
 +\sup_{z\in \overline{M}_{2}}p_{\alpha}((f_{j}-\widetilde{F})(z))\e^{-\frac{1}{n}|\re(z)|}
<\infty
\end{flalign*}
by \eqref{eq:f_j_minus_tilde_F_holom}. Thus the right-hand side of \eqref{thm27.10} is bounded from above.

Let us turn to the still pending estimates in \eqref{eq:estimate_f_j_minus_F_ast}, 
so we have to take a look at the sets $T_{n}(U_{j})\cap H_{i}$, $i=0,1$. 
\begin{center}
\includegraphics[scale=0.8]{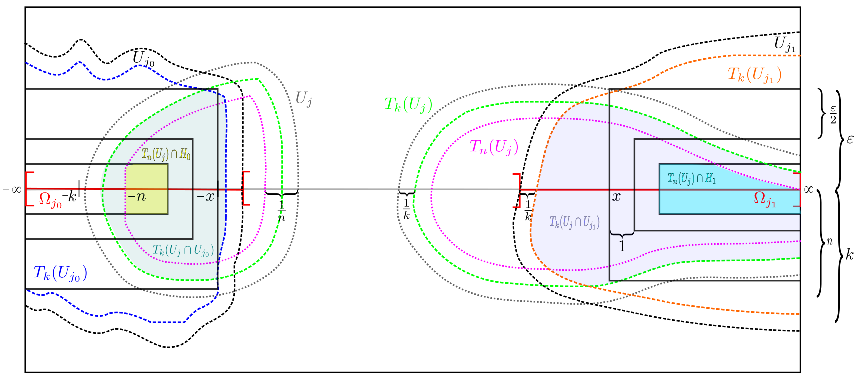}
\captionsetup{type=figure}
\caption{case: $\infty\in\Omega_{j}$, $-\infty\notin\Omega_{j}$, $\infty\notin\Omega_{j_{0}}$, $-\infty\in\Omega_{j_{0}}$, 
$\infty\in\Omega_{j_{1}}$, $-\infty\notin\Omega_{j_{1}}$}
\end{center}
We choose $k\in\N$ such that $k>n$ and $\frac{1}{k}<\min(1,\tfrac{\varepsilon}{2})$ and, 
in addition, $-k<x+1$, if $\infty\notin \Omega_{j_{0}}$ resp.\ $-\infty\notin\Omega_{j_{1}}$. 
Let $z\in H_{i}$, $i=0,1$, with $|\im(z)|<k$. Then $z\in U_{j_{i}}$ and
\[
\re(z)\leq -x-2 <k,\;\;\text{if}\;i=0,\,\infty\notin\Omega_{j_{0}},\quad\text{resp.}\quad
\re(z)\geq x+2>-k ,\;\;\text{if}\;i=1,\,-\infty\notin\Omega_{j_{1}},
\]
by the choice of $k$ as well as
\[
\d(z,\C\cap\partial U_{j_{i}})\geq\min\bigl(1,\frac{\varepsilon}{2}\bigr)>\frac{1}{k},
\]
implying $z\in T_{k}(U_{j_{i}})$. Since $k>n$, we have $T_{n}(U_{j})\subset T_{k}(U_{j})$ and 
thus $(T_{n}(U_{j})\cap H_{i})\subset[T_{k}(U_{j})\cap  T_{k}(U_{j_{i}})]$. 
Now, let $(T_{n}(U_{j})\cap H_{i})\neq\varnothing$ for some $i=0,1$ 
(in the case ``$=\varnothing$'' we have 
$\sup_{z\in T_{n}(U_{j})\cap H_{i}}\ldots=-\infty$ in \eqref{eq:estimate_f_j_minus_F_ast}). 
Let $z\in T_{k}(U_{j})\cap  T_{k}(U_{j_{i}})$, which is a non-empty set. Then $z\in U_{j}\cap U_{j_{i}}$ and $|\im(z)|<k$.
Since $\C\cap\partial(U_{j}\cap U_{j_{i}})$ is closed, there is $z_{0}\in\C\cap\partial(U_{j}\cap U_{j_{i}})$ with
\[
\d(z,\C\cap\partial(U_{j}\cap U_{j_{i}}))=|z-z_{0}|
\]
if $\C\cap\partial(U_{j}\cap U_{j_{i}})\neq\varnothing$. Moreover,
\[
[\C\cap\partial(U_{j}\cap U_{j_{i}})]\subset [(\C\cap\partial U_{j})\cup(\C\cap\partial U_{j_{i}})]
\]
and thus we obtain
\[
\d(z,\C\cap\partial(U_{j}\cap U_{j_{i}}))=|z-z_{0}|\geq
\left.\begin{cases}
\d(z,\C\cap\partial U_{j})&,\;z_{0}\in\C\cap\partial U_{j},\\
\d(z,\C\cap\partial U_{j_{i}})&,\;z_{0}\in\C\cap\partial U_{j_{i}},
\end{cases}\right\}
> \frac{1}{k}
\]
if $\C\cap\partial(U_{j}\cap U_{j_{i}})\neq\varnothing$. In the case $\C\cap\partial(U_{j}\cap U_{j_{i}})=\varnothing$ 
we note that $\d(z,\C\cap\partial(U_{j}\cap U_{j_{i}}))=\infty>\tfrac{1}{k}$.
If $\pm\infty\notin \Omega_{j}\cap \Omega_{j_{i}}$, we have in addition $-k<\re(z)<k$. 
Therefore, $T_{k}(U_{j})\cap  T_{k}(U_{j_{i}})$ is bounded and its closure is a subset of $U_{j}\cap U_{j_{i}}\cap\C$ 
if $\pm\infty\notin \Omega_{j}\cap \Omega_{j_{i}}$, and $[T_{k}(U_{j})\cap T_{k}(U_{j_{i}})]\subset T_{k}(U_{j}\cap U_{j_{i}})$ 
if $-\infty\in\Omega_{j}\cap\Omega_{j_{i}}$ or $\infty\in\Omega_{j}\cap\Omega_{j_{i}}$. 
Because $f_{j}-f_{j_{i}}\in\mathcal{O}^{exp}(U_{j}\cap U_{j_{i}},E)$, this yields to
\begin{flalign*}
&\hspace{0.35cm}\sup_{z\in T_{n}(U_{j})\cap H_{i}}p_{\alpha}((f_{j}-f_{j_{i}})(z))\e^{-\frac{1}{n}|\re(z)|}\\
&\leq
 \begin{cases}
  \sup_{z\in\overline{T_{k}(U_{j})\cap T_{k}(U_{j_{i}})}}p_{\alpha}((f_{j}-f_{j_{i}})(z))\e^{-\frac{1}{k}|\re(z)|}&,
  \;\pm\infty\notin\Omega_{j}\cap \Omega_{j_{i}},\\
  \vertiii{f_{j}-f_{j_{i}}}_{U_{j}\cap U_{j_{i}},k,\alpha} &,\;\text{else},
 \end{cases}\\
&<\infty.
\end{flalign*}
Combining our results, we conclude $\vertiii{f_{j}-F^{\ast}}_{U_{j},n,\alpha}<\infty$ 
for all $n\in\N$, $n\geq 2$, and $\alpha\in\mathfrak{A}$ 
by \eqref{eq:estimate_f_j_minus_F_ast} and thus $f_{j}-F^{\ast}\in\mathcal{O}^{exp}(U_{j},E)$, 
i.e.\ $[F^{\ast}]_{\mid\Omega_{j}}=[f_{j}]$.

b) Let $[f]\in bv(\Omega,E)=\mathcal{O}^{exp}(U\setminus\overline{\R},E)/\mathcal{O}^{exp}(U,E)$ 
where $U\in\mathcal{U}(\Omega)$ and $\Omega\subset\overline{\R}$ open. 
By \prettyref{lem:bv_indep_U} there is 
$F\in\mathcal{O}^{exp}(\overline{\C}\setminus\overline{\Omega},E)\subset\mathcal{O}^{exp}(\overline{\C}\setminus\overline{\R},E)$ 
such that $f-F\in\mathcal{O}^{exp}(U,E)$. Hence $[F]\in bv(\overline{\R},E)$ is an extension of $[f]$ to $\overline{\R}$.

c)(i) For an open set $\Omega\subset\overline{\R}$, $\Omega\neq\varnothing$, we have the following (algebraic) isomorphisms
\begin{align*}
  \mathcal{R}(\Omega,E)
&=L(\mathcal{P}_{\ast}(\overline{\Omega}),E)/L(\mathcal{P}_{\ast}(\partial\Omega),E)
 \cong\mathcal{O}^{exp}(\overline{\C}\setminus\overline{\Omega},E)/\mathcal{O}^{exp}(\overline{\C}\setminus\partial\Omega,E)\\
&\cong\mathcal{O}^{exp}((\Omega\times\R)\setminus\overline{\R},E)/\mathcal{O}^{exp}(\Omega\times\R,E)
 =bv(\Omega,E).
\end{align*}

(ii) The first isomorphism is due to \prettyref{thm:duality} and given by the map
\begin{gather*}
G_{\Omega}\colon L(\mathcal{P}_{\ast}(\overline{\Omega}),E)/L(\mathcal{P}_{\ast}(\partial\Omega),E)
\to\mathcal{O}^{exp}(\overline{\C}\setminus\overline{\Omega},E)/\mathcal{O}^{exp}(\overline{\C}\setminus\partial\Omega,E),\\
[T]\mapsto[\widetilde{T}]_{\sim},\quad\text{with}\;[\widetilde{T}]_{\overline{\Omega}}=H^{-1}_{\overline{\Omega}}(T),
\end{gather*}
where $H_{\overline{\Omega}}$ is the isomorphism from \prettyref{thm:duality} and 
we denote by $[\cdot]$ the equivalence classes in $L(\mathcal{P}_{\ast}(\overline{\Omega}),E)/L(\mathcal{P}_{\ast}(\partial\Omega),E)$,
by $[\cdot]_{\sim}$ the ones in $\mathcal{O}^{exp}(\overline{\C}\setminus\overline{\Omega},E)/
\mathcal{O}^{exp}(\overline{\C}\setminus\partial\Omega,E)$ 
and by $[\cdot]_{\overline{\Omega}}$ the ones in $\mathcal{O}^{exp}(\overline{\C}\setminus\overline{\Omega},E)/
\mathcal{O}^{exp}(\overline{\C},E)$.

\emph{well-defined}: Let $T_{0},T_{1}\in L(\mathcal{P}_{\ast}(\overline{\Omega}),E)$ such that $[T_{0}]=[T_{1}]$, i.e.\ 
$T_{0}-T_{1}\in L(\mathcal{P}_{\ast}(\partial\Omega),E)$. Then
\[
H^{-1}_{\overline{\Omega}}(T_{0}-T_{1})=H^{-1}_{\partial\Omega}(T_{0}-T_{1})
\]
by \eqref{eq:unabh.H_inv} and
\begin{align*}
  [\widetilde{T}_{0}-\widetilde{T}_{1}]_{\overline{\Omega}}
&=[\widetilde{T}_{0}]_{\overline{\Omega}}-[\widetilde{T}_{1}]_{\overline{\Omega}}
 =H^{-1}_{\overline{\Omega}}(T_{0})-H^{-1}_{\overline{\Omega}}(T_{1})
 =H^{-1}_{\overline{\Omega}}(T_{0}-T_{1})\\
&=H^{-1}_{\partial\Omega}(T_{0}-T_{1})
 \in\mathcal{O}^{exp}(\overline{\C}\setminus\partial\Omega,E)/\mathcal{O}^{exp}(\overline{\C},E)
\end{align*}
holds. Thus $\widetilde{T}_{0}-\widetilde{T}_{1}\in\mathcal{O}^{exp}(\overline{\C}\setminus\partial\Omega,E)$, 
i.e.\ $[\widetilde{T}_{0}-\widetilde{T}_{1}]_{\sim}=0$.
On the other hand, let $T\in L(\mathcal{P}_{\ast}(\overline{\Omega}),E)$ and 
$\widetilde{T}_{0},\widetilde{T}_{1}\in\mathcal{O}^{exp}(\overline{\C}\setminus\overline{\Omega},E)$ 
such that $[\widetilde{T}_{0}]_{\overline{\Omega}}=[\widetilde{T}_{1}]_{\overline{\Omega}}=H^{-1}_{\overline{\Omega}}(T)$. 
Then $\widetilde{T}_{0}-\widetilde{T}_{1}\in\mathcal{O}^{exp}(\overline{\C},E)
\subset\mathcal{O}^{exp}(\overline{\C}\setminus\partial\Omega,E)$ 
and hence $[\widetilde{T}_{0}-\widetilde{T}_{1}]_{\sim}=0$. 

\emph{injectivity}: Let $T\in L(\mathcal{P}_{\ast}(\overline{\Omega}),E)$ with $G_{\Omega}(T)=[\widetilde{T}]_{\sim}=0$. 
Then $\widetilde{T}\in\mathcal{O}^{exp}(\overline{\C}\setminus\partial\Omega,E)$ and thus
\[
 H^{-1}_{\overline{\Omega}}(T)
=[\widetilde{T}]_{\overline{\Omega}}
\in\mathcal{O}^{exp}(\overline{\C}\setminus\partial\Omega,E)/\mathcal{O}^{exp}(\overline{\C},E).
\]
Therefore, we get
\[
T=H_{\overline{\Omega}}(H^{-1}_{\overline{\Omega}}(T))=H_{\partial\Omega}(H^{-1}_{\overline{\Omega}}(T))
 \in L(\mathcal{P}_{\ast}(\partial\Omega),E)
\]
by \eqref{eq:unabh.H} and so $[T]=0$.

\emph{surjectivity}: Let $T_{0}\in\mathcal{O}^{exp}(\overline{\C}\setminus\overline{\Omega},E)$. 
Then we have $H_{\overline{\Omega}}([T_{0}]_{\overline{\Omega}})\in L(\mathcal{P}_{\ast}(\overline{\Omega}),E)$ 
by \prettyref{thm:duality}. We define $T\coloneq H_{\overline{\Omega}}([T_{0}]_{\overline{\Omega}})$ and get
\[
 H^{-1}_{\overline{\Omega}}(T)
=H^{-1}_{\overline{\Omega}}(H_{\overline{\Omega}}([T_{0}]_{\overline{\Omega}}))
=[T_{0}]_{\overline{\Omega}}
\]
by \prettyref{thm:duality} again. This means that $G_{\Omega}([T])=[T_{0}]_{\sim}$.

(iii) The second isomorphism is defined by the map
\begin{gather*}
J_{\Omega}\colon \mathcal{O}^{exp}(\overline{\C}\setminus\overline{\Omega},E)/\mathcal{O}^{exp}(\overline{\C}\setminus\partial\Omega,E)
\to\mathcal{O}^{exp}((\Omega\times\R)\setminus\overline{\R},E)/\mathcal{O}^{exp}(\Omega\times\R,E),\\
[f]_{\sim}\mapsto[f_{\mid((\Omega\times\R)\setminus\overline{\R})\cap\C}]_{\Omega},
\end{gather*}
where $[\cdot]_{\Omega}$ denotes the equivalence classes in 
$\mathcal{O}^{exp}((\Omega\times\R)\setminus\overline{\R},E)/\mathcal{O}^{exp}(\Omega\times\R,E)$.
This map is well-defined since $\mathcal{O}^{exp}(\overline{\C}\setminus\partial\Omega,E)\subset\mathcal{O}^{exp}(\Omega\times\R,E)$.

\emph{injectivity}: Let $f\in\mathcal{O}^{exp}(\overline{\C}\setminus\overline{\Omega},E)$ with $J_{\Omega}([f]_{\sim})=0$, i.e.\ 
$f\in\mathcal{O}^{exp}(\Omega\times\R,E)$. Then it follows that $f\in\mathcal{O}(\C\setminus\partial\Omega,E)$. 
Further, the estimate
\begin{equation}\label{eq:estimate_f_partial_omega}
    |f|_{\partial\Omega,n,\alpha}
\leq\underbrace{\sup_{z\in  S_{n}(\overline{\Omega})}p_{\alpha}(f(z))\e^{-\frac{1}{n}|\re(z)|}}_{=|f|_{\overline{\Omega},n,\alpha}} 
 +\sup_{z\in S_{n}(\partial\Omega)\setminus S_{n}(\overline{\Omega})}p_{\alpha}(f(z))\e^{-\frac{1}{n}|\re(z)|}
\end{equation}
holds for all $n\in\N$, $n\geq 2$, and $\alpha\in\mathfrak{A}$.
Let us examine the set $S_{n}(\partial\Omega)\setminus S_{n}(\overline{\Omega})$. 
We have for $z\in S_{n}(\partial\Omega)\setminus S_{n}(\overline{\Omega})$
\[
\re(z)\in
\begin{cases}
[\min(\R\cap\partial\Omega),\max(\R\cap\partial\Omega)]&,\;\pm\infty\notin\overline{\Omega},\\
[-n,n]&,\;\pm\infty\in\partial\Omega,\\
(-\infty,n]&,\;-\infty\in\Omega,\,\infty\in\partial\Omega,\\
[-n,\infty)&,\;-\infty\in\partial\Omega,\,\infty\in\Omega,\\
\R&,\;\pm\infty\in\Omega,\\
[-n,\max(\R\cap\partial\Omega)]&,\;-\infty\in\partial\Omega,\,\infty\notin\overline{\Omega},\\
(-\infty,\max(\R\cap\partial\Omega)]&,\;-\infty\in\Omega,\,\infty\notin\overline{\Omega},\\
[\min(\R\cap\partial\Omega),n]&,\;-\infty\notin\overline{\Omega},\,\infty\in\partial\Omega,\\
[\min(\R\cap\partial\Omega),\infty)&,\;-\infty\notin\overline{\Omega},\,\infty\in\Omega,
\end{cases}
\]
and $|\im(z)|\leq\tfrac{1}{n}$. Furthermore, we observe that 
$W\coloneq\bigcup_{x\in\R\cap\partial\Omega}\D_{\tfrac{1}{n}}(x)$ is open and
\begin{equation}\label{eq:Sn_partial_minus_Sn_overline}
 \overline{S_{n}(\partial\Omega)\setminus S_{n}(\overline{\Omega})}
=([\overline{S_{n}(\partial\Omega)\setminus S_{n}(\overline{\Omega})}]\setminus W)\subset\overline{W^{C}}
=W^{C}\subset\C\setminus\partial\Omega.
\end{equation}
So, if $\pm\infty\notin\Omega$, then $\overline{S_{n}(\partial\Omega)\setminus S_{n}(\overline{\Omega})}$ 
is a compact subset of $\C\setminus\partial\Omega$. 
Due to \eqref{eq:estimate_f_partial_omega} and since $f\in\mathcal{O}(\C\setminus\partial\Omega,E)$, 
we get $|f|_{\partial\Omega,n,\alpha}<\infty$ in this case.
Let $-\infty\in\Omega$ or $\infty\in\Omega$. Then there are $x_{i}\in\R$, $i=0,1$, 
such that $[-\infty,x_{0}]\subset\Omega$ resp.\ $[x_{1},\infty]\subset\Omega$. We choose $k\in\N$ such that $k>n$ and, 
in addition,
\[
k>x_{0},\;\;\text{if}\;-\infty\in\Omega,\;\infty\notin\Omega,\quad\text{resp.}\quad
-k<x_{1},\;\;\text{if}\;-\infty\notin\Omega,\;\infty\in\Omega.
\]
Then we obtain for $z\in[S_{n}(\partial\Omega)\setminus S_{n}(\overline{\Omega})]\setminus T_{k}(\Omega\times\R)\eqqcolon M$
\[
|\re(z)|\leq
\begin{cases}
\max(|x_{0}|,n)&,\;-\infty\in\Omega,\,\infty\in\partial\Omega,\\
\max(|x_{1}|,n)&,\;-\infty\in\partial\Omega,\,\infty\in\Omega,\\
\max(|x_{0}|,|x_{1}|)&,\;\pm\infty\in\Omega,\\
\max(|x_{0}|,|\max(\R\cap\partial\Omega)|)&,\;-\infty\in\Omega,\,\infty\notin\overline{\Omega},\\
\max(|\min(\R\cap\partial\Omega)|,|x_{1}|)&,\;-\infty\notin\overline{\Omega},\,\infty\in\Omega,
\end{cases}
\]
by the choice of $k$ and as $\partial\Omega\subset\Omega^{C}$. Hence $M$ is bounded, thus $\overline{M}$ compact, 
and $\overline{M}\subset(\C\setminus\partial\Omega)$ by \eqref{eq:Sn_partial_minus_Sn_overline}. 
Therefore, we gain 
\begin{flalign*}
&\hspace{0.35cm}\sup_{z\in S_{n}(\partial\Omega)\setminus S_{n}(\overline{\Omega})}p_{\alpha}(f(z))\e^{-\frac{1}{n}|\re(z)|}\\
&\leq\underbrace{\sup_{z\in T_{k}(\Omega\times\R)}p_{\alpha}(f(z))\e^{-\frac{1}{k}|\re(z)|}}_{=\vertiii{f}_{\Omega\times\R,k,\alpha}}
 +\sup_{z\in\overline{M}}p_{\alpha}(f(z))\e^{-\frac{1}{n}|\re(z)|}<\infty
\end{flalign*}
since $f\in\mathcal{O}^{exp}(\Omega\times\R,E)$ and $f\in\mathcal{O}(\C\setminus\partial\Omega,E)$. 
By \eqref{eq:estimate_f_partial_omega} we have $|f|_{\partial\Omega,n,\alpha}<\infty$ 
for all $n\in\N$, $n\geq 2$, and $\alpha\in\mathfrak{A}$ 
in this case as well and thus $f\in\mathcal{O}^{exp}(\overline{\C}\setminus\partial\Omega,E)$, 
proving the injectivity of $J_{\Omega}$.

\emph{surjectivity}: 
Let $[f]_{\Omega}\in\mathcal{O}^{exp}((\Omega\times\R)\setminus\overline{\R},E)/\mathcal{O}^{exp}(\Omega\times\R,E)$. 
By \prettyref{lem:bv_indep_U} there is $F\in\mathcal{O}^{exp}(\overline{\C}\setminus\overline{\Omega},E)$ 
such that $f-F\in\mathcal{O}^{exp}(\Omega\times\R,E)$, i.e.\ $J_{\Omega}([F]_{\sim})=[f]_{\Omega}$.

(iv) The last step is to prove that these isomorphisms, which we denote by $h_{\Omega}\coloneq J_{\Omega}\circ G_{\Omega}$, 
are compatible with the respective restrictions, i.e.\ that for open sets $\Omega_{1}\subset\Omega\subset\overline{\R}$ 
the diagram
\[
\begin{xy}
  \xymatrix{
      \mathcal{R}(\Omega,E) \ar[r]^{h_{\Omega}} \ar[d]_{R^{\mathcal{R}}_{\Omega,\Omega_{1}}}    &   bv(\Omega,E)
       \ar[d]^{R^{bv}_{\Omega,\Omega_{1}}} \\
      \mathcal{R}(\Omega_{1},E) \ar[r]_{h_{\Omega_{1}}}             &   bv(\Omega_{1},E)
  }
\end{xy}
\]
commutes. Let $T\in L(\mathcal{P}_{\ast}(\overline{\Omega}),E)$. We choose a representative $T_{0}$ of 
$R^{\mathcal{R}}_{\Omega,\Omega_{1}}([T])$. By the definition of the restriction
\begin{equation}\label{eq:T_0_minus_T}
T_{0}-T\in L(\mathcal{P}_{\ast}(\overline{\Omega}\setminus\Omega_{1}),E)
\end{equation}
is valid. Let $\widetilde{T}_{0}$ be a representative of $H^{-1}_{\overline{\Omega}}(T_{0})$. Then we have
\[
 (h_{\Omega_{1}}\circ R^{\mathcal{R}}_{\Omega,\Omega_{1}})([T])
=h_{\Omega_{1}}([T_{0}]_{1})
=(J_{\Omega_{1}}\circ G_{\Omega_{1}})([T_{0}]_{1})
=[\widetilde{T}_{0}\phantom{}_{\mid((\Omega_{1}\times\R)\setminus\overline{\R})\cap\C}]_{\Omega_{1}}.
\]
On the other hand, let $\widetilde{T}$ be a representative of $H^{-1}_{\overline{\Omega}}(T)$. Then we get
\[
 (R^{bv}_{\Omega,\Omega_{1}}\circ h_{\Omega})([T])
=R^{bv}_{\Omega,\Omega_{1}}([\widetilde{T}_{\mid((\Omega\times\R)\setminus\overline{\R})\cap\C}]_{\Omega})
=[\widetilde{T}_{\mid((\Omega_{1}\times\R)\setminus\overline{\R})\cap\C}]_{\Omega_{1}}.
\]
Further,
\[
 [\widetilde{T}_{0}-\widetilde{T}]_{\overline{\Omega}}
=H^{-1}_{\overline{\Omega}}(T_{0}-T)=H^{-1}_{\overline{\Omega}\setminus\Omega_{1}}(T_{0}-T)
\in\mathcal{O}^{exp}(\overline{\C}\setminus(\overline{\Omega}\setminus\Omega_{1}),E)/\mathcal{O}^{exp}(\overline{\C},E)
\]
by \eqref{eq:T_0_minus_T} and \eqref{eq:unabh.H_inv}. Therefore, 
$\widetilde{T}_{0}-\widetilde{T}\in\mathcal{O}^{exp}(\overline{\C}\setminus(\overline{\Omega}\setminus\Omega_{1}),E)
\subset\mathcal{O}^{exp}(\Omega_{1}\times\R,E)$, which implies 
$(h_{\Omega_{1}}\circ R^{\mathcal{R}}_{\Omega,\Omega_{1}})([T])=(R^{bv}_{\Omega,\Omega_{1}}\circ h_{\Omega})([T])$. 
By virtue of \prettyref{prop:sheaf_presheaf_isom} it follows that $\mathcal{R}(E)$ is a sheaf which is isomorphic to $bv(E)$.
\end{proof}

\prettyref{thm:sheaf_flabby} a), b) for $E=\C$ can be found in \cite[Corollary 3.2.3, p.\ 482]{Kawai}, 
and \prettyref{thm:sheaf_flabby} a)--c) for Fr\'echet spaces $E$ in \cite[3.8 Folgerung, p.\ 40]{J}, \cite[3.12 Satz, p.\ 44]{J} 
and \cite[Satz, p.\ 45--46]{J}. 
The counterpart of \prettyref{thm:sheaf_flabby} in the theory of vector-valued hyperfunctions is \cite[Theorem 6.9, p.\ 1125]{D/L}
and immediately we get the following corollary whose counterpart for hyperfunctions is \cite[Corollary 6.10, p. 1126]{D/L}.

\begin{cor}\label{cor:flabby}
Let $\Omega\subset\overline{\R}$ be open and $E$ sequentially complete and strictly admissible. 
Then $\{\mathcal{R}(\omega,E)\;|\;\omega\subset\Omega\;\text{open}\}$, 
equipped with the restrictions of \prettyref{def:restrictions_op_sheaf}, forms a flabby sheaf.
\end{cor}

\prettyref{cor:flabby} provides an answer to a problem stated by Ito, at least for $E$-valued Fourier hyperfunctions in one variable 
(see \cite[Problem B, p.\ 18]{Ito2002}).

Now, we want to describe the sections with support in a given compact set $K\subset\overline{\R}$. 
We recall the definition of the support of a section of a sheaf (see \cite[1.10 Definition, p.\ 7]{Bre}). 
Let $\Omega$ be a topological space, $(\mathcal{F},R^{\mathcal{F}})$ a sheaf on $\Omega$ and $f\in\mathcal{F}(\Omega)$ 
a section of a sheaf. Then the \emph{support} of $f$, denoted by $\operatorname{supp}_{\mathcal{F}}f$ 
or shortly $\operatorname{supp} f$, 
is the complement of the largest open subset of $\Omega$ on which $f=0$, i.e.\
\[
\operatorname{supp} f=\Omega\setminus \bigcup_{V\in Z_{f}}V
\]
where $Z_{f}\coloneq\{V\;|\;V\subset \Omega \,\text{open},\,f_{\mid V}=0\}$ (condition $(S1)$ is used in this definition). 
This directly yields to the following description of the support of an element of $bv(\Omega,E)$ 
for an open set $\Omega\subset\overline{\R}$ and a sequentially complete strictly admissible space $E$. Namely,
let $f=[F]\in\mathcal{O}^{exp}(U\setminus\overline{\R},E)/\mathcal{O}^{exp}(U,E)$, where $U\in\mathcal{U}(\Omega)$, 
and $\Omega_{1}\subset\Omega$ is open. If $-\infty\in\Omega$ or $\infty\in\Omega$, we define the set
\begin{align*}
S_{n}(U,\Omega_{1})&\coloneq\phantom{\cap}
\{z\in U\cap\C\;|\;\d(z,(\overline{\Omega}\cap\R)\setminus\Omega_{1})>\tfrac{1}{n},\,
                 \d(z,\C\cap\partial U)>\tfrac{1}{n},\,|\im(z)|<n\}\\
&\phantom{\coloneq}\cap
\begin{cases}
\C &,\;\pm\infty\in\Omega,\\
\{z\in\C\;|\;\re(z)>-n\} &,\;\infty\in\Omega,\,-\infty\notin\Omega,\\
\{z\in\C\;|\;\re(z)<n\} &,\;\infty\notin\Omega,\,-\infty\in\Omega,
\end{cases}\\
&\phantom{\coloneq}\setminus
\begin{cases}
[(-\infty,-n]\cup[n,\infty)]+\mathsf{i}[-\tfrac{1}{n},\tfrac{1}{n}] &,\;\pm\infty\notin\Omega_{1},\\
(-\infty,-n]+\mathsf{i}[-\tfrac{1}{n},\tfrac{1}{n}] &,\;-\infty\notin\Omega_{1},\,\infty\in\Omega_{1},\\
[n,\infty)+\mathsf{i}[-\tfrac{1}{n},\tfrac{1}{n}] &,\;\infty\notin\Omega_{1},\,-\infty\in\Omega_{1},\\
\varnothing &,\;\pm\infty\in\Omega_{1},
\end{cases}
\end{align*}
for $n\in\N$, $n\geq 2$.

If $-\infty\in\Omega$ or $\infty\in\Omega$, then $f_{\mid\Omega_{1}}=0$ is equivalent to
\begin{enumerate}
\item [a)] $F$ can be extended to a holomorphic function on $[(U\setminus\overline{\R})\cup\Omega_{1}]\cap\C$ if $\pm\infty\notin\Omega_{1}$, or
\item [b)] $F$ can be extended to a holomorphic function on $[(U\setminus\overline{\R})\cup\Omega_{1}]\cap\C$ and
 \begin{equation}\label{eq:supp}
 |F|_{U,\Omega_{1},n,\alpha}\coloneq\sup_{z\in S_{n}(U,\Omega_{1})}p_{\alpha}(F(z))\e^{-\frac{1}{n}|\re(z)|}<\infty
 \end{equation}
 for every $n\in\N$, $n\geq 2$, and $\alpha\in\mathfrak{A}$ if $-\infty\in\Omega_{1}$ or $\infty\in\Omega_{1}$.
\end{enumerate}
We remark that \eqref{eq:supp} is valid in (a) as well. 
If $\pm\infty\notin\Omega$, then $f_{\mid\Omega_{1}}=0$ is equivalent to statement a).

Observing that
\[
 \bigl[(U\setminus\overline{\R})\cup \bigcup_{V\in Z_{f}}V\bigr]\cap\C 
=[(U\setminus\overline{\R})\cup(\Omega\setminus\operatorname{supp}f)]\cap\C
=(U\setminus \operatorname{supp}f)\cap\C,
\]
since $U\in\mathcal{U}(\Omega)$, and
\[
 (\overline{\Omega}\cap\R)\setminus \bigcup_{V\in Z_{f}}V
=(\operatorname{supp}f\cup\partial\Omega)\cap\R,
\]
where the closure and the boundary are taken in $\overline{\R}$, 
we get $F\in\mathcal{O}((U\setminus\operatorname{supp}f)\cap\C,E)$ and, if $-\infty\in\Omega_{1}$ or $\infty\in\Omega_{1}$, 
in addition,
\[
|F|_{U,\bigcup_{V\in Z_{f}}V,n,\alpha}=\sup_{z\in S_{n}(U,\bigcup_{V\in Z_{f}}V)}p_{\alpha}(F(z))\e^{-\frac{1}{n}|\re(z)|}<\infty
\]
for every $n\in\N$, $n\geq 2$, and $\alpha\in\mathfrak{A}$ where we have
\[
\d(z,(\overline{\Omega}\cap\R)\setminus\bigcup_{V\in Z_{f}}V)=\d(z,(\operatorname{supp}f\cup\partial\Omega)\cap\R)
\]
in the definition of $S_{n}(U,\bigcup_{V\in Z_{f}}V)$.

Now, let $K\subset\Omega$ be compact, set
\[
bv_{K}(\Omega,E)\coloneq\{f\in bv(\Omega,E)\;|\;\operatorname{supp}f\subset K\}
\]
and for $U\in\mathcal{U}(\Omega)$
\[
\mathcal{O}^{exp}(U\setminus K,E)\coloneq\{f\in\mathcal{O}((U\setminus K)\cap\C,E)\;|\;\forall\;n\in\N,\,n\geq 2,\,\alpha\in\mathfrak{A}:\;
|F|_{U,\Omega\setminus K,n,\alpha}<\infty\}
\]
if $-\infty\in\Omega$ or $\infty\in\Omega$, resp.\
\[
\mathcal{O}^{exp}(U\setminus K,E)\coloneq\mathcal{O}((U\setminus K)\cap\C,E)
\]
if $\pm\infty\notin\Omega$.
Due to the considerations above and \prettyref{lem:bv_indep_U} we gain the following description of $bv_{K}(\Omega,E)$ 
whose special cases that $E=\C$ or more general that $E$ is a Fr\'echet space 
are given in \cite[Theorem 3.2.1, p.\ 480]{Kawai} and \cite[3.6 Satz, p.\ 37]{J}.

\begin{lem}\label{lem:supp_iso}
Let $\Omega\subset\overline{\R}$ be open, $K\subset\Omega$ compact and $E$ sequentially complete and strictly admissible. 
For any $U\in\mathcal{U}(\Omega)$ we have the (algebraic) isomorphism
\[
bv_{K}(\Omega,E)\cong\mathcal{O}^{exp}(U\setminus K,E)/\mathcal{O}^{exp}(U,E).
\]
In particular, we have
\[
      bv_{K}(\overline{\R},E)\cong\mathcal{O}^{exp}(\overline{\C}\setminus K)/\mathcal{O}^{exp}(\overline{\C},E)
\cong L(\mathcal{P}_{\ast}(K),E).
\]
\end{lem}
\begin{proof}
Using \prettyref{lem:bv_indep_U}, we represent $bv(\Omega,E)$ by $\mathcal{O}^{exp}(U\setminus\overline{\R},E)/\mathcal{O}^{exp}(U,E)$. 
Then the identity-map
\begin{gather*}
\{[F]\in\mathcal{O}^{exp}(U\setminus\overline{\R},E)/\mathcal{O}^{exp}(U,E)\;|\;\operatorname{supp}[F]\subset K\}
\to\mathcal{O}^{exp}(U\setminus K,E)/\mathcal{O}^{exp}(U,E),\\
[F]\mapsto[F],
\end{gather*}
is (well-)defined and surjective by the considerations above and obviously injective.

Now, let $\Omega\coloneq\overline{\R}$, set $\Omega_{1}\coloneq\overline{\R}\setminus K$ and choose $U\coloneq\overline{\C}$. 
We claim that the definition of the space $\mathcal{O}^{exp}(\overline{\C}\setminus K,E)$ in the sense above 
and in the sense of \prettyref{def:smooth_weighted_space} coincide (and therefore the spaces have the same symbol). 
Let $n\in\N$, $n\geq 2$. Then
\[
 \d(z,(\overline{\Omega}\cap\R)\setminus\Omega_{1})
=\d(z,K\cap\R)
\]
and
\[
\d(z,\C\cap\partial U)=\d(z,\varnothing)=\infty>\frac{1}{n}
\]
holds for $z\in\C$. Further,
\begin{equation*}
\left.\begin{aligned}
\pm\infty&\notin\overline{\R}\setminus K\\
-\infty&\notin\overline{\R}\setminus K\\
\infty&\notin\overline{\R}\setminus K\\
\pm\infty&\in\overline{\R}\setminus K\\
\end{aligned}
\right\}
\qquad \text{is equivalent to}\qquad
\left\{\begin{aligned}
\pm\infty&\in K\\
-\infty&\in K\\
\infty&\in K\\
\pm\infty&\notin K\\
\end{aligned}
\right.
\end{equation*}
and hence we obtain $S_{n}(\overline{\C},\overline{\R}\setminus K)=S_{n}(K)$. Thus the claim is proved. 
Therefore,
\[
      bv_{K}(\overline{\R},E)
\cong \mathcal{O}^{exp}(\overline{\C}\setminus K,E)/\mathcal{O}^{exp}(\overline{\C},E)
\cong L(\mathcal{P}_{\ast}(K),E)
\]
holds by \prettyref{thm:duality}, which proves the endorsement.
\end{proof}

We remark that this isomorphism induces a reasonable locally convex Hausdorff topology on $bv_{K}(\overline{\R},E)$ 
since $L(\mathcal{P}_{\ast}(K),E)$ has such a topology. 

As already mentioned, we are convinced that a reasonable theory of $E$-valued Fourier hyperfunctions (in one variable) 
should produce a flabby sheaf $\mathcal{F}$ on $\overline{\R}$ such that 
the set of sections supported by a compact subset $K\subset\overline{\R}$ coincides, 
in the sense of being isomorphic, with $L(\mathcal{P}_{\ast}(K),E)$ since the restricted sheaf $\mathcal{F}_{\mid\R}$ 
then satisfies the conditions of Doma\'nski and Langenbruch for a reasonable theory of $E$-valued hyperfunctions. 
In addition, the map $\mathfrak{F}\colon \mathcal{F}(\overline{\R})\to\mathcal{F}(\overline{\R})$, 
defined by $\mathfrak{F}\coloneq J^{-1}\circ\mathscr{F}_{\star}\circ J$, 
where $J\colon\mathcal{F}(\overline{\R})\to L(\mathcal{P}_{\ast}(\overline{\R}),E)$ is an isomorphism existing by assumption 
and $\mathscr{F}_{\star}$ the Fourier transformation of \prettyref{cor:Fourier-Trafo}, 
can be regarded as the Fourier transformation on the space of global sections and is an isomorphism.

If $E$ is sequentially complete and strictly admissible, the sheaves $bv(E)$ and $\mathcal{R}(E)$ satisfy this condition 
for a reasonable theory of $E$-valued Fourier hyperfunctions by \prettyref{thm:sheaf_flabby} 
and \prettyref{lem:supp_iso} (for $\mathcal{R}(E)$ remark that sheaf isomorphisms preserve supports, 
so the definition of a support in \prettyref{prop:traeger} b) was well-chosen). 
The next theorem confirms that the sufficient condition of $E$ being strictly admissible is also necessary 
for a reasonable theory of $E$-valued Fourier hyperfunctions in one variable 
if $E$ is an ultrabornological PLS-space (such spaces are complete) and describes further equivalent sufficient 
and necessary conditions.
We use its counterpart for vector-valued hyperfunctions \cite[Theorem 8.9, p.\ 1139]{D/L} in the proof.

\begin{thm}\label{thm:PA_necessary}
Let $E$ be a complex ultrabornological PLS-space. Then the following assertions are equivalent:
\begin{enumerate}
\item [a)] There is a flabby sheaf $\mathcal{F}$ on some open set $\varnothing\neq\Omega\subset\overline{\R}$ such that
 \begin{align*}
 \mathcal{F}_{K}(\Omega)\coloneq &\, \{T\in\mathcal{F}(\Omega)\;|\;\operatorname{supp}_{\mathcal{F}}(T)\subset K\}\\
                        \cong&\,  L(\mathcal{P}_{\ast}(K),E)\quad\text{for any compact}\;K\subset\Omega.
 \end{align*}
\item[b)] There is a flabby sheaf $\mathcal{F}$ on $\overline{\R}$ such that
 \begin{align*}
 \mathcal{F}_{K}(\overline{\R})\coloneq &\,\{T\in\mathcal{F}(\overline{\R})\;|\;\operatorname{supp}_{\mathcal{F}}(T)\subset K\}\\
                                \cong&\, L(\mathcal{P}_{\ast}(K),E)\quad\text{for any compact}\;K\subset\overline{\R}.
 \end{align*}
\item [c)] $E$ is strictly admissible.
\item [d)] $P(D)\colon\mathcal{C}^{\infty}(U,E)\to\mathcal{C}^{\infty}(U,E)$ is surjective for some (any) 
 elliptic linear partial differential operator $P(D)$ and some (any) open set $U\subset\R^{d}$ and some (any) $d\in\N$, $d\geq 2$.
\item [e)] $E$ has $(PA)$.
\end{enumerate}
\end{thm}
\begin{proof}
$e)\Leftrightarrow d)$: \cite[Corollary 4.1, p.\ 1113]{D/L} resp.\ \cite[Corollary 3.9, p.\ 1112]{D/L}\\
$e)\Rightarrow c)$: \prettyref{thm:examples_strictly_admiss} c)\\
$c)\Rightarrow b)$: \prettyref{thm:sheaf_flabby} and \prettyref{lem:supp_iso}\\
$b)\Rightarrow a)$: Obvious with $\Omega\coloneq\overline{\R}$.\\
$a)\Rightarrow e)$: Let there be a flabby sheaf $\mathcal{F}$ on some open set $\varnothing\neq\Omega\subset\overline{\R}$ 
such that
\begin{align*}
\mathcal{F}_{K}(\Omega)&=\{T\in\mathcal{F}(\Omega)\;|\;\operatorname{supp}_{\mathcal{F}}(T)\subset K\}\\
                       &\cong L(\mathcal{P}_{\ast}(K),E)\quad\text{for any compact}\;K\subset\Omega.
\end{align*}
Then the restriction $\mathcal{F}_{\mid\Omega\cap\R}$ of $\mathcal{F}$ to $\Omega\cap\R$ is a flabby sheaf as well such that
\begin{align*}
  (\mathcal{F}_{\mid\Omega\cap\R})_{K}(\Omega\cap\R)
&=\bigl\{T\in\mathcal{F}_{\mid\Omega\cap\R}(\Omega\cap\R)\;|\;\operatorname{supp}_{\mathcal{F}_{\mid\Omega\cap\R}}(T)\subset K\bigr\}\\
&\cong L(\mathscr{A}(K),E)\quad\text{for any compact}\;K\subset(\Omega\cap\R)
\end{align*}
since $\mathcal{P}_{\ast}(K)=\mathscr{A}(K)$ for every compact set $K\subset\R$. 
By virtue of \cite[Theorem 8.9, p.\ 1139]{D/L} this implies that $E$ has $(PA)$.
\end{proof} 

Due to the preceding theorem a reasonable theory of $E$-valued Fourier hyperfunctions (in one variable) 
does not exist for the ultrabornological PLS-spaces $E$ from \prettyref{ex:PLS_non_PA} a).

\begin{rem}\label{rem:itos_luecke_2}
It follows from \prettyref{thm:PA_necessary} ``$d) \Leftrightarrow e)$'' with $P(D)=\overline{\partial}$ and the fact that 
ultrabornological PLS-spaces are complete, in particular, quasi-complete that 
\cite[Theorem 3.1, p.\ 989]{Ito1982} (Dolbeaut-Grothendieck resolution of $^{E}{\widetilde{\mathcal{O}}}^{p}$, 
cf.\ \cite[2.1.3 Theorem, p.\ 76]{Ito1984}) is not correct for $p=n=1$ and ultrabornological PLS-spaces $E$ without $(PA)$, 
for instance, for the spaces $E$ from \prettyref{ex:PLS_non_PA} a).
\end{rem}

Clearly, there are still some open problems. 

\begin{prob}\fakephantomsection\label{prob:open}
\begin{enumerate} 
 \item[(i)] Is \emph{strict admissibility} a necessary condition for the existence 
 of a reasonable theory of $E$-valued Fourier hyperfunctions for general sequentially complete $\C$-lcHs $E$? 
 In particular, does such a reasonable theory exist for the spaces $E$ from \prettyref{ex:PLS_non_PA} b)? 
 \item[(ii)] Are strict admissibility and admissibility equivalent?
 \item[(iii)] Is strict admissibility of a sequentially complete $E$ equivalent to belonging to the classes of spaces 
 from \prettyref{thm:examples_strictly_admiss}?
 \item[(iv)] Do the results for $E$-valued Fourier hyperfunctions in one variable ($d=1$) carry over to several variables ($d\geq 2$)? 
\end{enumerate}
\end{prob}

One way to tackle \prettyref{prob:open} (iv) might be to adapt the approach 
from vector-valued hyperfunctions \cite{D/L} as described in 
\cite[Chapter 7, p.\ 153--155]{ich}. Maybe, another way is to use the heat method 
developed by Matsuzawa in \cite{Mat1,Mat2,Mat3}, 
namely, to represent $\C$-valued hyperfunctions as boundary values of solutions of the heat equation, 
which was transferred to $\C$-valued Fourier hyperfunctions in
\cite{ChungChungKim1994,Dhungana_2007,KawaiMatsuzawa1989,KimChungKim1993}.

\subsection*{Acknowledgements}
The present paper contains the main result of my PhD thesis \cite{ich}, 
written under the supervision of M.\ Langenbruch at the University of Oldenburg. 
I am deeply grateful to him for his support and advice. 
Further, it is worth to mention that some of the results appearing in the PhD thesis are essentially due to him. 
I am much obliged to the late P.\ Doma\'nski who helped me to understand PLS-spaces and the property $(PA)$ 
during a stay in Oldenburg. I am thankful to A.\ Defant who helped me with nuclear spaces and I.\ Shestakov 
for fruitful discussions.

\bibliography{biblio}

\begin{thebibliography}{81}
\providecommand{\natexlab}[1]{#1}
\providecommand{\url}[1]{\texttt{#1}}
\expandafter\ifx\csname urlstyle\endcsname\relax
  \providecommand{\doi}[1]{doi: #1}\else
  \providecommand{\doi}{doi: \begingroup \urlstyle{rm}\Url}\fi

\bibitem[Bonet and Doma\'nski(2008)]{Dom1}
J.~Bonet and P.~Doma\'nski.
\newblock {The splitting of the exact sequences of PLS-spaces and smooth
  dependence of solutions of linear partial differential equations}.
\newblock \emph{Adv. Math.}, 217:\penalty0 561--585, 2008.
\newblock \doi{10.1016/j.aim.2007.07.010}.

\bibitem[Bredon(1997)]{Bre}
G.E. Bredon.
\newblock \emph{Sheaf theory}.
\newblock Grad. Texts in Math. 170. Springer, New York, 2nd edition, 1997.
\newblock \doi{10.1007/978-1-4612-0647-7}.

\bibitem[Chung et~al.(1994)Chung, Chung, and Kim]{ChungChungKim1994}
J.~Chung, S.-Y. Chung, and D.~Kim.
\newblock A characterization for {F}ourier hyperfunctions.
\newblock \emph{Publ. RIMS, Kyoto Univ.}, 30\penalty0 (2):\penalty0 203--208,
  1994.
\newblock \doi{10.2977/prims/1195166129}.

\bibitem[Defant and Floret(1993)]{Defant}
A.~Defant and K.~Floret.
\newblock \emph{Tensor norms and operator ideals}.
\newblock Math. Stud. 176. North-Holland, Amsterdam, 1993.

\bibitem[Dhungana et~al.(2007)Dhungana, Chung, and Kim]{Dhungana_2007}
B.P. Dhungana, S.-Y. Chung, and D.~Kim.
\newblock Characterization of {F}ourier hyperfunctions by solutions of the
  {H}ermite heat equation.
\newblock \emph{Integral Transforms Spec. Funct.}, 18\penalty0 (7):\penalty0
  471--480, 2007.
\newblock \doi{10.1080/10652460701320729}.

\bibitem[Doma\'nski(2004)]{Dom2}
P.~Doma\'nski.
\newblock {Classical PLS-spaces: spaces of distributions, real analytic
  functions and their relatives}.
\newblock In Z.~Ciesielski, A.~Pe{\l}czy\'nski, and L.~Skrzypczak, editors,
  \emph{Orlicz Centenary Volume}, volume~64 of \emph{Banach Center
  Publications}, pages 51--70, Warsaw, 2004. Inst. Math., Polish Acad. Sci.
\newblock \doi{10.4064/bc64-0-5}.

\bibitem[Doma\'nski and Langenbruch(2008)]{D/L}
P.~Doma\'nski and M.~Langenbruch.
\newblock Vector valued hyperfunctions and boundary values of vector valued
  harmonic and holomorphic functions.
\newblock \emph{Publ. RIMS, Kyoto Univ.}, 44\penalty0 (4):\penalty0 1097--1142,
  2008.
\newblock \doi{10.2977/prims/1231263781}.

\bibitem[Doma\'nski and Langenbruch(2010)]{D/L3}
P.~Doma\'nski and M.~Langenbruch.
\newblock On the {L}aplace transform for vector valued hyperfuntions.
\newblock \emph{Funct. Approx. Comment. Math.}, 43\penalty0 (2):\penalty0
  129--159, 2010.
\newblock \doi{10.7169/facm/1291903394}.

\bibitem[Doma\'nski and Langenbruch(2012)]{D/L4}
P.~Doma\'nski and M.~Langenbruch.
\newblock On the abstract {C}auchy problem for operators in locally convex
  spaces.
\newblock \emph{RACSAM Rev. R. Acad. Cienc. Exactas F\'{\i}s. Nat. Ser. A
  Mat.}, 106\penalty0 (2):\penalty0 247--273, 2012.
\newblock \doi{10.1007/s13398-011-0052-4}.

\bibitem[Floret and Wloka(1968)]{F/W/Buch}
K.~Floret and J.~Wloka.
\newblock \emph{Einf\"uhrung in die Theorie der lokalkonvexen R\"aume}.
\newblock Lecture Notes in Math. 56. Springer, Berlin, 1968.
\newblock \doi{10.1007/BFb0098549}.

\bibitem[Graf(2010)]{graf2010}
U.~Graf.
\newblock \emph{Introduction to hyperfunctions and their integral transforms}.
\newblock Birkh\"auser, Basel, 2010.
\newblock \doi{10.1007/978-3-0346-0408-6}.

\bibitem[Grosse-Erdmann(1992)]{grosse-erdmann1992}
K.-G. Grosse-Erdmann.
\newblock \emph{The Borel-Okada theorem revisited}.
\newblock Habilitation thesis. Fernuniversit\"at Hagen, Hagen, 1992.

\bibitem[Grothendieck(1953)]{Grothendieck1953}
A.~Grothendieck.
\newblock {Sur certains espaces de fonctions holomorphes. I}.
\newblock \emph{J. Reine Angew. Math.}, 192:\penalty0 35--64, 1953.
\newblock \doi{10.1515/crll.1953.192.35}.

\bibitem[Grothendieck(1966)]{Gro}
A.~Grothendieck.
\newblock \emph{Produits tensoriels topologiques et espaces nucl\'eaires}.
\newblock Mem. Amer. Math. Soc. 16. AMS, Providence, R.I., 4th edition, 1966.
\newblock \doi{10.1090/memo/0016}.

\bibitem[Honda and Umeta(2012)]{honda2013}
N.~Honda and K.~Umeta.
\newblock On the sheaf of {L}aplace hyperfunctions with holomorphic parameters.
\newblock \emph{J. Math. Sci. Univ. Tokyo}, 19\penalty0 (4):\penalty0 559--586,
  2012.

\bibitem[H\"ormander(1990)]{H3}
L.~H\"ormander.
\newblock \emph{An introduction to complex analysis in several variables}.
\newblock North-Holland, Amsterdam, 3rd edition, 1990.

\bibitem[H\"ormander(2003)]{H1}
L.~H\"ormander.
\newblock \emph{The analysis of linear partial differential operators I}.
\newblock Classics Math. Springer, Berlin, 2nd edition, 2003.
\newblock \doi{10.1007/978-3-642-61497-2}.

\bibitem[Imai(1992)]{imai1992}
I.~Imai.
\newblock \emph{Applied hyperfunction theory}.
\newblock Mathematics and its applications. Kluwer, Dordrecht, 1992.
\newblock \doi{10.1007/978-94-011-2548-2}.

\bibitem[Ion and Kawai(1975)]{Ion/Ka}
P.D.F. Ion and T.~Kawai.
\newblock Theory of vector-valued hyperfunctions.
\newblock \emph{Publ. RIMS, Kyoto Univ.}, 11:\penalty0 1--19, 1975.
\newblock \doi{10.2977/prims/1195191684}.

\bibitem[Ito(1982)]{Ito1982}
Y.~Ito.
\newblock {On the Oka-Cartan-Kawai Theorem B for the sheaf
  $^{E}{\widetilde{\mathcal{O}}}$}.
\newblock \emph{Publ. RIMS, Kyoto Univ.}, 18:\penalty0 987--993, 1982.
\newblock \doi{10.2977/prims/1195183290}.

\bibitem[Ito(1984)]{Ito1984}
Y.~Ito.
\newblock {Theory of (vector valued) {F}ourier hyperfunctions. Their
  realization as boundary values of (vector valued) slowly increasing
  holomorphic functions, (I)}.
\newblock \emph{J. Math. Univ. Tokushima}, 18:\penalty0 57--101, 1984.

\bibitem[Ito(1985)]{Ito1985}
Y.~Ito.
\newblock {Theory of (vector valued) {F}ourier hyperfunctions. Their
  realization as boundary values of (vector valued) slowly increasing
  holomorphic functions, (II)}.
\newblock \emph{J. Math. Univ. Tokushima}, 19:\penalty0 25--61, 1985.

\bibitem[Ito(1988)]{Ito1988}
Y.~Ito.
\newblock Fourier hyperfunctions of general type.
\newblock \emph{J. Math. Kyoto Univ.}, 28\penalty0 (2):\penalty0 213--265,
  1988.
\newblock \doi{10.1215/kjm/1250520481}.

\bibitem[Ito(1989)]{Ito1989}
Y.~Ito.
\newblock {Theory of (vector valued) {F}ourier hyperfunctions. Their
  realization as boundary values of (vector valued) slowly increasing
  holomorphic functions, (III)}.
\newblock \emph{J. Math. Univ. Tokushima}, 23:\penalty0 23--38, 1989.

\bibitem[Ito(1990{\natexlab{a}})]{Ito1990_1}
Y.~Ito.
\newblock {Errata of \glqq Theory of (vector valued) {F}ourier hyperfunctions.
  Their realization as boundary values of (vector valued) slowly increasing
  holomorphic functions, (III), J. Math. Tokushima Univ., 23(1989),
  23--38\grqq{}}.
\newblock \emph{J. Math. Univ. Tokushima}, 24:\penalty0 23--24,
  1990{\natexlab{a}}.

\bibitem[Ito(1990{\natexlab{b}})]{Ito1990_2}
Y.~Ito.
\newblock {Theory of (vector valued) {F}ourier hyperfunctions. Their
  realization as boundary values of (vector valued) slowly increasing
  holomorphic functions (IV)}.
\newblock \emph{J. Math. Univ. Tokushima}, 24:\penalty0 13--21,
  1990{\natexlab{b}}.

\bibitem[Ito(1992{\natexlab{a}})]{Ito1992_1}
Y.~Ito.
\newblock {Theory of (vector valued) {F}ourier hyperfunctions. Their
  realization as boundary values of (vector valued) slowly increasing
  holomorphic functions, (V)}.
\newblock \emph{J. Math. Univ. Tokushima}, 26:\penalty0 39--50,
  1992{\natexlab{a}}.

\bibitem[Ito(1992{\natexlab{b}})]{Ito1992_2}
Y.~Ito.
\newblock Vector valued {F}ourier hyperfunctions.
\newblock \emph{J. Math. Kyoto Univ.}, 32\penalty0 (2):\penalty0 259--285,
  1992{\natexlab{b}}.
\newblock \doi{10.1215/kjm/1250519536}.

\bibitem[Ito(1998)]{Ito1998}
Y.~Ito.
\newblock Sato hyperfunctions valued in a locally convex space.
\newblock \emph{J. Math. Univ. Tokushima}, 32:\penalty0 27--41, 1998.

\bibitem[Ito(2002)]{Ito2002}
Y.~Ito.
\newblock Theory of general {F}ourier hyperfunctions.
\newblock \emph{J. Math. Univ. Tokushima}, 36:\penalty0 7--34, 2002.

\bibitem[Ito and Nagamachi(1975{\natexlab{a}})]{ItoNag1975_1}
Y.~Ito and S.~Nagamachi.
\newblock Theory of {H}-valued {F}ourier hyperfunctions.
\newblock \emph{J. Math. Univ. Tokushima}, 9:\penalty0 1--33,
  1975{\natexlab{a}}.

\bibitem[Ito and Nagamachi(1975{\natexlab{b}})]{ItoNag1975_2}
Y.~Ito and S.~Nagamachi.
\newblock Theory of {H}-valued {F}ourier hyperfunctions.
\newblock \emph{Proc. Japan Acad. Ser. A Math. Sci}, 51\penalty0 (7):\penalty0
  558--561, 1975{\natexlab{b}}.
\newblock \doi{10.3792/pja/1195518523}.

\bibitem[Jarchow(1981)]{Jarchow}
H.~Jarchow.
\newblock \emph{Locally convex spaces}.
\newblock Mathematische Leitf\"aden. Teubner, Stuttgart, 1981.
\newblock \doi{10.1007/978-3-322-90559-8}.

\bibitem[Junker(1979)]{J}
K.~Junker.
\newblock \emph{Vektorwertige Fourierhyperfunktionen und ein Satz vom
  Bochner-Schwartz-Typ}.
\newblock PhD thesis, Universit\"at D\"usseldorf, D\"usseldorf, 1979.

\bibitem[Kaballo(2014)]{Kaballo}
W.~Kaballo.
\newblock \emph{Aufbaukurs Funktionalanalysis und Operatortheorie}.
\newblock Springer, Berlin, 2014.
\newblock \doi{10.1007/978-3-642-37794-5}.

\bibitem[Kaneko(1988)]{Kan}
A.~Kaneko.
\newblock \emph{Introduction to hyperfunctions}.
\newblock Mathematics and its applications. Kluwer, Dordrecht, 1988.

\bibitem[Kawai(1970)]{Kawai}
T.~Kawai.
\newblock On the theory of {F}ourier hyperfunctions and its applications to
  partial differential equations with constant coefficients.
\newblock \emph{J. Fac. Sci. Univ. Tokyo, Sect. IA}, 17:\penalty0 467--517,
  1970.
\newblock \doi{10.15083/00039821}.

\bibitem[Kawai and Matsuzawa(1989)]{KawaiMatsuzawa1989}
T.~Kawai and T.~Matsuzawa.
\newblock On the boundary value of a solution of the heat equation.
\newblock \emph{Publ. RIMS, Kyoto Univ.}, 25\penalty0 (3):\penalty0 491--498,
  1989.
\newblock \doi{10.2977/prims/1195173353}.

\bibitem[Kim et~al.(1993)Kim, Chung, and Kim]{KimChungKim1993}
K.W. Kim, S.-Y. Chung, and D.~Kim.
\newblock Fourier hyperfunctions as the boundary values of smooth solutions of
  heat equations.
\newblock \emph{Publ. RIMS, Kyoto Univ.}, 29\penalty0 (2):\penalty0 289--300,
  1993.
\newblock \doi{10.2977/prims/1195167274}.

\bibitem[Komatsu(1973{\natexlab{a}})]{Kom2}
H.~Komatsu.
\newblock Hyperfunctions and linear partial differential equations.
\newblock In \emph{Hyperfunctions and pseudo-differential equations (Proc.,
  Katata, 1971)} \citet{proc.katata1971}, pages 180--191.
\newblock \doi{10.1007/BFb0068152}.

\bibitem[Komatsu(1973{\natexlab{b}})]{proc.katata1971}
H.~Komatsu, editor.
\newblock \emph{Hyperfunctions and pseudo-differential equations (Proc.,
  Katata, 1971)}, Lecture Notes in Math. 287, Berlin, 1973{\natexlab{b}}.
  Springer.
\newblock \doi{10.1007/BFb0068143}.

\bibitem[Komatsu(1987)]{Kom3}
H.~Komatsu.
\newblock Laplace transforms of hyperfunctions - {A} new foundation of the
  {H}eaviside calculus.
\newblock \emph{J. Fac. Sci. Univ. Tokyo, Sect. IA}, 34:\penalty0 805--820,
  1987.
\newblock \doi{10.15083/00039471}.

\bibitem[Komatsu(1988)]{Kom4}
H.~Komatsu.
\newblock Operational calculus, hyperfunctions and ultradistributions.
\newblock In M.~Kashiwara and T.~Kawai, editors, \emph{Algebraic Analysis, Vol.
  1}, chapter~28, pages 357--372. Academic Press, San Diego, 1988.
\newblock \doi{10.1016/B978-0-12-400465-8.50036-5}.

\bibitem[Komatsu(1993)]{Kom5}
H.~Komatsu.
\newblock Operational calculus and semi-groups of operators.
\newblock In H.~Komatsu, editor, \emph{Functional analysis and related topics
  (Proc., Kyoto, 1991)}, Lecture Notes in Math. 1540, pages 213--234, Berlin,
  1993. Springer.
\newblock \doi{10.1007/BFb0085469}.

\bibitem[Komatsu(1996)]{Kom6}
H.~Komatsu.
\newblock Solution of differential equations by means of {L}aplace
  hyperfunctions.
\newblock In M.~Morimoto and T.~Kawai, editors, \emph{Structure of solutions of
  diffential equations (Proc., Katata/Kyoto, 1995)}, pages 227--252, River
  Edge, N. J., 1996. World Scientific.

\bibitem[K\"{o}the(1969)]{Koethe1969}
G.~K\"{o}the.
\newblock \emph{{Topological vector spaces I}}.
\newblock Grundlehren Math. Wiss. 159. Springer, Berlin, 1969.
\newblock \doi{10.1007/978-3-642-64988-2}.

\bibitem[Kruse(2014)]{ich}
K.~Kruse.
\newblock \emph{Vector-valued Fourier hyperfunctions}.
\newblock PhD thesis, Universit\"at Oldenburg, Oldenburg, 2014.

\bibitem[Kruse(2019)]{kruse2018_2}
K.~Kruse.
\newblock The approximation property for weighted spaces of differentiable
  functions.
\newblock In M.~Kosek, editor, \emph{{Function Spaces XII (Proc., Krak\'ow,
  2018)}}, volume 119 of \emph{Banach Center Publ.}, pages 233--258, Warszawa,
  2019. Inst. Math., Polish Acad. Sci.
\newblock \doi{10.4064/bc119-14}.

\bibitem[Kruse(2020{\natexlab{a}})]{kruse2017}
K.~Kruse.
\newblock {Weighted spaces of vector-valued functions and the
  $\varepsilon$-product}.
\newblock \emph{Banach J. Math. Anal.}, 14\penalty0 (4):\penalty0 1509--1531,
  2020{\natexlab{a}}.
\newblock \doi{10.1007/s43037-020-00072-z}.

\bibitem[Kruse(2020{\natexlab{b}})]{kruse2018_4}
K.~Kruse.
\newblock On the nuclearity of weighted spaces of smooth functions.
\newblock \emph{Ann. Polon. Math.}, 124\penalty0 (2):\penalty0 173--196,
  2020{\natexlab{b}}.
\newblock \doi{10.4064/ap190728-17-11}.

\bibitem[Kruse(2020{\natexlab{c}})]{kruse2019_1}
K.~Kruse.
\newblock Parameter dependence of solutions of the {C}auchy--{R}iemann equation
  on weighted spaces of smooth functions.
\newblock \emph{RACSAM Rev. R. Acad. Cienc. Exactas F\'{\i}s. Nat. Ser. A
  Mat.}, 114\penalty0 (3):\penalty0 1--24, 2020{\natexlab{c}}.
\newblock \doi{10.1007/s13398-020-00863-x}.

\bibitem[Kruse(2020{\natexlab{d}})]{kruse2019_4}
K.~Kruse.
\newblock Vector-valued holomorphic functions in several variables.
\newblock \emph{Funct. Approx. Comment. Math.}, 63\penalty0 (2):\penalty0
  247--275, 2020{\natexlab{d}}.
\newblock \doi{10.7169/facm/1861}.

\bibitem[Kruse(2021)]{kruse2018_1}
K.~Kruse.
\newblock Series representations in spaces of vector-valued functions via
  {S}chauder decompositions.
\newblock \emph{Math. Nachr.}, 294\penalty0 (2):\penalty0 354--376, 2021.
\newblock \doi{10.1002/mana.201900172}.

\bibitem[Kruse(2022{\natexlab{a}})]{kruse2018_3}
K.~Kruse.
\newblock Extension of vector-valued functions and sequence space
  representation.
\newblock \emph{Bull. Belg. Math. Soc. Simon Stevin}, 29\penalty0 (3):\penalty0
  307--322, 2022{\natexlab{a}}.
\newblock \doi{10.36045/j.bbms.211009}.

\bibitem[Kruse(2022{\natexlab{b}})]{kruse2018_5}
K.~Kruse.
\newblock Surjectivity of the $\overline{\partial}$-operator between weighted
  spaces of smooth vector-valued functions.
\newblock \emph{Complex Var. Elliptic Equ.}, 67\penalty0 (11):\penalty0
  2676--2707, 2022{\natexlab{b}}.
\newblock \doi{10.1080/17476933.2021.1945587}.

\bibitem[Kruse(2023{\natexlab{a}})]{kruse2019_2}
K.~Kruse.
\newblock The inhomogeneous {C}auchy-{R}iemann equation for weighted smooth
  vector-valued functions on strips with holes.
\newblock \emph{Collect. Math.}, 74\penalty0 (1):\penalty0 81--112,
  2023{\natexlab{a}}.
\newblock \doi{10.1007/s13348-021-00337-2}.

\bibitem[Kruse(2023{\natexlab{b}})]{kruse2023}
K.~Kruse.
\newblock \emph{On vector-valued functions and the $\varepsilon$-product}.
\newblock Habilitation thesis. Hamburg University of Technology, Hamburg,
  2023{\natexlab{b}}.
\newblock \doi{10.15480/882.4898}.

\bibitem[Kultze(1970)]{Kultze}
R.~Kultze.
\newblock \emph{Garbentheorie}.
\newblock Mathematische Leif\"aden. Teubner, Stuttgart, 1970.
\newblock \doi{10.1007/978-3-322-80091-6}.

\bibitem[Langenbruch(1978)]{L2}
M.~Langenbruch.
\newblock {Randverteilungen von Null\"osungen hypoelliptischer
  Differentialgleichungen}.
\newblock \emph{Manuscripta Math.}, 26:\penalty0 17--35, 1978.
\newblock \doi{10.1007/BF01167965}.

\bibitem[Martineau(1960/1961)]{martineau1961}
A.~Martineau.
\newblock {Les hyperfonctions de M. Sato}.
\newblock \emph{S\'em. Bourbaki}, 6:\penalty0 127--139, 1960/1961.

\bibitem[Matsuzawa(1987)]{Mat1}
T.~Matsuzawa.
\newblock {A calculus approach to hyperfunctions I}.
\newblock \emph{Nagoya Math. J.}, 108:\penalty0 53--66, 1987.
\newblock \doi{10.1017/S0027763000002646}.

\bibitem[Matsuzawa(1989)]{Mat2}
T.~Matsuzawa.
\newblock {A calculus approach to hyperfunctions II}.
\newblock \emph{Trans. Amer. Math. Soc.}, 313\penalty0 (2):\penalty0 619--654,
  1989.
\newblock \doi{10.2307/2001421}.

\bibitem[Matsuzawa(1990)]{Mat3}
T.~Matsuzawa.
\newblock {A calculus approach to hyperfunctions III}.
\newblock \emph{Nagoya Math. J.}, 118:\penalty0 133--153, 1990.
\newblock \doi{10.1017/S0027763000003032}.

\bibitem[Meise and Vogt(1997)]{meisevogt1997}
R.~Meise and D.~Vogt.
\newblock \emph{Introduction to functional analysis}.
\newblock Oxf. Grad. Texts Math. 2. Clarendon Press, Oxford, 1997.

\bibitem[Morimoto(1978)]{Mori2}
M.~Morimoto.
\newblock {Analytic functionals with non-compact carrier}.
\newblock \emph{Tokyo J. Math.}, 1\penalty0 (1):\penalty0 77--103, 1978.
\newblock \doi{10.3836/tjm/1270216594}.

\bibitem[Morimoto(1992)]{Mori1}
M.~Morimoto.
\newblock \emph{An introduction to Sato's hyperfunctions}.
\newblock Transl. Math. Monogr. 129. AMS, Providence, R.I., 1992.

\bibitem[Mugibayashi and Nagamachi(1976{\natexlab{a}})]{Mu/Na1}
N.~Mugibayashi and S.~Nagamachi.
\newblock Hyperfunction quantum field theory.
\newblock \emph{Comm. Math. Phys.}, 46\penalty0 (2):\penalty0 119--134,
  1976{\natexlab{a}}.
\newblock \doi{10.1007/BF01608492}.

\bibitem[Mugibayashi and Nagamachi(1976{\natexlab{b}})]{Mu/Na2}
N.~Mugibayashi and S.~Nagamachi.
\newblock {Hyperfunction quantum field theory II. Euclidean Green's functions}.
\newblock \emph{Comm. Math. Phys.}, 49\penalty0 (3):\penalty0 257--275,
  1976{\natexlab{b}}.
\newblock \doi{10.1007/BF01608731}.

\bibitem[\={O}uchi(1973{\natexlab{a}})]{Ouchi1}
S.~\={O}uchi.
\newblock On abstract {C}auchy problems in the sense of hyperfunction.
\newblock In  \citet{proc.katata1971}, pages 135--152.
\newblock \doi{10.1007/BFb0068149}.

\bibitem[\={O}uchi(1973{\natexlab{b}})]{Ouchi2}
S.~\={O}uchi.
\newblock {Semi-groups of operators in locally convex spaces}.
\newblock \emph{J. Math. Soc. Japan}, 25\penalty0 (2):\penalty0 265--276,
  1973{\natexlab{b}}.
\newblock \doi{10.2969/jmsj/02520265}.

\bibitem[Saburi(1985)]{saburi1985}
Y.~Saburi.
\newblock Fundamental properties of modified {F}ourier hyperfunctions.
\newblock \emph{Tokyo J. Math.}, 8\penalty0 (1):\penalty0 231--273, 1985.
\newblock \doi{10.3836/tjm/1270151582}.

\bibitem[Sato(1958)]{sato1958}
M.~Sato.
\newblock Theory of hyperfunctions.
\newblock \emph{S\={u}gaku}, 10:\penalty0 1--27, 1958.
\newblock \doi{10.11429/sugaku1947.10.1}.
\newblock [In Japanese].

\bibitem[Sato(1959)]{Sato1}
M.~Sato.
\newblock {Theory of hyperfunctions, I}.
\newblock \emph{J. Fac. Sci. Univ. Tokyo, Sect. IA}, 8:\penalty0 139--193,
  1959.
\newblock \doi{10.15083/00039918}.

\bibitem[Sato(1960)]{Sato2}
M.~Sato.
\newblock {Theory of hyperfunctions, II}.
\newblock \emph{J. Fac. Sci. Univ. Tokyo, Sect. IA}, 8:\penalty0 387--437,
  1960.
\newblock \doi{10.15083/00039916}.

\bibitem[Schapira(1970)]{Schapira}
P.~Schapira.
\newblock \emph{Th\'eorie des hyperfonctions}.
\newblock Lecture Notes in Math. 126. Springer, Berlin, 1970.
\newblock \doi{10.1007/BFb0059484}.

\bibitem[Schwartz(1957)]{Sch1}
L.~Schwartz.
\newblock {Th\'eorie des distributions \`a valeurs vectorielles. I}.
\newblock \emph{Ann. Inst. Fourier (Grenoble)}, 7:\penalty0 1--142, 1957.
\newblock \doi{10.5802/aif.68}.

\bibitem[Sebasti\~{a}o~e Silva(1950)]{SebastiaoeSilva1950}
J.~Sebasti\~{a}o~e Silva.
\newblock As fun\c{c}\~{o}es anal\'{i}ticas e a an\'{a}lise funcional.
\newblock \emph{Port. Math.}, 9\penalty0 (1--2):\penalty0 1--130, 1950.

\bibitem[Stankovi\'c(1997)]{stankovic1997}
B.~Stankovi\'c.
\newblock Hyperfunctions.
\newblock \emph{Zb. Rad. (Beogr.)}, 7\penalty0 (15):\penalty0 71--110, 1997.

\bibitem[Tillmann(1955)]{Tillmann1955}
H.G. Tillmann.
\newblock {Dualit\"at in der Funktionentheorie auf Riemannschen Fl\"achen}.
\newblock \emph{J. Reine Angew. Math.}, 195:\penalty0 76--101, 1955.
\newblock \doi{10.1515/crll.1955.195.76}.

\bibitem[Tr\`eves(2006)]{Treves}
F.~Tr\`eves.
\newblock \emph{Topological vector spaces, distributions and kernels}.
\newblock Dover, Mineola, N.Y., 2006.

\bibitem[Vogt(1983)]{vogt1983}
D.~Vogt.
\newblock {On the solvability of $P(D)f=g$ for vector valued functions}.
\newblock In H.~Komatsu, editor, \emph{{Generalized functions and linear
  differential equations 8 (Proc., Kyoto, 1982)}}, volume 508 of \emph{RIMS
  K\^{o}ky\^{u}roku}, pages 168--181, Kyoto, 1983. RIMS.

\end{thebibliography}
\bibliographystyle{plainnat}
\end{document}